\documentclass[leqno]{article}

\usepackage{CJKutf8}
\usepackage{a4wide}
\usepackage[]{amsfonts} \usepackage[]{amsmath}
\usepackage[]{amssymb} \usepackage[]{latexsym}
\usepackage{graphicx,color} \usepackage{amsthm}
\usepackage{url}
\usepackage{graphicx}
\usepackage[numbers]{natbib}
\usepackage{bm}
\usepackage{amssymb,amsmath,color,times}
 \usepackage{mathptmx}      
\usepackage{latexsym}
  \usepackage[title]{appendix}

\begin{document}
\begin{CJK}{UTF8}{gbsn}

\theoremstyle{plain}
\newtheorem{theorem}{Theorem}[section] \newtheorem*{theorem*}{Theorem}
\newtheorem{proposition}[theorem]{Proposition} \newtheorem*{proposition*}{Proposition}
\newtheorem{lemma}[theorem]{Lemma} \newtheorem*{lemma*}{Lemma}
\newtheorem{corollary}[theorem]{Corollary} \newtheorem*{corollary*}{Corollary}

\theoremstyle{definition}
\newtheorem{definition}[theorem]{Definition} \newtheorem*{definition*}{Definition}
\newtheorem{example}[theorem]{Example} \newtheorem*{example*}{Example}
\newtheorem{remark}[theorem]{Remark} \newtheorem*{remark*}{Remark}
\newtheorem{problem}[theorem]{Problem} \newtheorem*{problem*}{Problem}
\newtheorem{hypotheses}[theorem]{Hypotheses} \newtheorem{assumption}[theorem]{Assumption}
\newtheorem{notation}[theorem]{Notation} \newtheorem*{question}{Question}

\newcommand{\ds}{\displaystyle} \newcommand{\nl}{\newline}
\newcommand{\eps}{\varepsilon}
\newcommand{\bE}{\mathbb{E}}
\newcommand{\cB}{\mathcal{B}}
\newcommand{\cF}{\mathcal{F}}
\newcommand{\cA}{\mathcal{A}}
\newcommand{\cM}{\mathcal{M}}
\newcommand{\cD}{\mathcal{D}}
\newcommand{\cH}{\mathcal{H}}
\newcommand{\cN}{\mathcal{N}}
\newcommand{\cL}{\mathcal{L}}
\newcommand{\cLN}{\mathcal{LN}}
\newcommand{\bP}{\mathbb{P}}
\newcommand{\bQ}{\mathbb{Q}}
\newcommand{\bN}{\mathbb{N}}
\newcommand{\bR}{\mathbb{R}}
\newcommand{\barsigma}{\overline{\sigma}}
\newcommand{\VIX}{\mbox{VIX}}
\newcommand{\erf}{\mbox{erf}}
\newcommand{\LMMR}{\mbox{LMMR}}
\newcommand{\cLcir}{\mathcal{L}_{\!_{C\!I\!R}}}
\newcommand{\rhor}{\raisebox{1.5pt}{$\rho$}}
\newcommand{\varphir}{\raisebox{1.5pt}{$\varphi$}}
\newcommand{\taur}{\raisebox{1pt}{$\tau$}}
\newcommand{\spx}{S\&P 500 }

\title{Robust optimal investment and risk control for an insurer with general insider information}

\author{Chao Yu \and Yuhan Cheng \and Yilun Song}

\maketitle

\begin{abstract}
 In this paper, we study the robust optimal investment and risk control problem for an insurer who owns the insider information about the financial market and the insurance market under model uncertainty. Both financial risky asset process and insurance risk process are assumed to be very general jump diffusion processes. The insider information is of the most general form rather than the initial enlargement type. We use the theory of forward integrals to give the first half characterization of the robust optimal strategy and transform the anticipating stochastic differential game problem into the nonanticipative stochastic differential game problem. Then we adopt the stochastic maximum principle to obtain the total characterization of the robust strategy. We discuss the two typical situations when the insurer is `small' and `large' by Malliavin calculus. For the `small' insurer, we obtain the closed-form solution in the continuous case and the half closed-form solution in the case with jumps. For the `large' insurer, we reduce the problem to the quadratic backward stochastic differential equation (BSDE) and obtain the closed-form solution in the continuous case without model uncertainty. We discuss some impacts of the model uncertainty, insider information and the `large' insurer on the optimal strategy.
\end{abstract}

\section{Introduction}
The optimal investment and risk control problem for insurers is a classical topic in actuarial research. An effective method for insurers to manage risk is reinsurance. Many works have been carried out concerning this topic including maximizing the expected utility from terminal wealth (see \cite{Liu09,Zheng16}), the mean-variance criterion (see \cite{Zeng16}), and minimizing the probability of ruin (see \cite{Bai08}).
 
 \cite{Zou14} considered the risk control problem in another way. In their model, the insurer manage risk by selecting the number of insurance policies instead of taking reinsurance business. They used the the jump-diffusion process to model the risk process of the insurer, which is negatively correlated with the risky asset process in the financial market. The optimal strategy is based on the criterion of expected utility maximization. \cite{Zou17} extended the the risk control model in \cite{Zou14} by introducing an extra jump-diffusion process, which is negatively correlated with the capital returns. \cite{Zhou17} also considered the above risk control problem based on the mean-variance criterion.

 Recently, there is a growing emphasis on the robust optimal strategy for the investment and risk control problem. As pointed out by \cite{Chen02}, the risk-based models that constitute the paradigm have well documented empirical failures. A robust optimal strategy is the optimal strategy under the worst-case probability assuming that the ambiguity aversion is considered. In other words, when the insurer is ambiguity averse, she might not believe the model is accurate by empirical statistics, which forces her to choose the different (robust) optimal strategy compared to the ambiguity-neutral insurer. Usually, this means giving up a certain degree of utility in case of following a seriously misleading strategy. \cite{Lin12} studied the optimal investment problem for an ambiguity-averse insurer in a jump-diffusion risk process. \cite{Yi13} studied the robust optimal investment-reinsurance problem under Heston's stochastic volatility (SV) model, and \cite{Pun15} extended it to the multi-dimensional SV model. Under the mean-variance criterion, \cite{Yi15} studied the optimal proportional reinsurance-investment strategy; \cite{Zeng16} studied the equilibrium strategy of a robust optimal reinsurance-investment strategy for an insure. \cite{Gu20} considered the reinsurance and investment game problem between two insurance companies under model uncertainty.

Besides model uncertainty, the insider information is another important topic of the investment and risk control problem, which has been taken into account by many scholars in recent years. Most of the existing works in the related literature including the above articles suppose the insurer develops strategies based on the public information flow, which is generated by the market noises. However, in the real world, many insurers or insurance companies owns the extra private information flow about the financial risky asset and the insurance polices, which is also called the insider information. They could gain certain extra profit by reconsidering their strategies based on their insider information. The optimal investment strategy problem was first studied by \cite{Pikovsky96} using the technique of enlargement of filtration. The insider information is modeled by a random variable, i.e., the initial  enlargement type, in \cite{Pikovsky96}. Although many researchers have studied the insider trading problems in finance (see \cite{Biagini05,Kohatsu06,DiNunno06,DiNunno09,Danilova10,Draouil15,Draouil16,Baltas19}), there are few articles concerning the optimal investment and risk control problem for an insurer with insider information. \cite{Baltas12} studied the optimal investment-reinsurance problem under insider information. Both financial asset process and the insurance risk process in \cite{Baltas12} are modeled by the Brownian motion. Then main method is based on the technique of enlargement of filtration proposed in \cite{Pikovsky96}. \cite{Xiong14} generalized the model in \cite{Baltas12} by introducing jumps in the risk process. \cite{Yan17} studied the investment-reinsurance game between two insurance companies under some extra insider information. \cite{Peng16} used another method based on the theory of forward integrals proposed in \cite{Biagini05} to study the optimal investment and risk control problem under insider information, which extended the insider information to the most general form (not necessarily the initial  enlargement type). They also introduced jumps in both financial asset process and insurance risk process.
 
 When combing model uncertainty with insider information, the optimal investment and risk control problem turns to an anticipating stochastic differential game problem, which is very difficult to solve directly. The reason is that one can not apply the forward integral method (\cite{Biagini05}) directly since there is no useful result with variation to the other controlling process. \cite{An13} developed a general stochastic maximum principle for the anticipating stochastic differential game problem by using Malliavin calculus. However, the result could only be applied to controlled It\^{o}-L\'{e}vy processes due to the complexity of the Malliavin derivative. 
 
 To the best our knowledge, only \cite{Peng21} combined the model uncertainty with insider information in the optimal investment-reinsurance problem. Inspired by \cite{Draouil15}, they used the Donsker $\delta$ functional technique to transform the problem into a nonanticipative (i.e., adapted) stochastic differential game problem, and then use the stochastic maximum principle to solve the problem. However, they only considered the special case when the insider information is of the initial enlargement type. Moreover, no closed-form solution or general-form solution is obtained when both model uncertainty and insider information are considered. The reason is that the subjective probability measure ${\cal Q}^v$ constructed in their paper is not exactly a probability measure when $v(y)=v(Y)$, which leads to the result that the final equation is a nested linear BSDE.
 
In this paper, we study the optimal investment and risk control problem under both model uncertainty and insider information (based on the criterion of expected utility maximization). Unlike \cite{Peng21}, the insider information in our model is of the most general form instead of the initial enlargement type. Thus the Donsker $\delta$ functional technique in \cite{Peng21} is invalid. Inspired by \cite{Biagini05,DiNunno06}, we use the theory of forward integrals to get the half equations about the robust optimal strategy, and then transform the problem into a nonaticipative stochastic differential game problem to obtain the other half equations by the stochastic maximum principle. We discuss the two typical cases where the insider is `small' (i.e., her strategy has no influence on the market) and `large' (i.e., her strategy has certain influence on the market), and obtain the corresponding solutions. The main contributions of this paper are as follows: 

 \begin{itemize}
\item A new optimal investment and risk control problem is established dealing with the model uncertainty, the general insider information and jumps. The definition of the set of prior probability measures is first introduced by the semimartingale decomposition theorem in the It\^{o} theory; 
 \end{itemize}
 
 \begin{itemize}
\item The insider information is of the most general form, which extends the models for insider information considered in \cite{Peng21}. Thus a new approach by combining the theory of forward integrals with stochastic maximum principles is adopted to solve the anticipating stochastic differential game problem, which is different from the Donsker $\delta$ technique used in \cite{Peng21}. In the theory of forward integrals, to obtain the similar semimartingale property of noise processes under the insider information filtration in \cite{DiNunno06}, we first introduce the predictable covariation process and the predictable version of the Girsanov theorem, which is different from the original method in \cite{DiNunno06}. With the help of the above method, we derive the equation of the robust optimal strategy (see equations (\ref{halfequ_1}), (\ref{halfequ_2}), (\ref{half_equ3_0}) and (\ref{half_equ3}));
  \end{itemize}
  
   \begin{itemize}
\item  For the `small' insurer, we reduce the problem to the non-nested linear BSDE instead of the nested linear BSDE in \cite{Peng21} by introducing the fixed point function. Thus we obtain the formula for the robust optimal strategy. When the utility function is of the logarithmic form and the insurer owns the future information about the financial market or the insurance market, we obtain the closed-form solution in the continuous case and half closed-form solution in the case with jumps by using the Donsker $\delta$ functional technique. This extends the results in \cite{Peng21}, where the closed-form solution was obtained only when there is no model uncertainty;
  \end{itemize}
  
  \begin{itemize}
\item  The case of `large' insurer is also considered in this paper, which can also be viewed as an extension of \cite{Peng21}. When the utility function is of the logarithmic form and the insurer owns the future information about the financial market or the insurance market, we reduce the problem to the quadratic BSDE in the continuous case by using the Donsker $\delta$ functional technique. We also obtain the closed-form solution when there is no model uncertainty, which is also an extension of \cite{Peng16};
  \end{itemize}
  
    \begin{itemize}
\item  For both cases of the `small' insurer and the `large' insurer, we obtain the analytic expression for the value function in the continuous case. By comparing value functions in different situation, we find that both insider information and `large' insurer have a positive impact on the utility and ambiguity-aversion has a negative impact on the utility. Moreover, we obtain the critical future time for a `small' insurer, which implies what kind of insider information is needed for an ambiguity-averse insurer to gain the same utility of an ambiguity-neutral insurer with no insider information.
  \end{itemize}
  
  This paper is organized as follows. In Section~\ref{forwardin}, we introduce the basic theory of forward integrals. In Section~\ref{Sec:model}, we formulate the optimal investment and risk control problem for an insurer under both model uncertainty and insider information. We give the first characterization and the total characterization of the robust optimal strategy in Section~\ref{half_section} and Section~\ref{tt_section}, respectively. We give further characterizations (BSDEs) of the robust optimal strategy when considering the situation of small insurer and large insurer in Section~\ref{sec_example2} and Section~\ref{sec_example1}, respectively, and derive closed-form solutions or half closed-form solutions in some particular cases. Finally, we summarize our conclusions in Section~\ref{sec:conclusion}.

   \section{The forward integral}\label{forwardin}

The forward integral was introduced by \cite{Berger82} and defined by \cite{Russo93}. This type of integral has been studied before and applied to insider trading in financial mathematics (see \cite{Russo00,Biagini05,DiNunno06}). In this section, we follow \cite{DiNunno09} and give a short review of the forward integral with respect to the Brownian motion and the Poisson random measure.
 
Let $(\Omega,{\cal F},\{{\cal F}_t\}_{t\ge 0},{\mathbb P})$ be a filtered probability space. $W^1$ and $W^2$ are two independent Brownian motions. $\tilde N^1(\mathrm{d}t,\mathrm{d}z)=N^1(\mathrm{d}t,\mathrm{d}z)-G^1(\mathrm{d}z)\mathrm{d}t$ and $\tilde N^2(\mathrm{d}t,\mathrm{d}z)=N^2(\mathrm{d}t,\mathrm{d}z)-G^2(\mathrm{d}z)\mathrm{d}t$ are two independent compensated Poisson random measures with L\'{e}vy measures $G^1(\mathrm{d}z)$ and $G^2(\mathrm{d}z)$, respectively. Then $\eta_t^i:=\int_0^t\int_{\mathbb{R}_0}z\tilde N^i(\mathrm{d}s,\mathrm{d}z)$ is a pure jump L\'{e}vy process, $i=1,2$. We assume that the filtration $\{{\cal F}_t\}_{t\ge 0}$ is the $\mathbb{P}$-augmentation of the filtration generated by random variables $W^1_s,W^2_s,\tilde N^1(s,z),\tilde N^2(s,z)$, $z\in \mathbb{R}_0:=\mathbb{R}\backslash\{0\}$, $0\le s\le t$, which satisfies the usual condition. We refer to \cite{Karatzas91,Cinlar11, Sato99} for the basic relative definitions.

Fix a constant $T>0$ as the time horizon. We denote by ${\cal C}_{\text{ucp}}$ and ${\cal D}_{\text{ucp}}$ the Fr\'{e}chet spaces (i.e., a complete metrizable locally convex topological vector space, see \cite{Huang00}) of continuous processes and c\`{a}dl\`{a}g processes (not necessarily adapted to the filitration $\{{\cal F}_t\}_{0\le t\le T}$) equipped with the metrizable topology of the uniform convergence in probability (ucp) on $[0,T]$ (see \cite{Protter05, Russo00, Billingsley99}), respectively.
  
    \begin{definition}\label{forw}
Let $\varphi  $ be a measurable process such that $\int_0^T|\varphi_t|\mathrm{d}t<\infty$. The forward integral of $\varphi$ with respect to $W^i$ is defined by
\begin{equation} 
\begin{aligned}
\int_0^t \varphi_s\mathrm{d}^-W^i_s:=\lim_{\varepsilon\rightarrow 0^+}\varepsilon^{-1}\int_0^t\varphi_s\left(W^i_{(s+\varepsilon)\wedge t}-W^i_s\right)\mathrm{d}s ,
\end{aligned}
\end{equation}
if the limit exists in ${\cal C}_{\text{ucp}}$, $i=1,2$. In this case $\varphi$ is called forward integrable with respect to $W^i$, $i=1,2$. If the limit exists in probability for every $t\in [0,T]$, we say $\varphi$ is forward integrable in the weak sense.
\end{definition}
     
   The following multiplication formula for the forward integral is an immediate consequence of Definition \ref{forw}.
     
          \begin{proposition}\label{pro1}
Suppose $\varphi$ is forward integrable with respect to $W^i$, $i=1,2$, and $F$ is a random variable. Then $F\varphi$ is forward integrable and
 \begin{equation} 
\begin{aligned}
\int_0^T F\varphi_t \mathrm{d}^-W^i_t= F\int_0^T \varphi_t \mathrm{d}^-W^i_t,\quad i=1,2.
\end{aligned}
\end{equation}
   \end{proposition}

   The next result shows that the forward integral is also an extension of the It\^{o} integral with respect to a semimartingale.  
   
          \begin{proposition} \label{fgen1}
Let $\{{\cal E}_t\}_{0\le t\le T}$ (${\cal F}_t\subset {\cal E}_t$, $0\le t\le T$) be another filtration which satisfies the usual condition. Suppose that $W^i$ is an ${\cal E}_t$-semimartingale\footnote{In fact, by the L\'{e}vy theorem (see \cite[Theorem 3.3.16]{Karatzas91}), the continuous ${\cal E}_t$-martingale part of $W$ is an ${\cal E}_t$-Brownian motion.}, and $\varphi$ is an ${\cal E}_t$-progressively measurable process which is It\^{o} integrable with respect to $W^i$, $i=1,2$. Then $\varphi$ is forward integrable and
 \begin{equation} 
\begin{aligned}
\int_0^T  \varphi_t \mathrm{d}^-W^i_t=  \int_0^T \varphi_t \mathrm{d}W^i_t,\quad i=1,2.
\end{aligned}
\end{equation}
   \end{proposition}
   
   \begin{proof}
   Since $\varepsilon^{-1}\int_0^T\varphi_t\left(W^i_{(t+\varepsilon)\wedge T}-W^i_t\right)\mathrm{d}t=  \int_0^T  \varepsilon^{-1}\int_{(s-\varepsilon)\vee 0}^{s} u_t\mathrm{d}t\mathrm{d}W^i_s$ by Fubini therorem, the conclusion follows from the convergence property for the It\^{o} integral with respect to continuous local martingales (see \cite[Proposition 3.2.26]{Karatzas91}).
   
   \end{proof}
   
   Let us now give the corresponding definition of the forward integral with respect to the compensated Poisson random measure, and some related properties, which are slightly different from the original versions in \cite{DiNunno09}.
      
\begin{definition}\label{forn}
Let $\zeta =\{\zeta(t,z)\}_{(t,z)\in [0,T]\times\mathbb{R}_0}$ be a measurable random field such that\\ $\int_0^T\int_{\mathbb{R}_0}|\zeta(t,z)|1_{U_n}G^i(\mathrm{d}z)\mathrm{d}t<\infty$ for all $n\in\mathbb{N}_+$, $i=1,2$. The forward integral of $\zeta$ with respect to $\tilde N^i$ is defined by
\begin{equation} \label{forwardN}
\begin{aligned}
\int_0^t\int_{\mathbb{R}_0} \zeta(s,z)\tilde N^i(\mathrm{d}^-s,\mathrm{d}z)&:=\lim_{n\rightarrow \infty}\int_0^t\int_{\mathbb{R}_0} \zeta(s,z)1_{U_n}\tilde N^i(\mathrm{d}s,\mathrm{d}z)\\
&:=\lim_{n\rightarrow \infty}\left(\int_0^t\int_{\mathbb{R}_0} \zeta(s,z)1_{U_n} N^i(\mathrm{d}s,\mathrm{d}z)- \int_0^t\int_{\mathbb{R}_0} \zeta(s,z)1_{U_n}  G^i(\mathrm{d}z)\mathrm{d}s\right),
\end{aligned}
\end{equation}
if the limit exists in ${\cal D}_{\text{ucp}}$, $i=1,2$. In this case $\zeta$ is called forward integrable with respect to $\tilde N^i$, $i=1,2$. If the limit exists in probability for every $t\in [0,T]$, we say $\zeta$ is forward integrable in the weak sense.
Here, $\{U_n\}_{n=1}^\infty$ is an increasing sequence of compact sets $U_n\subset \mathbb{R}_0$ with $G^i(U_n)<\infty$ such that $\lim_{n\rightarrow \infty}U_n=\mathbb{R}_0$, $i=1,2$. 
  \end{definition}

Similar to the case of the Brownian motion, we also have the following properties about the forward integral with respect to the compensated Poisson random measure.

          \begin{proposition}\label{pro2}
Suppose $\zeta$ is forward integrable with respect to $\tilde N^i$, $i=1,2$, and $F$ is a random variable. Then $F\zeta$ is forward integrable and
 \begin{equation} 
\begin{aligned}
\int_0^T\int_{\mathbb{R}_0}F \zeta(t,z)\tilde N^i(\mathrm{d}^-t,\mathrm{d}z)=F\int_0^T\int_{\mathbb{R}_0} \zeta(t,z)\tilde N^i(\mathrm{d}^-t,\mathrm{d}z),\quad i=1,2.
\end{aligned}
\end{equation}
   \end{proposition}

     \begin{proposition} \label{fgen2}
Let $ \{{\cal E}_t\}_{0\le t\le T}$(${\cal F}_t\subset {\cal E}_t$, $0\le t\le T$)  be another filtration which satisfies the usual condition. Suppose that the ${\cal E}_t$-compensator\footnote{Since $N^i(\mathrm{d}t,\mathrm{d}z)$ is also the jump measure of the ${\cal E}_t$-adapted c\`{a}dl\`{a}g process $\eta^i_t$, $N^i(\mathrm{d}t,\mathrm{d}z)$ is also an ${\cal E}_t$-integer-valued random measure, which has a unique ${\cal E}_t$-compensator (the L\'{e}vy system) $\hat N_{{\cal E}}^i(\mathrm{d}t,\mathrm{d}z)$ with respect to $\{{\cal E}_t\}_{0\le t\le T}$, $i=1,2$ (see Definition \ref{integer_rm} and Remark \ref{r_integer_rm}).} $\hat N^i_{\cal {E}}(\mathrm{d}t,\mathrm{d}z)$ of  $N^i(\mathrm{d}t,\mathrm{d}z)$ is absolutely continuous in time, i.e., $\hat N^i_{\cal {E}}(\mathrm{d}t,\mathrm{d}z)=G_{\cal E}^i(t,\mathrm{d}z)\mathrm{d}t$, and $\zeta$ is an ${\cal E}_t$-predictable random field which is It\^{o} integrable with respect to $\tilde N^i$\footnote{When $\hat N^i_{\cal {E}} \neq \hat N^i $, $i=1,2$, we require the existence of the  It\^{o} integral $\int_0^T\int_{\mathbb{R}_0}\zeta(t,z)\tilde N^i(\mathrm{d}t,\mathrm{d}z):=\int_0^T\int_{\mathbb{R}_0}\zeta(t,z)\tilde N_{\cal {E}}^i(\mathrm{d}t,\mathrm{d}z)+\int_0^T\int_{\mathbb{R}_0}\zeta(t,z)\hat N_{\cal {E}}^i(\mathrm{d}t,\mathrm{d}z)-\int_0^T\int_{\mathbb{R}_0}\zeta(t,z)G^i(\mathrm{d}z)\mathrm{d}t$.}, $i=1,2$. Then $\zeta$ is forward integrable and
 \begin{equation} 
\begin{aligned}
\int_0^T\int_{\mathbb{R}_0} \zeta(t,z)\tilde N^i(\mathrm{d}^-t,\mathrm{d}z)= \int_0^T\int_{\mathbb{R}_0} \zeta(t,z) \tilde N^i(\mathrm{d}t,\mathrm{d}z),\quad i=1,2.
\end{aligned}
\end{equation}
   \end{proposition}

     \begin{proof}
     When the ${\cal E}_t$-compensator $\hat N^i_{\cal {E}} \neq \hat N^i $, $i=1,2$, it is an immediate consequence of the convergence property for the It\^{o} integral with respect to random measures (see \cite[Proposition 3.39]{Eberlein19}). Thus, we only consider the case when $\hat N^i_{\cal {E}} =\hat N^i $, $i=1,2$. By Theorem 11.22 in \cite{He92},  $\zeta 1_{U_n}$ is $\hat N^i$-integrable (hence $\tilde N^i$-integrable), $i=1,2$, which ensures the validity of the two terms in the right hand of (\ref{forwardN}). The conclusion follows from Proposition 3.39 in \cite{Eberlein19}.
        
     \end{proof}

   The It\^{o} formula for the forward integral with respect to the Brownian motion and the compensated Poisson random measure was first proved in \cite{Russo00} and \cite{DiNunno05}, respectively. Here we give the It\^{o} formula for the forward integral with respect to the L\'{e}vy process.

     \begin{theorem}\label{itofor}
  Consider a process of the form 
  \begin{equation}
  Y_t=Y_0+\int_0^t\alpha_s\mathrm{d}s+\sum_{i=1}^2\int_0^t\varphi_i(s)\mathrm{d}^-W^i_s+\sum_{i=1}^2\int_0^t\int_{\mathbb{R}_0}\zeta_i(s,z)\tilde N^i(\mathrm{d}^-s,\mathrm{d}z),
  \end{equation}
   where $Y_0$ is a constant, $\alpha$ is a c\`{a}gl\`{a}d process, $\varphi_i$ is c\`{a}gl\`{a}d and forward integrable with respect to $W^i$, and $\zeta_i(t,z)$ is c\`{a}gl\`{a}d in $t$ for fixed $z$ and continuous in $z$ around zero for a.a. $(t,\omega)$ and forward integrable with respect to $\tilde N^i$, $i=1,2$. Moreover, suppose that 
     \begin{equation*}
     \int_0^T |\alpha_t| \mathrm{d}t+  \sum_{i=1}^2   \int_0^T|\varphi_i(t)|^2 \mathrm{d}t+\sum_{i=1}^2\int_0^T\int_{\mathbb{R}_0}|\zeta_i(t,z)|^2 G^i(\mathrm{d}z)\mathrm{d}t<\infty.
       \end{equation*}
    Then for any function $f\in C^2(\mathbb{R})$ we have (in the weak sense)
   \begin{equation}\label{equitos}
\begin{aligned}
f(Y_t)=&f(Y_0)+\int_0^tf'(Y_s)\alpha_s\mathrm{d}s+\sum_{i=1}^2\int_0^tf'(Y_s)\varphi_i(s)\mathrm{d}^-W^i_s+\frac{1}{2}\sum_{i=1}^2\int_0^tf''(Y_s)\varphi_i^2(s)\mathrm{d}s\\
&+\sum_{i=1}^2\int_0^t\int_{\mathbb{R}_0}\left[f(Y_{s-}+\zeta_i(s,z))-f(Y_{s-})\right]\tilde N^i(\mathrm{d}^-s,\mathrm{d}z)\\
&+\sum_{i=1}^2\int_0^t\int_{\mathbb{R}_0}\left[f(Y_{s-}+\zeta_i(s,z))-f(Y_{s-})-f'(Y_{s-})\zeta_i(s,z)\right]G^i(\mathrm{d}z)\mathrm{d}s.
\end{aligned}
\end{equation}
 \end{theorem}
        \begin{proof}
    We refer to \cite{Russo00} for the proof in the Brownian motion case and to \cite{DiNunno05} for the compensated Poisson random measure case.
        
     \end{proof}

\section{Model formulation}
 \label{Sec:model}
We assume that all uncertainties come from the fixed filtered probability space $(\Omega,{\cal F},\{{\cal F}_t\}_{t\ge 0},{\mathbb P})$ defined in Section \ref{forwardin} in our model. Fix a terminal time $T>0$. Suppose all filtrations given in this section satisfy the usual condition. 

Consider an insurer who can invest in the financial market. Two assets are available for investment, one risk-free asset (bond) $B$ and one risky asset (stock) $S$, whose dynamics are governed by the following anticipating stochastic differential equations (SDEs)  
 \begin{eqnarray} 
  \left\{ \begin{aligned}&\mathrm{d}B_t=r_tB_t\mathrm{d}t,\quad 0\le t\le T,
  \\&   \mathrm{d}S_t=S_{t-}\left(  \mu(t,\pi_t)\mathrm{d}t+\sigma_t\mathrm{d}^- W^1_t+\int_{\mathbb{R}_0}\gamma_1(t,z)\tilde N^1(\mathrm{d}^-t,\mathrm{d}z)  \right), \quad 0\le t\le T,\end{aligned} \right.   
  \label{fin-m}
 \end{eqnarray}
with constant initial values $1$ and $S_0>0$, respectively. Suppose that there are large investors in the market and they have access to insider information which can be characterized by a larger filtration $\{{\cal G}^1_t\}_{0\le t\le T}$ (i.e., ${\cal F}_t\subset {\cal G}^1_t$, $0\le t\le T$). Assume the insurer is also a large investor and owns insider information  characterized by another filtration $\{{\cal H}_t\}_{0\le t\le T}$ such that
\begin{equation}
{\cal F}_t\subset{\cal G}^1_t\subset {\cal H}_t, \quad0\le t\le T.
\end{equation}
 The investment strategies they take influence the coefficients of the financial asset processes. Thus, we suppose that $r_t, {\mu}(t,x), \sigma_t$ and $\gamma_1(t,z)$ , $t\in [0,T]$, $x\in \mathbb{R}$, $z\in{\mathbb R}_0$ are all c\`{a}gl\`{a}d stochastic coefficients adapted to $\{{\cal G}^1_t\}$ for fixed $x$ and $z$, and $\mu(t,\cdot)$ is $C^1$ for every $t\in[0,T]$. The mean rate of return $\mu$ on the risky asset partly depends on the investment strategy $\pi$ of the insurer (see \cite{Kohatsu06, DiNunno09}). Here,  the investment strategy $\pi_t$ is defined as an ${\cal H}_t$-adapted c\`{a}gl\`{a}d process, which represents the proportion of the insurer's total wealth $X_t$ invested in the risky asset $S_t$ at time $t$.

The insurer's risk (per policy) is given by 
 \begin{eqnarray} 
 \mathrm{d}R_t=a_t\mathrm{d}t+b_t\mathrm{d}^-\bar W_t+\int_{\mathbb{R}_0}\gamma_2(t,z)\tilde N^2(\mathrm{d}^-t,\mathrm{d}z),\quad 0\le t\le T,
  \label{ins-m}
 \end{eqnarray}
with zero initial value. Here, we choose $\bar W_t:=\rho W^1_t+\sqrt{1-\rho^2}W^2_t$ with $-1< \rho\le0$ to describe the correlation between the insurer's liabilities and her capital income from financial investment (see \cite{Zou14}). This means that claims are the required payments to the insured holders due to either defaults of the obligors or for collateral calls when the prices of the insured risky asset declines. Similarly, the coefficients might be influenced by the uncertain economic environment characterized by a larger filtration $\{{\cal G}^2_t\}_{0\le t\le T}$ such that
\begin{equation}
{\cal F}_t\subset {\cal G}^2_t\subset {\cal H}_t, \quad 0\le t\le T.
\end{equation}
 Thus, we suppose that $a_t$, $b_t$ and $\gamma_2(t,z)$, $t\in[0,T]$, $z\in\mathbb{R}_0$ are all c\`{a}gl\`{a}d stochastic coefficients adapted to $\{{\cal G}^2_t\}$ for fixed $z$ (see \cite{Peng18}). 
 
Note that $W^i$ and $\eta^i$, $i=1,2$, may not be semimartingales with respect to $\{{\cal G}^1_t\}$ or $\{{\cal G}^2_t\}$, and the forward integrals defined in Section \ref{forwardin} are the natural generalization of semimartingale It\^{o} integrals (see Propositions \ref{fgen1} and \ref{fgen2}). Thus, the integrals with respect to $W^1$, $W^2$, $\tilde N^1$ and $\tilde N^2$ in (\ref{fin-m})  and (\ref{ins-m}) are viewed as the forward integrals instead of the It\^{o} integrals. 

Denote by $\hat N_{\cal H}^i(\mathrm{d}t,\mathrm{d}z)$ the ${\cal H}_t$-compensator of $N^i(\mathrm{d}t,\mathrm{d}z)$, $i=1,2$. We make some assumptions on the coefficients as follows:
 
\begin{itemize}
 \item $\sigma\ge \epsilon>0$ (for some positive constant $\epsilon $) is forward integrable with respect to $W^1$, $b $ is forward integrable with respect to $W^1$ and $W^2$, $\ln(1+\gamma_1(t,z))$ and $\gamma_2(t,z)\ge\epsilon>0$ (for some positive constant $\epsilon $) are continuous in $z$ around zero for a.a. $(t,\omega)$ and forward integrable with respect to $\tilde N^1$ and $\tilde N^2$, respectively;
\end{itemize}
\begin{itemize}
 \item  For each investment strategy $\pi$ given above,
 \begin{equation*}\begin{aligned}
 &\int_0^T \left(|r_t|+|\mu(t,\pi_t)| +|a_t| \right)\mathrm{d}t+\int_0^T\left(\sigma_t^2+b_t^2\right)\mathrm{d}t+\int_0^T\int_{\mathbb{R}_0}\left(\left|\ln(1+\gamma_1(t,z))\right|^2+|\gamma_1(t,z)|^k\right)G^1(\mathrm{d}z)\mathrm{d}t\\
 &+\int_0^T\int_{\mathbb{R}_0}|\gamma_2(t,z)|^kG^2(\mathrm{d}z)\mathrm{d}t+\int_0^T\int_{\mathbb{R}_0}\left|\ln(1+\gamma_1(t,z))-\gamma_1(t,z)\right|G^1(\mathrm{d}z)\mathrm{d}t\\
 &+\int_0^T\int_{\mathbb{R}_0}|\gamma_1(t,z)|^k\hat N_{\cal H}^1(\mathrm{d}t,\mathrm{d}z) +\int_0^T\int_{\mathbb{R}_0}|\gamma_2(t,z)|^k\hat N_{\cal H}^2(\mathrm{d}t,\mathrm{d}z) <\infty, \quad k=1,2.
 \end{aligned} \end{equation*}
\end{itemize}

Under the above conditions, we can solve the anticipating SDEs (\ref{fin-m}) and (\ref{ins-m}) by using the It\^{o} formula for forward integrals (Theorem \ref{itofor}). Moreover, the process of the risky asset $S$ is given by
 \begin{equation}
 \begin{aligned}
S_t=S_0\exp \Bigg\{   &   \int_0^t  \left(  \mu(s,\pi_s)-\frac{1}{2} \sigma^2_s \right) \mathrm{d}s+\int_0^t \sigma_s \mathrm{d}^-W^1_s +\int_0^t\int_{\mathbb{R}_0}\ln(1+\gamma_1(s,z))\tilde N^1( \mathrm{d}^-s, \mathrm{d}z)\\
&+\int_0^t\int_{\mathbb{R}_0}\left[     \ln(1+\gamma_1(s,z))-\gamma_1(s,z)\right]G^1( \mathrm{d}z) \mathrm{d}s \Bigg\}, \quad 0\le t\le T.
\end{aligned}
 \end{equation}
 
Let $\kappa$ be the liability ratio process of the insurer (see \cite{Zou14}). Then $\kappa_tX_t$ is the total number of insurance policies the insurer choose at time $t$. Since the insider information of the insurer also contains the insurance claim, the insurance strategy $\kappa $ is defined as an ${\cal H}_t$-adapted c\`{a}gl\`{a}d process as well. 

Note that both the investment strategy $\pi$ and the insurance strategy $\kappa$ can take negative values, which is to be interpreted as short-selling the risky asset and buying insurance policies from other insurers, respectively. Denote by $u=(\pi,\kappa)$ the total strategy for the insurer. Then her wealth process $X^u$ corresponding to $u$ is governed by the following anticipating SDE (see \cite[page 372]{Karatzas91} for the deduction):
 \begin{equation}\label{wealthsde}
 \begin{aligned}
\frac{\mathrm{d}X^u_t}{X^u_{t-}}=&\left[r_t+(\mu(t,\pi_t)-r_t)\pi_t+(\lambda_t-a_t)\kappa_t\right]\mathrm{d}t+\left( \sigma_t\pi_t-\rho b_t\kappa_t   \right)\mathrm{d}^-W^1_t-\sqrt{1-\rho^2}b_t\kappa_t\mathrm{d}^-W^2_t\\
&+ \int_{\mathbb{R}_0}\pi_t\gamma_1(t,z)\tilde N^1(\mathrm{d}^-t,\mathrm{d}z)- \int_{\mathbb{R}_0}\kappa_t\gamma_2(t,z)\tilde N^2(\mathrm{d}^-t,\mathrm{d}z), \quad 0\le t\le T,
\end{aligned}
 \end{equation}
with constant initial value $X_0>0$. Here, $\lambda_t$ is defined as a ${\cal G}^2_t$-adapted c\`{a}gl\`{a}d process such that $\int_0^T|\lambda_t|\mathrm{d}t<\infty$ and $\lambda_t>a_t>0$ (in the sense of $\mathbb{P}$-a.s.) for every $t\in[0,T]$. It represents the premium per policy for the insurer at time $t$. 

We can apply the It\^{o} formula for forward integrals to solve the anticipating SDE (\ref{wealthsde}). Before that, we impose the following admissible conditions on $u$.

\begin{definition}\label{admiss}
We define ${\cal A}_1$ as the set of all above strategies $u=(\pi,\kappa)$ satisfying the following conditions:
\begin{itemize}
 \item[(i)] $\sigma\pi$ is forward integrable with respect to $W^1$, $b\kappa$ is forward integrable with respect to $W^1 $ and $W^2$, $\ln(1+\pi_t\gamma_1(t,z))$ ($\pi_t\gamma_1>-1+\epsilon_\pi$ for a.e. $(t,z)$ with respect to both $\mathrm{d}t\times G^1(\mathrm{d}z)$ and $ \hat N_{\cal H}^1(\mathrm{d}t,\mathrm{d}z)$ for some $\epsilon_\pi\in(0,1)$ depending on $\pi$) and $\ln(1-\kappa_t\gamma_2(t,z))$ ($\kappa\gamma_2<1-\epsilon_\kappa$ for a.e. $(t,z)$ with respect to both $\mathrm{d}t\times G^2(\mathrm{d}z)$ and $  \hat N_{\cal H}^2(\mathrm{d}t,\mathrm{d}z)$ for some $\epsilon_\kappa\in(0,1)$ depending on $\kappa$) are continuous in $z$ around zero for a.a. $(t,\omega)$ and forward integrable with respect to $\tilde N_1$ and $\tilde N_2$, respectively;
 \end{itemize}
 \begin{itemize}
\item[(ii)]  Assume the following integrability:
 \begin{equation*}\begin{aligned}
 &\int_0^T \left[ |\mu(t,\pi_t)-r_t||\pi_t|+|\lambda_t-a_t||\kappa_t|+|\sigma_tb_t\pi_t\kappa_t|\right]\mathrm{d}t+\int_0^T\left(\sigma^2_t\pi_t^2+b_t^2\kappa_t^2\right) \mathrm{d}t\\
 &+\int_0^T\int_{\mathbb{R}_0}\left(  \left|\ln(1+\pi_t\gamma_1(t,z))\right|^k |\pi_t\gamma_1(t,z)|^k\right)G^1(\mathrm{d}z)\mathrm{d}t+\int_0^T\int_{\mathbb{R}_0}\left(\left|\ln(1-\kappa_t\gamma_2(t,z))\right|^k+|\kappa_t\gamma_2(t,z)|^k   \right)G^2(\mathrm{d}z)\mathrm{d}t\\
 & +\int_0^T\int_{\mathbb{R}_0}\left(  \left|\ln(1+\pi_t\gamma_1(t,z))\right|^k  + |\pi_t\gamma_1(t,z)|^k\right)\hat N_{\cal H}^1(\mathrm{d}t,\mathrm{d}z)+\int_0^T\int_{\mathbb{R}_0}\left(\left|\ln(1-\kappa_t\gamma_2(t,z))\right|^k+|\kappa_t\gamma_2(t,z)|^k   \right)\hat N_{\cal H}^2(\mathrm{d}t,\mathrm{d}z)  \\
 &<\infty,\quad k=1,2.
 \end{aligned} \end{equation*}
\end{itemize}
\end{definition}

Let $u\in{\cal A}_1$. By Theorem \ref{itofor}, the solution of (\ref{wealthsde}) is given by
 \begin{equation}\label{equofxt}
 \begin{aligned}
X^u_t=X_0\exp \Bigg\{   &   \int_0^t  \left[ r_s+(\mu(s,\pi_s)-r_s)\pi_s+(\lambda_s-a_s)\kappa_s-\frac{1}{2}\sigma_s^2 \pi_s^2+\rho\sigma_sb_s\pi_s\kappa_s-\frac{1}{2}b_s^2\kappa_s^2 \right] \mathrm{d}s\\
&+\int_0^t (\sigma_s\pi_s-\rho b_s\kappa_s) \mathrm{d}^-W^1_s -\int_0^t\sqrt{1-\rho^2}b_s\kappa_s\mathrm{d}^-W_s^2+\int_0^t\int_{\mathbb{R}_0}\ln (1+\pi_s\gamma_1(s,z))\tilde N^1(\mathrm{d}^-s,\mathrm{d}z)\\
&+\int_0^t\int_{\mathbb{R}_0}\ln(1-\kappa_s\gamma_2(s,z))\tilde N^2( \mathrm{d}^-s, \mathrm{d}z)+\int_0^t\int_{\mathbb{R}_0}\left[     \ln(1+\pi_s\gamma_1(s,z))-\pi_s\gamma_1(s,z)\right]G^1( \mathrm{d}z) \mathrm{d}s\\
&+\int_0^t\int_{\mathbb{R}_0}\left[     \ln(1-\kappa_s\gamma_2(s,z))+\kappa_s\gamma_2(s,z)\right]G^2( \mathrm{d}z) \mathrm{d}s \Bigg\}, \quad 0\le t\le T.
\end{aligned}
 \end{equation}

Consider a model uncertainty setup. Assume that the insurer is ambiguity averse, implying that she is concerned about the accuracy of statistical estimation, and possible misspecification errors. Thus, a family of parametrized subjective probability measures $\{{{\cal{Q}}}^v\}$ equivalent to the original probability measure $\mathbb{P}$ is assumed to be exist in the real world (see \cite{Gu20}). However, since the insurer has insider information filtration $\{{\cal H}_t\}$ under which $W^i$ and $\eta^i$ might not be semimartingales, $i=1,2$, a generalization for the construction of $\{{{\cal{Q}}}^v\}$ need to be considered by means of the forward integral. 

\begin{definition}\label{admiss_1}
We define ${\cal A}_2$ as the set of all ${\cal H}_t$-adapted c\`{a}gl\`{a}d processes $v_t=(\theta_1(t),\theta_2(t),\theta_3(t),\theta_4(t))$, $t\in [0,T]$, satisfying the following conditions:
\begin{itemize}
 \item[(i)] $\theta_1(t)$, $\theta_2(t)$, $\theta_3(t)$, $\theta_4(t)$ are all forward integrable with respect to $W^1$, $W^2$, $\tilde N^1$, $\tilde N^2$, respectively;
 \end{itemize}
 \begin{itemize}
 \item[(ii)]  Assume the following integrability:
  \begin{equation*}
 \begin{aligned}
\mathbb{E}\Bigg\{ & \int_0^T\left[\theta_1^2(s)+\theta_2^2(s)\right]\mathrm{d}s+ \int_0^T\int_{\mathbb{R}_0}\left[ |\ln(1+\theta_3(s))| ^k+|\theta_3(s)|^k\right]G^1(\mathrm{d}z)\mathrm{d}s \\
&+ \int_0^T\int_{\mathbb{R}_0}\left[ |\ln(1+\theta_4(s))| ^k+|\theta_4(s)|^k\right]G^2(\mathrm{d}z)\mathrm{d}s+\int_0^T\int_{\mathbb{R}_0}\left[ |\ln(1+\theta_3(s))| ^k+|\theta_3(s)|^k\right]\hat N_{{\cal H}}^1(\mathrm{d}s,\mathrm{d}z)\\
 & +\int_0^T\int_{\mathbb{R}_0}\left[ |\ln(1+\theta_4(s))|^k +|\theta_4(s)|^k \right]\hat N_{{\cal H}}^2(\mathrm{d}s,\mathrm{d}z)  \Bigg\}<\infty, \quad k=1,2;
    \end{aligned}
 \end{equation*}
 \end{itemize}
 \begin{itemize}
\item[(iii)]   $\int_0^t\theta_1(s)\mathrm{d}^-W^1_s+\int_0^t\theta_2(s)\mathrm{d}^-W^2_s+\int_0^t\int_{\mathbb{R}_0}\theta_3(s)\tilde N^1(\mathrm{d}^-s,\mathrm{d}z)+\int_0^t\int_{\mathbb{R}_0}\theta_4(s)\tilde N^2(\mathrm{d}^-s,\mathrm{d}z)$ is an ${\cal H}_t$-special semimartingale (see Appendix \ref{appendix_A} for the definition), the local martingale part (in the canonical decomposition) of which satisfies the Novikov condition (see Theorem \ref{novikov}).
\end{itemize}
\end{definition}

For $v\in {\cal A}_2$, the Dol\'{e}ans-Dade exponential $\varepsilon^v_t$ is the unique ${\cal H}_t$-martingale governed by (see Theorem \ref{novikov})
 \begin{equation}\label{theta_sde}
 \begin{aligned}
 \varepsilon^v_t=1+\varepsilon^v_{t-}\Bigg(  &\int_0^t\theta_1(s)\mathrm{d}^-W^1_s+\int_0^t\theta_2(s)\mathrm{d}^-W^2_s+\int_0^t\int_{\mathbb{R}_0}\theta_3(s)\tilde N^1(\mathrm{d}^-s,\mathrm{d}z)+\int_0^t\int_{\mathbb{R}_0}\theta_4(s)\tilde N^2(\mathrm{d}^-s,\mathrm{d}z) \Bigg)^M, \quad 0\le t\le T,
\end{aligned}
 \end{equation}
where $(\cdot)^M$ denotes the local martingale part of a special semimartingale. Thus, we have $\varepsilon^v_T>0$ and $\int_\Omega\varepsilon^v_T\mathrm{d}\mathbb{P}=1$, which induces a probability ${{\cal{Q}}}^v$ equivalent to $\mathbb{P}$ such that $ \frac{\mathrm{d}{\cal{Q}}^v}{\mathrm{d}\mathbb{P}} =\varepsilon^v_T$. Then all such ${\cal{Q}}^v$ form a set of subjective probability measures $\{{\cal{Q}}^v\}_{v\in{\cal A}_2}$.

Taking into account the extra insider information and model uncertainty, the optimization problem for the insurer can be formulated as a (zero-sum) anticipating stochastic differential game. In other words, we need to solve the following problem.

\begin{problem}\label{sdg}
 Select a pair $(u^*,v^*)\in {\cal A}_1'\times {\cal A}_2'$ (see Definition \ref{admiss12} below) such that
\begin{equation}\label{max_u}
V:=J(u^*,v^*)= \sup_{u\in {\cal A}_1'} \inf_{v\in{\cal A}_2'} J(u,v)=  \inf_{v\in{\cal A}_2'}\sup_{u\in {\cal A}_1'}  J(u,v),
\end{equation}
where the performance functional is given by
\begin{equation}\begin{aligned}
J(u,v):=\mathbb{E}_{{\cal{Q}}^v}\left[   U(X_T^u) +\int_0^Tg(s,v_s)\mathrm{d}s\right]=\mathbb{E}\left[  \varepsilon_T^v U(X_T^u) +\int_0^T \varepsilon_s^vg(s,v_s)\mathrm{d}s\right],
\end{aligned}\end{equation}
 the utility function $U:(0,\infty)\rightarrow \mathbb{R}$ is a strictly increasing and concave function with a strictly decreasing derivative, the penalty function $g: [0,T]\times\mathbb{R}^4\times\Omega\rightarrow \mathbb{R}$ is a measurable function and Fr\'{e}chet differentiable in $v$. The term $\mathbb{E}_{{\cal{Q}}^v}\left[\int_0^T g(s,v_s)\mathrm{d}s \right]$ is viewed as a step adopted to penalize the difference between ${\cal{Q}}^v$ and $\mathbb{P}$.  We call $V$ the value (or the optimal expected utility under the worst-case probability) of Problem \ref{sdg}.
\end{problem}

\begin{definition}\label{admiss12}
Define ${\cal A}_1'$ as some subset of $ {\cal A}_1$ with $\mathbb{E}\left[|U(X_T^u)|^2+|U'(X_T^u)X_T^u|^2\right]<\infty$ for all $u\in{\cal A}_1'$. Define ${\cal A}_2'$ as some subset of $  {\cal A}_2$ with $\mathbb{E}\left[ |{\varepsilon_T^v}|^2  + \int_0^T|g(s,v_s)|^2\mathrm{d}s \right]<\infty$ for all $v\in {\cal A}_2'$.
\end{definition}
 
 \begin{remark}\label{aremarkofl2m}
For $v\in{\cal A}_2'$, $\left|\varepsilon_t^v\right|^2 $ is an ${\cal H}_t$-submartingale and $\mathbb{E} |\varepsilon_t^v|^2\le \mathbb{E} |\varepsilon_T^v |^2<\infty$ for all $t\in[0,T]$ by the Jensen inequality. 
 \end{remark}
 
 \begin{remark}
Here, we suppose ${\cal A}_1'$ and ${\cal A}_2'$ are some subsets of ${\cal A}_1$ and ${\cal A}_2$ with the above integrability conditions, respectively, since we need more assumptions in Sections \ref{half_section}-\ref{sec_example2} to characterize the optimal pair $(u^*,v^*)\in {\cal A}_1'\times{\cal A}_2'$. If $u^*\in {\cal A}_1'$ or $v^*\in{\cal A}_2'$ dose not satisfy those assumptions, respectively, we can impose those assumptions to ${\cal A}_1'$ or ${\cal A}_2'$ to narrow the two sets such that $(u^*,v^*)\in {\cal A}_1'\times{\cal A}_2'$ fits those assumptions (see Remark \ref{remark_themain1}).
 \end{remark}

\section{A half characterization of the robust optimal strategy}
 \label{half_section}

 \begin{assumption}\label{assump1}
If $(u^*,v^*)\in {\cal A}_1'\times {\cal A}_2'$ is robust optima for Problem \ref{sdg}, then for all bounded $\alpha \in {\cal A}_1' $, there exists some $\delta>0$ such that $ u^*+y\alpha\in {\cal A}_1' $ for all $|y|<\delta$. Moreover, the following family of random variables 
  \begin{equation*}
 \begin{aligned}
\left \{  \varepsilon^{v^* }_T U'(X_T^{u^*+y\alpha}) \frac{\mathrm{d}}{\mathrm{d}y}X_T^{u^*+y\alpha}   \right\}_{y\in(-\delta,\delta)}
 \end{aligned}
 \end{equation*}
is $\mathbb{P}$-uniformly integrable, where $\frac{\mathrm{d}}{\mathrm{d}y}$ means that $\frac{\mathrm{d}}{\mathrm{d}y}X_T^{u^*+y\alpha}$ exists and the interchange of differentiation and integral with respect to $\ln X_T^{u^*+y\alpha}$ in (\ref{equofxt}) is justified.
 \end{assumption}
 
  \begin{assumption}\label{assump2}
  Let $u_s=\left(\vartheta_1 1_{(t,t+h]}(s),\vartheta_21_{(t,t+h]}(s)\right)$, $0\le s\le T$, for fixed $0\le t<t+h\le T$, where the random variable $\vartheta_i$ is of the form $1_{A_t}$ for any ${\cal H}_t$-measurable set $A_t$, $i=1,2$. Then $u\in{\cal A}_1'$.
  \end{assumption}
  
  \begin{theorem}\label{mainth1}
Suppose $(u^*,v^*)\in {\cal A}_1'\times {\cal A}_2'$ is optimal for Problem \ref{sdg} under Assumptions \ref{assump1} and \ref{assump2}. Then the following stochastic processes
    \begin{equation}\label{equ1ofth1}
 \begin{aligned}
m_1^{u^*}(t):=&\int_0^t   \left(\mu(s,\pi_s^*)-r_s+\frac{\partial}{\partial x}\mu(s,\pi^*_s)\pi^*_s-\sigma^2_s\pi^*_s+\rho\sigma_sb_s\kappa_s^*\right)\mathrm{d}s+\int_0^t \sigma_s\mathrm{d}^-W_s^1 \\
&+\int_0^t\int_{\mathbb{R}_0}\frac{ \gamma_1(s,z)}{1+\pi_s^*\gamma_1(s,z)}\tilde N^1(\mathrm{d}^-s,\mathrm{d}z) -\int_0^t\int_{\mathbb{R}_0}\frac{ \pi_s^*\gamma^2_1(s,z)}{1+\pi_s^*\gamma_1(s,z)} G^1(\mathrm{d} z)\mathrm{d}s,\quad 0\le t\le T,
\end{aligned}
 \end{equation}
 and
    \begin{equation}\label{equ2ofth1}
 \begin{aligned}
m_2^{u^*}(t):=&\int_0^t   \left(\lambda_s-a_s+\rho\sigma_sb_s\pi^*_s-b^2_s\kappa_s^*\right)\mathrm{d}s-\int_0^t \rho b_s\mathrm{d}^-W_s^1-\int_0^t \sqrt{1-\rho^2}b_s\mathrm{d}^-W_s^2 \\
&-\int_0^t\int_{\mathbb{R}_0}\frac{ \gamma_2(s,z)}{1-\kappa_s^*\gamma_2(s,z)}\tilde N^2(\mathrm{d}^-s,\mathrm{d}z) 
-\int_0^t\int_{\mathbb{R}_0}\frac{ \kappa_s^*\gamma^2_2(s,z)}{1-\kappa_s^*\gamma_2(s,z)} G^2(\mathrm{d} z)\mathrm{d}s,\quad 0\le t\le T,
\end{aligned}
 \end{equation}
 are $({\cal H}_t,Q_{u^*,v^*})$-martingales, where $F_{u^*,v^*}(s):=\mathbb{E}\Bigg[ \frac{\varepsilon_T^{v^*}U'(X_T^{u^*}) X_T^{u^*}}{\mathbb{E}\left(\varepsilon_T^{v^*}U'(X_T^{u^*}) X_T^{u^*}\right)}  \big| {\cal H}_s \Bigg]$, $s\in[0, T]$, and $Q_{u^*,v^*}$ is an equivalent probability measure of $\mathbb{P}$ on $(\Omega, {\cal F})$ defined by $Q_{u^*,v^*}(\mathrm{d}\omega):=F_{u^*,v^*}(T)\mathbb{P}(\mathrm{d}\omega)$.
  \end{theorem}
  
 \begin{proof}
 Suppose that the pair $(u^*,v^*)\in {\cal A}_1'\times{\cal A}_2'$ is optimal. Then for any bounded $(\alpha_1,0) \in {\cal A}_1' $ and $|y|<\delta$, we have $J(u^*+y(\alpha_1,0),v^*)\le J(u^*,v^*) $, which implies that $y=0$ is a maximum point of the function $y\mapsto J(u^*+y(\alpha_1,0),v^*)$. Thus, we have $\frac{\mathrm{d}}{\mathrm{d}y}J(u^*+y(\alpha_1,0),v^*){|}_{y=0}=0$ once the differentiability is established. Thanks to Assumption \ref{assump1}, we can deduce by (\ref{equofxt}) that 
  \begin{equation}\label{euq_bianfen_1}
 \begin{aligned}
&\frac{\mathrm{d}}{\mathrm{d}y}J(u^*+y(\alpha_1,0),v^*){|}_{y=0}
\\=&\mathbb{E}\Bigg\{  \varepsilon_T^{v^*}U'(X_T^{u^*}) X_T^{u^*}
\Bigg[\int_0^T  \alpha_1(s)\Big(\mu(s,\pi_s^*)-r_s+\frac{\partial}{\partial x}\mu(s,\pi^*_s)\pi^*_s-\sigma^2_s\pi^*_s+\rho\sigma_sb_s\kappa_s^*\Big)\mathrm{d}s\\
&+\int_0^T\alpha_1(s)\sigma_s\mathrm{d}^-W_s^1 +\int_0^T\int_{\mathbb{R}_0}\frac{\alpha_1(s)\gamma_1(s,z)}{1+\pi_s^*\gamma_1(s,z)}\tilde N^1(\mathrm{d}^-s,\mathrm{d}z)-\int_0^T\int_{\mathbb{R}_0}\frac{\alpha_1(s)\pi_s^*\gamma^2_1(s,z)}{1+\pi_s^*\gamma_1(s,z)} G^1(\mathrm{d} z)\mathrm{d}s
\Bigg]
\Bigg\}\\
=&0.
 \end{aligned}
 \end{equation}
 Now let us fix $0\le t<t+h\le T$. By Assumption \ref{assump2}, we can choose $(\alpha_1,0) \in{\cal A}_1'$ of the form
   \begin{equation*}
 \begin{aligned}
 \alpha_1(s)=\vartheta_1 1_{(t,t+h]}(s),\quad 0\le s\le T,
   \end{aligned}
 \end{equation*}
 where $\vartheta_1=1_{A_t}$ for an arbitrary ${\cal H}_t$-measurable set $A_t$. Then (\ref{euq_bianfen_1}) gives
    \begin{equation*}
 \begin{aligned}
0=\mathbb{E}_{Q_{u^*,v^*}}\Bigg\{  \vartheta_1
\Bigg[&\int_t^{t+h}   \left(\mu(s,\pi_s^*)-r_s+\frac{\partial}{\partial x}\mu(s,\pi^*_s)\pi^*_s-\sigma^2_s\pi^*_s+\rho\sigma_sb_s\kappa_s^*\right)\mathrm{d}s+\int_t^{t+h}\sigma_s\mathrm{d}^-W_s^1\\
& +\int_t^{t+h}\int_{\mathbb{R}_0}\frac{ \gamma_1(s,z)}{1+\pi_s^*\gamma_1(s,z)}\tilde N^1(\mathrm{d}^-s,\mathrm{d}z) -\int_t^{t+h}\int_{\mathbb{R}_0}\frac{ \pi_s^*\gamma^2_1(s,z)}{1+\pi_s^*\gamma_1(s,z)} G^1(\mathrm{d} z)\mathrm{d}s
\Bigg]
\Bigg\}
   \end{aligned}
 \end{equation*}
 by Propositions \ref{pro1} and \ref{pro2}.
 Since this holds for all such $\vartheta_1$,  we can conclude that
     \begin{equation*}
 \begin{aligned}
0=\mathbb{E}_{Q_{u^*,v^*}}\left[ m_1^{u^*}(t+h)-m_1^{u^*}(t)\big{|}{\cal H}_t\right].
   \end{aligned}
 \end{equation*}
 Hence, $m_1^{u^*}(t)$, $t\in[0,T]$, is an ${\cal H}_t$-martingale under the probability measure $Q_{u^*,v^*}$. By similar arguments, we can deduce that $m_2^{u^*}(t)$, $t\in[0,T]$, is an $ {\cal H}_t $-martingale under $Q_{u^*,v^*}$ as well.  

 \end{proof}

Moreover, we obtain the following result under the original probability measure $\mathbb P$. Unless otherwise stated, all statements are back to $\mathbb{P}$ from now on.

  \begin{theorem}\label{mainth2}
 Suppose $(u^*,v^*)\in {\cal A}_1'\times {\cal A}_2'$ is optimal for Problem \ref{sdg} under Assumptions \ref{assump1} and \ref{assump2}. Then the following stochastic processes
  \begin{equation*}
 \begin{aligned}
\hat m_1^{u^*,v^*}(t):=m_1^{u^*}(t)-\int_0^t F_{u^*,v^*}(s-)\mathrm{d}\big\langle(F_{u^*,v^*})^{-1}, m_1^{u^*}\big\rangle^{Q_{u^*,v^*}}_s,\quad 0\le t\le T,
\end{aligned}
 \end{equation*}
 and
    \begin{equation*}
 \begin{aligned}
\hat m_2^{u^*,v^*}(t):= m_2^{u^*}(t)-\int_0^t F_{u^*,v^*}(s-)\mathrm{d}\big\langle(F_{u^*,v^*})^{-1}, m_2^{u^*}\big\rangle^{Q_{u^*,v^*}}_s,\quad 0\le t\le T,
\end{aligned}
 \end{equation*}
 are ${\cal H}_t$-local martingales, provided that the $({\cal H}_t,{Q_{u^*,v^*}})$-predictable covariation process $\langle(F_{u^*,v^*})^{-1}, m_i^{u^*}\big\rangle^{Q_{u^*,v^*}}_t$, $t\in[0,T]$, exists (see Appendix \ref{appendix_A} for the definition of $\langle\cdot,\cdot\rangle_t$) and is absolutely continuous, $i=1,2$. Here, $m_1^{u^*}(t)$ and $m_2^{u^*}(t)$ are given in Theorem \ref{mainth1}.
  \end{theorem}

\begin{proof}
If $(u^*,v^*)\in{\cal A}_1'\times {\cal A}_2'$ is optimal, then by Theorem \ref{mainth1} we know that $m_i^{u^*}(t)$, $t\in [0,T]$, is an $({\cal H}_t,Q_{u^*,v^*})$-martingale, $i=1,2$. The conclusion is an immediate result from the Girsanov theorem (see Theorem \ref{girsanov}).

 \end{proof}

\begin{remark}\label{remark_themain1}
Here we use the predictable version of the Girsanov theorem, which is different from the original method in \cite{DiNunno09}. The covariation process $[\cdot,\cdot]_t$ is replaced by its compensator $\langle\cdot,\cdot\rangle_t$, which is called the predictable covariation process, in Theorem \ref{mainth2}. One reason is that $\langle\cdot,\cdot\rangle_t$ is predictable and usually pathwise absolutely continuous (see (\ref{reasonfor1})) while $[\cdot,\cdot]_t$ is not (see Remark \ref{whycomp}). In fact, the conditions of the existence and the absolute continuity of $\langle \cdot,\cdot\rangle_t$ can be relaxed when the logarithmic utility is considered (see (\ref{etominus1}), \cite[Theorem 2]{Qian88} and \cite[Theorem 4.26]{Eberlein19}). Or we can just impose these conditions to the definitions of ${\cal A}_1'$ and ${\cal A}_2'$. Another reason is that it could be easier to calculate the It\^{o} integral with respect to $\langle\cdot,\cdot\rangle_t$ than $[\cdot,\cdot]_t$ when $[\cdot,\cdot]_t\neq0$ in practice (see Corollary \ref{mainth6}, Remark \ref{whycomp} and Theorem \ref{girsanov2}). Note that only the case of $[\cdot,\cdot]_t=0$ is considered in \cite{DiNunno09}.
\end{remark}

Further, since $\hat m_1^{u^*,v^*}(t)$, $t\in [0,T]$, is an ${\cal H}_t$-local martingale, we have the orthogonal decomposition of $\hat m_1^{u^*,v^*}(t)$ into a continuous part $\hat m_{1,c}^{u^*,v^*}(t)$ and a purely discontinuous part $\hat m_{1,d}^{u^*,v^*}(t)$, $t\in[0,T]$, which is given by (see \cite{He92,DiNunno09})
    \begin{equation}\label{decomp_1}
 \begin{aligned}
&\hat m_{1,c}^{u^*,v^*}(t):=  \int_0^t \sigma_s\mathrm{d}^-W_s^1 +A^{u^*,v^*}_{1,c}(t),  \quad 0\le t\le T,\\
&\hat m_{1,d}^{u^*,v^*}(t):=   \int_0^t\int_{\mathbb{R}_0}\frac{ \gamma_1(s,z)}{1+\pi_s^*\gamma_1(s,z)}\tilde N^1(\mathrm{d}^-s,\mathrm{d}z)+A^{u^*,v^*}_{1,d}(t)  , \quad 0\le t\le T,
\end{aligned}
 \end{equation}
 where $A^{u^*,v^*}_{1,c}(t)$ and $A^{u^*,v^*}_{1,d}(t)$ are the unique ${\cal H}_t$-adapted (absolutely continuous) bounded variation processes such that
    \begin{equation}\label{decomp_sum}
 \begin{aligned}
A_{1,c}^{u^*,v^*}(t)+A_{1,d}^{u^*,v^*}(t)=&\int_0^t   \left(\mu(s,\pi_s^*)-r_s+\frac{\partial}{\partial x}\mu(s,\pi^*_s)\pi^*_s-\sigma^2_s\pi^*_s+\rho\sigma_sb_s\kappa_s^*\right)\mathrm{d}s\\
&-\int_0^t\int_{\mathbb{R}_0}\frac{ \pi_s^*\gamma^2_1(s,z)}{1+\pi_s^*\gamma_1(s,z)} G^1(\mathrm{d} z)\mathrm{d}s -\int_0^t F_{u^*,v^*}(s-)\mathrm{d}\big\langle(F_{u^*,v^*})^{-1}, m_1^{u^*}\big\rangle^{Q_{u^*,v^*}}_s, \quad 0\le t\le T.
\end{aligned}
 \end{equation}
 
 We can see from (\ref{decomp_1}) that $ \int_0^t \sigma_s\mathrm{d}^-W_s^1$ is a continuous ${\cal H}_t$-semimartingale. Using the fact that $\int_0^t \sigma_s^{-1}\mathrm{d}\hat m_{1,c}^{u^*,v^*}(s)=W^1_t+\int_0^t \sigma_s^{-1}\mathrm{d}A_{1,c}^{u^*,v^*}(s)$ and $\langle  \hat m_{1,c}^{u^*,v^*}\rangle_t= \left\langle  \int_0^\cdot \sigma_s\mathrm{d}^-W_s^1\right\rangle_t=\int_0^t\sigma_s^2\mathrm{d}s$, we have $\langle W^1\rangle_t=t$. Thus,  by the L\'{e}vy theorem (see \cite{Karatzas91}), the canonical decomposition of the continuous ${\cal H}_t$-semimartingale $W^1_t$ can be given as $W^1_t=W^1_{{\cal H}}(t)+\int_0^t\phi_1(s)\mathrm{d}s$, where $W^1_{{\cal H}}(t)$ is an ${\cal H}_t$-Brownian motion and $\phi_1(t)$ is an ${\cal H}_t$-progressively measurable process. Moreover, applying Proposition \ref{fgen1} we have $\int_0^t\sigma_s\mathrm{d}^-W_s^1=\int_0^t\sigma_s\mathrm{d}W^1_{{\cal H}}(s)+\int_0^t\sigma_s\phi_1(s)\mathrm{d}s$.
 
 On the other hand, the process $ \int_0^t\int_{\mathbb{R}_0}\frac{ \gamma_1(s,z)}{1+\pi_s^*\gamma_1(s,z)}\tilde N^1(\mathrm{d}^-s,\mathrm{d}z)$ is an ${\cal H}_t$-special semimartingale by (\ref{decomp_1}). Moreover, the ${\cal H}_t$-predictable random field $\frac{ \gamma_1(s,z)}{1+\pi_s^*\gamma_1(s,z)}$ is It\^{o} integrable with respect to $\tilde N^1_{{\cal H}}:=N^1-\hat N^1_{{\cal H}}$ by assumptions in Definition \ref{admiss}. Thus, we have $ \int_0^t\int_{\mathbb{R}_0}\frac{ \gamma_1(s,z)}{1+\pi_s^*\gamma_1(s,z)}\tilde N^1(\mathrm{d}^-s,\mathrm{d}z)=\int_0^t\int_{\mathbb{R}_0}\frac{ \gamma_1(s,z)}{1+\pi_s^*\gamma_1(s,z)}\tilde N_{{\cal H}}^1(\mathrm{d}s,\mathrm{d}z)+\int_0^t\int_{\mathbb{R}_0}\frac{ \gamma_1(s,z)}{1+\pi_s^*\gamma_1(s,z)}(\hat N_{{\cal H}}^1-\hat N^1)(\mathrm{d}s,\mathrm{d}z)$. Since the ${\cal H}_t$-predictable bounded variation process $\int_0^t\int_{\mathbb{R}_0}\frac{ \gamma_1(s,z)}{1+\pi_s^*\gamma_1(s,z)}\hat N_{{\cal H}}^1 (\mathrm{d}s,\mathrm{d}z)$ exists (due to Definition \ref{admiss}), we deduce that $\hat N^1_{{\cal H}}$ is absolutely continuous in time by (\ref{decomp_1}) and the uniqueness of the canonical decomposition of a special semimartingale.
 
 Similar conclusion can be given for the ${\cal H}_t$-local martingale $\hat m_1^{u^*,v^*}(t)$, $t\in [0,T]$. In summary, we have the following theorem.

    \begin{theorem}\label{mainth3}
  Suppose $(u^*,v^*)\in {\cal A}_1'\times{\cal A}_2'$ is optimal for Problem \ref{sdg} under the conditions of Theorem \ref{mainth2}. Then we have the following decompositions 
  \begin{equation} \label{semi_decomp}
 \begin{aligned}
&W^i_t=W^i_{\cal H}(t)+\int_0^t\phi_i(s)\mathrm{d}s,  \quad 0\le t\le T,\\
& \hat N^i_{\cal H}(\mathrm{d}t,\mathrm{d}z)=G^i_{\cal H}(t,\mathrm{d}z)\mathrm{d}t,  \quad 0\le t\le T, \quad z\in\mathbb{R}_0,
\end{aligned}
 \end{equation}
 where $W^i_{\cal H}(t)$ is an ${\cal H}_t$-Brownian motion, $\phi_i(t)$ is an ${\cal H}_t$-progressively measurable process satisfying $\int_0^T|\phi_i(t)|\mathrm{d}t<\infty$, and $G^i_{\cal H}(t,\mathrm{d}z)$ is some random transition measure, $i=1,2$. Moreover, by the uniqueness of the canonical decomposition of a special semimartingale, $u^*$ solves the following equations  
   \begin{equation}\label{halfequ_1}
 \begin{aligned}
0
= &\int_0^t   \left(\mu(s,\pi_s^*)-r_s+\frac{\partial}{\partial x}\mu(s,\pi^*_s)\pi^*_s-\sigma^2_s\pi^*_s+\rho\sigma_sb_s\kappa_s^*\right)\mathrm{d}s  +\int_0^t\sigma_s\phi_1(s)\mathrm{d}s -\int_0^t\int_{\mathbb{R}_0}\frac{ \pi_s^*\gamma^2_1(s,z)}{1+\pi_s^*\gamma_1(s,z)} G^1(\mathrm{d} z)\mathrm{d}s \\
&    +\int_0^t\int_{\mathbb{R}_0}\frac{ \gamma_1(s,z)}{1+\pi_s^*\gamma_1(s,z)}\Big[G^1_{\cal H}(s,\mathrm{d}z)-G^1(\mathrm{d}z)\Big]\mathrm{d}s-\int_0^t F_{u^*,v^*}(s-)\mathrm{d}\big\langle(F_{u^*,v^*})^{-1}, m_1^{u^*}\big\rangle^{Q_{u^*,v^*}}_s,\quad 0\le t\le T,
\end{aligned}
 \end{equation}
 and
    \begin{equation}\label{halfequ_2}
 \begin{aligned}
0=&\int_0^t   \left(\lambda_s-a_s+\rho\sigma_sb_s\pi^*_s-b^2_s\kappa_s^*\right)\mathrm{d}s -\int_0^t\rho b_s\phi_1(s)\mathrm{d}s-\int_0^t \sqrt{1-\rho^2}b_s\phi_2(s)\mathrm{d}s -\int_0^t\int_{\mathbb{R}_0}\frac{ \kappa_s^*\gamma^2_2(s,z)}{1-\kappa_s^*\gamma_2(s,z)} G^2(\mathrm{d} z)\mathrm{d}s \\
&-\int_0^t\int_{\mathbb{R}_0}\frac{ \gamma_2(s,z)}{1-\kappa_s^*\gamma_2(s,z)}\left[G^2_{\cal H}(s,\mathrm{d}z)-G^2(\mathrm{d}z)\right]\mathrm{d}s-\int_0^t F_{u^*,v^*}(s-)\mathrm{d}\big\langle(F_{u^*,v^*})^{-1}, m_2^{u^*}\big\rangle^{Q_{u^*,v^*}}_s,\quad 0\le t\le T.
\end{aligned}
 \end{equation}
    \end{theorem}

\begin{remark}
Theorem \ref{mainth3} generalizes the main result (Theorem 4.2) in \cite{Peng16}. However, this is not a trivial generalization of \cite{Peng16}. Without further introducing the new probability measure $Q_{u^*,v^*}$ in Theorem \ref{mainth1} and the predictable version of the Girsanov theorem in Theorem \ref{mainth2}, the model uncertainty and more general utility functions could not be considered here, nor could an extra but important decomposition theorem be obtained\footnote{ This provides an opportunity to turn the abstract expression of the other half controlled process $\varepsilon^v_t$ (see (\ref{theta_sde})) into a  nonanticipative stochastic differential equation (see (\ref{half_equ2})). Thus, we can use the stochastic maximal  principle to solve the whole anticipating stochastic differential game problem (Problem \ref{sdg}) in the next section.} in Theorem \ref{mainth3}. Moreover, all coefficients here are all anticipating processes (i.e., the anticipating environments of the financial market and the insurance market are considered, or namely, the classical SDEs of the risky asset process and the insurer's risk process in \cite{Peng16} are replaced by the anticipating SDEs (\ref{fin-m}) and (\ref{ins-m}), respectively), and a large insurer is considered here (i.e., the mean rate of return $\mu$ on the risky asset is influenced by her investment strategy $\pi$).
\end{remark}

 Further, by Theorem \ref{mainth3} and Propositions \ref{fgen1} and \ref{fgen2}, the dynamic of the ${\cal H}_t$-martingale $\varepsilon^v_t$ (see (\ref{theta_sde})) can be rewritten as  
  \begin{equation} \label{half_equ2}
 \begin{aligned}
 \varepsilon^v_t=1+\varepsilon^v_{t-}\Bigg( & \int_0^t\theta_1(s)\mathrm{d}W_{\cal H}^1(s)+\int_0^t\theta_2(s)\mathrm{d}W_{\cal H}^2(s)\\
 &+\int_0^t\int_{\mathbb{R}_0}\theta_3(s)\tilde N_{\cal H}^1(\mathrm{d}s,\mathrm{d}z)+\int_0^t\int_{\mathbb{R}_0}\theta_4(s)\tilde N^2_{\cal H}(\mathrm{d}s,\mathrm{d}z) \Bigg) ,\quad 0\le t\le T,
\end{aligned}
 \end{equation}
 for $v=(\theta_1,\theta_2,\theta_3,\theta_4)\in{\cal A}_2'$. By the It\^{o} formula for It\^{o} processes (see Theorem \ref{itoforito}), we have
  \begin{equation} \label{etominus1}
 \begin{aligned}
 (\varepsilon_t^{v})^{-1}=&1+\int_0^t(\varepsilon_{s-}^{v})^{-1}(\theta_1(s)^2+\theta_2(s)^2) \mathrm{d}s  -\int_0^t (\varepsilon_{s-}^{v})^{-1}\theta_1(s)\mathrm{d}W^1_{\cal H}(s)-\int_0^t (\varepsilon_{s-}^{v})^{-1}\theta_2(s)\mathrm{d}W^2_{\cal H}(s)\\
 &  -\int_0^t\int_{\mathbb{R}_0}(\varepsilon_{s-}^{v})^{-1}\frac{\theta_3(s)}{1+\theta_3(s)}\tilde N^1_{\cal H}(\mathrm{d}s ,\mathrm{d}z) -\int_0^t\int_{\mathbb{R}_0}(\varepsilon_{s-}^{v})^{-1}\frac{\theta_4(s)}{1+\theta_4(s)}\tilde N^2_{\cal H}(\mathrm{d}s ,\mathrm{d}z)   \\
 & +\int_0^t\int_{\mathbb{R}_0}(\varepsilon_{s-}^{v})^{-1}\frac{\theta_3(s)^2}{1+\theta_3(s)}G^1_{\cal H}(s,\mathrm{d}z)\mathrm{d}s+\int_0^t\int_{\mathbb{R}_0}(\varepsilon_{s-}^{v})^{-1}\frac{\theta_4(s)^2}{1+\theta_4(s)}G^2_{\cal H}(s,\mathrm{d}z)\mathrm{d}s.
 \end{aligned}
 \end{equation}
Moreover, by the Girsanov theorem (see Theorem \ref{girsanov2}), the following terms
   \begin{equation} \label{girsanovinlog}
 \begin{aligned}
 &W^i_{{\cal H},{\cal{Q}}^v}(t)=W^i_{{\cal H}}(t)-\int_0^t\theta_{i}(s)\mathrm{d}s, \quad 0\le t\le T,\\
& \hat N_{{\cal H},{\cal{Q}}^v}^i(\mathrm{d}t,\mathrm{d}z)= (1+\theta_{i+2}(t)  )G_{{\cal H}}^i( t,\mathrm{d}z)\mathrm{d}t,\quad 0\le t\le T,\quad z\in{\mathbb{R}_0},
  \end{aligned}
 \end{equation}
 are $({\cal H}_t,{\cal{Q}}^v)$-Brownian motion and $({\cal H}_t,{\cal{Q}}^v)$-compensator of $N^i(\mathrm{d}t,\mathrm{d}z)$, respectively, $i=1,2$. 
 
 For the optimal pair $(u^*,v^*)$, it is often difficult to obtain a concrete expression of the $({\cal H}_t,{Q_{u^*,v^*}})$-predictable covariation process $\langle (F_{u^*,v^*})^{-1},m_i^{u^*}\rangle^{Q_{u^*,v^*}}_t$, $i=1,2$, in Theorem \ref{mainth3}. However, when the utility function of the logarithmic form, i.e., $U(x)=\ln(x)$, we have $F_{u^*,v^*}(t)=\varepsilon^{v^*}_t$ and $Q_{u^*,v^*}={\cal{Q}}^{v^*}$. By (\ref{equ1ofth1}), (\ref{equ2ofth1}), (\ref{etominus1}) and (\ref{girsanovinlog}), we can rewrite the ${\cal H}_t$-adapted processes $m_1^{u^*}(t)$, $m_2^{u^*}(t)$ and $(\varepsilon^{v^*}_t)^{-1}$ with respect to $W^i_{{\cal H},{\cal{Q}}^{v^*}}(t)$ and $ \tilde N_{{\cal H},{\cal{Q}}^{v^*}}^i(\mathrm{d}t,\mathrm{d}z)$ under the probability measure ${\cal{Q}}^{v^*}$, $i=1,2$. Then the $({\cal H}_t,{{\cal{Q}}^{v^*}})$-predictable covariation process $\langle (\varepsilon^{v^*})^{-1},m_1^{u^*} \rangle^{{\cal{Q}}^{v^*}}_t $ and $\langle (\varepsilon^{v^*})^{-1},m_2^{u^*} \rangle^{{\cal{Q}}^{v^*}}_t $ can be calculated as follows
   \begin{equation} \label{reasonfor1}
 \begin{aligned}
&\big\langle (\varepsilon^{v^*})^{-1},m_1^{u^*} \big\rangle^{{\cal{Q}}^{v^*}}_t
=-\int_{0}^t (\varepsilon_{s-}^{v^*})^{-1}\theta_1^*(s)\sigma_s\mathrm{d}s-\int_0^t\int_{\mathbb{R}_0} \frac{(\varepsilon_{s-}^{v^*})^{-1}\theta^*_3(s) \gamma_1(s,z)}{1+\pi^*_s\gamma_1(s,z)}G_{\cal H}^1(s,\mathrm{d}z)\mathrm{d}s   ,\\
&\big\langle (\varepsilon^{v^*})^{-1},m_2^{u^*} \big\rangle^{{\cal{Q}}^{v^*}}_t
=\int_{0}^t (\varepsilon_{s-}^{v^*})^{-1}\theta_1^*(s)\rho b_s\mathrm{d}s
+\int_{0}^t (\varepsilon_{s-}^{v^*})^{-1}\theta_2^*(s)\sqrt{1-\rho^2} b_s\mathrm{d}s
+\int_0^t\int_{\mathbb{R}_0} \frac{(\varepsilon_{s-}^{v^*})^{-1}\theta^*_4(s) \gamma_2(s,z)}{1-\kappa^*_s\gamma_2(s,z)}G_{\cal H}^2(s,\mathrm{d}z)\mathrm{d}s .
 \end{aligned}
 \end{equation}
We obtain the following corollary by substituting (\ref{reasonfor1}) into (\ref{halfequ_1}) and (\ref{halfequ_2}) in Theorem \ref{mainth3}, respectively.
  
 \begin{corollary}\label{mainth6}
 Assume that $U(x)=\ln(x)$. Suppose $(u^*,v^*)\in {\cal A}_1'\times{\cal A}_2'$ is optimal for Problem \ref{sdg} under the conditions of Theorem \ref{mainth2}. Then $u^*$ solves the following equations
    \begin{equation}\label{halfequ_21}
 \begin{aligned}
0
= & \mu(t,\pi_t^*)-r_t+\frac{\partial}{\partial x}\mu(t,\pi^*_t)\pi^*_t-\sigma^2_t\pi^*_t+\rho\sigma_t b_t\kappa_t^*   +\sigma_t\phi_1(t)+\sigma_t\theta_1^*(t) 
- \int_{\mathbb{R}_0}\frac{ \pi_t^*\gamma^2_1(t,z)}{1+\pi_t^*\gamma_1(t,z)} G^1(\mathrm{d} z)  \\
&    + \int_{\mathbb{R}_0}\frac{ \gamma_1(t,z)}{1+\pi_t^*\gamma_1(t,z)} (1+\theta_3^*(t))G_{{\cal H}}^1(t,\mathrm{d}z) 
-   \int_{\mathbb{R}_0}\frac{ \gamma_1(t,z)}{1+\pi_t^*\gamma_1(t,z)}G^1(\mathrm{d}z) ,\quad 0\le t\le T,
\end{aligned}
 \end{equation}
 and
    \begin{equation}\label{halfequ_22}
 \begin{aligned}
0=& \lambda_t-a_t+\rho\sigma_t b_t\pi^*_t-b^2_t\kappa_t^*
 - \rho b_t\phi_1(t)- \rho b_t\theta^*_1(t)  -  \sqrt{1-\rho^2}b_t\phi_2(t)-  \sqrt{1-\rho^2}b_t\theta_2^*(t)
   \\&- \int_{\mathbb{R}_0}\frac{ \kappa_s^*\gamma^2_2(s,z)}{1-\kappa_s^*\gamma_2(s,z)} G^2(\mathrm{d} z)  - \int_{\mathbb{R}_0}\frac{ \gamma_2(t,z)}{1-\kappa_t^*\gamma_2(t,z)} (1+\theta^*_4(t)) G_{\cal H}^2 (t,\mathrm{d}z) + \int_{\mathbb{R}_0}\frac{ \gamma_2(t,z)}{1-\kappa_t^*\gamma_2(t,z)}G^2( \mathrm{d}z) 
 ,\quad 0\le t\le T.
\end{aligned}
 \end{equation}
 \end{corollary}

 \begin{remark}\label{whycomp}
 If we follow the method in \cite{DiNunno09} and use the ordinary version of the Girsanov theorem (see Theorem \ref{girsanov}), the last term of $\hat m_i^{u^*,v^*}(t)$ in Theorem \ref{mainth2} reduces to $\int_0^t\varepsilon_s^{v^*} \mathrm{d}\big[ (\varepsilon^{v^*})^{-1} ,m_i^{u^*} \big]_s$. The jump term of $\big[ (\varepsilon^{v^*})^{-1} , m_i^{u^*} \big]_t$ is of the form $\int_0^t\int_{\mathbb{R}_0}   (\varepsilon^{v^*}_{s-})^{-1} \varphi^{u^*,v^*}_{i,1}(s,z) N^1_{\cal H}(\mathrm{d}s,\mathrm{d}z)+\int_0^t\int_{\mathbb{R}_0}  (\varepsilon^{v^*}_{s-})^{-1} \varphi^{u^*,v^*}_{i,2}(s,z) N^2_{\cal H}(\mathrm{d}s,\mathrm{d}z)$ for some ${\cal H}_t$-predictable random field  $\varphi^{u^*,v^*}_{i,j}(s,z)$, $j=1,2$, $i=1,2$, which is usually not continuous and contradicted with the absolute continuity assumption in Theorem \ref{mainth2}. Thus, it leads to the triviality of ${\cal A}_2'$. Moreover, the integral $\int_0^t\varepsilon_s^{v^*} \mathrm{d}\big[ (\varepsilon^{v^*})^{-1} ,m_i^{u^*} \big]_s$ looks confusing and is usually not continuous in time as well, $i=1,2$, which could not lead to the decompositions in Theorem \ref{mainth3}.
 \end{remark}

\section{A total characterization of the robust optimal strategy}
\label{tt_section}

 In the previous section, we give the characterization of $u^*$ for the optimal pair $(u^*,v^*)$ by using the maximality of $J(u^*,v^*)$ with respect to $u$. Thus, we obtain the relationship between $u^*$ and $v^*$ (see equations (\ref{halfequ_1}) and (\ref{halfequ_2})). However, we have not used the minimality of $J(u^*,v^*)$ with respect to $v$. Thus, we need the other half characterization of $v^*$.
 
It is very difficult to give a characterization of $v^*$ directly due to the complexity of the other half controlled process $\varepsilon^v_t$ (see the equation (\ref{theta_sde})). Fortunately, under the conditions of Theorem \ref{mainth2}, we get the decompositions $W^i_t=W^i_{\cal H}(t)+\int_0^t\phi_i(s)\mathrm{d}s$ and $ \hat N^i_{\cal H}(\mathrm{d}t,\mathrm{d}z)=G^i_{\cal H}(t,\mathrm{d}z)\mathrm{d}t$, $i=1,2$, with respect to the filtration $\{{\cal H}_t\}$ by Theorem \ref{mainth3}. Thus, we have a better expression of $\varepsilon^v_t$ for $v \in{\cal A}_2'$ (see (\ref{half_equ2})). Moreover, we can also rewrite the dynamic of the wealth process $X^u_t$ (see (\ref{wealthsde})) as follows:
 \begin{equation}\label{wealthsde2}
 \begin{aligned}
\frac{\mathrm{d}X^u_t}{X^u_{t-}}=&\Bigg[r_t+(\mu(t,\pi_t)-r_t)\pi_t+(\lambda_t-a_t)\kappa_t  +(\sigma_t\pi_t-\rho b_t\kappa_t  )\phi_1(t) -\sqrt{1-\rho^2}b_t\kappa_t\phi_2(t)\\
&+  \int_{\mathbb{R}_0}\pi_t\gamma_1(t,z) \left(G_{\cal H}^1(t,\mathrm{d}z)-G^1(\mathrm{d}z)\right)-  \int_{\mathbb{R}_0}\kappa_t\gamma_2(t,z) \left(G_{\cal H}^2(t,\mathrm{d}z)-G^2(\mathrm{d}z)\right) \Bigg]\mathrm{d}t\\
&+\left( \sigma_t\pi_t-\rho b_t\kappa_t   \right)\mathrm{d}W^1_{\cal H}(t)-\sqrt{1-\rho^2}b_t\kappa_t\mathrm{d}W^2_{\cal H}(t)+ \int_{\mathbb{R}_0}\pi_t\gamma_1(t,z)\tilde N^1_{\cal H}(\mathrm{d}t,\mathrm{d}z)- \int_{\mathbb{R}_0}\kappa_t\gamma_2(t,z)\tilde N^2_{\cal H}(\mathrm{d}t,\mathrm{d}z).
\end{aligned}
 \end{equation}

 Since (\ref{half_equ2}) and (\ref{wealthsde2}) can be viewed as the usual SDEs with respect to the ${\cal H}_t$-Brownian motion $W^i_{\cal H}(t)$ and the ${\cal H}_t$-compensated random measure $\tilde N^i_{\cal H}(\mathrm{d}t,\mathrm{d}z)$, $i=1,2$, Problem \ref{sdg} turns to a nonanticipative stochastic differential game problem with respect to the filtration $\{{\cal H}_t\}_{0\le t\le T}$. Thus, we can use the stochastic maximum principle to solve our problem. Notice that since all coefficients of (\ref{half_equ2}) and (\ref{wealthsde2}) are path-dependent stochastic processes, $\varepsilon_t^v$ and $X^u_t$ is not necessary a Markov process. Thus, we can not derive the corresponding HJB equation by the dynamic programming principle.
  
We make the following assumptions before our procedure.
 
 \begin{assumption}\label{assump3}
If $(u^*,v^*)\in {\cal A}_1'\times {\cal A}_2'$ is optimal for Problem \ref{sdg}, then for all bounded $  \beta \in {\cal A}_2'$, there exists some $\delta>0$ such that $ v^*+y\beta \in   {\cal A}_2'$ for all $|y|<\delta$. Moreover, the following family of random variables 
  \begin{equation*}
 \begin{aligned}
\left \{  \frac{\mathrm{d}}{\mathrm{d}y}\varepsilon_T^{v^*+y\beta}U(X_T^{u^* }) \right\}_{y\in(-\delta,\delta)}
 \end{aligned}
 \end{equation*}
is $\mathbb{P}$-uniformly integrable, and the following family of random fields
  \begin{equation*}
 \begin{aligned}
\left \{  \frac{\mathrm{d}}{\mathrm{d}y}\varepsilon^{v^*+y\beta}_tg(s,v^*+y\beta) +\varepsilon_t^{v^*+y\beta}\nabla_v g(s,v^*+y\beta)\beta^{\tau}\right\}_{y\in(-\delta,\delta)}
 \end{aligned}
 \end{equation*}
is $\mathrm{m} \times \mathbb{P}$-uniformly integrable, where $\mathrm{m}$ is the Borel-Lebesgue measure on $[0,T]$, $(\cdot)^\tau$ denotes the transpose of a vector, and $\frac{\mathrm{d}}{\mathrm{d}y}$ means that $\frac{\mathrm{d}}{\mathrm{d}y}\varepsilon_t^{v^*+y\beta}$ exists. 
 \end{assumption}
 
  \begin{assumption}\label{assump3_2}
If $(u^*,v^*)\in {\cal A}_1'\times {\cal A}_2'$ is optimal for Problem \ref{sdg} under the conditions of Theorem \ref{mainth2} and Assumption \ref{assump3}, then for all bounded $ (\alpha , \beta) \in  {\cal A}_1'\times{\cal A}_2'$, we can define $\tilde\psi^{u^*}_1(t):=\frac{\mathrm{d}}{\mathrm{d}y}X_t^{\pi^*+y\alpha_1}{|}_{y=0}$, $\tilde\psi^{u^*}_2(t):=\frac{\mathrm{d}}{\mathrm{d}y}X_t^{\kappa^*+y\alpha_2}{|}_{y=0}$ and $\psi^{v^*}_i(t):=\frac{\mathrm{d}}{\mathrm{d}y}\varepsilon_t^{\theta_i^*+y\beta_i}{|}_{y=0}$ by Assumptions \ref{assump1} and \ref{assump3}, $i=1,2,3,4$. 
 \end{assumption}
 
  \begin{assumption}\label{assump4}
  Let $ v_s=\left(\xi_1 1_{(t,t+h]}(s),\xi_21_{(t,t+h]}(s),\xi_3 1_{(t,t+h]}(s),\xi_41_{(t,t+h]}(s)\right)$, $0\le s\le T$, for fixed $0\le t<t+h\le T$, where the random variable $\xi_i$ is of the form $1_{A_t}$ for any ${\cal H}_t$-measurable set $A_t$, $i=1,2,3,4$. Then $v\in {\cal A}_2'$.
  \end{assumption}
  
Now we define the Hamiltonian $H:[0,T]\times\mathbb{R}\times\mathbb{R}\times\mathbb{R}^2  \times \mathbb{R}^4   \times\mathbb{R}^2\times\mathbb{A}\times\Omega\rightarrow\mathbb{R} $ by (see \cite{Yong99,Oksendal19} for the construction)
  \begin{equation*}
 \begin{aligned}
H(t,x,\varepsilon,u,v,p,q,\omega):=&g(s,v)\varepsilon+x\Bigg[r_t+(\mu(t,\pi)-r_t)\pi+(\lambda_t-a_t)\kappa+(\sigma_t\pi-\rho b_t\kappa)\phi_1(t)-\sqrt{1-\rho^2}b_t\kappa\phi_2(t) \\
&+\int_{\mathbb{R}_0}\pi\gamma_1(t,z)\left(   G^1_{\cal H} (t,\mathrm{d}z)-G^1(\mathrm{d}z)\right)-\int_{\mathbb{R}_0}\kappa\gamma_2(t,z)\left(   G^2_{\cal H} (t,\mathrm{d}z)-G^2(\mathrm{d}z)\right)\Bigg]p_1\\
&+x(\sigma_t\pi-\rho b_t\kappa)q_{11}- x\sqrt{1-\rho^2}b_t\kappa q_{12}+\int_{\mathbb{R}_0}x\pi\gamma_1(t,z)q_{13}(z)G^1_{\cal H}( t,\mathrm{d}z)\\
& -\int_{\mathbb{R}_0}x\kappa\gamma_2(t,z)q_{14}(z)G^2_{\cal H}( t,\mathrm{d}z)   +
\varepsilon\theta_1q_{21}+\varepsilon\theta_2q_{22}+\int_{\mathbb{R}_0}\varepsilon\theta_3q_{23}(z)G^1_{\cal H}(t,\mathrm{d}z)\\
&+\int_{\mathbb{R}_0}\varepsilon\theta_4q_{24}(z)G^2_{\cal H}(t,\mathrm{d}z),
 \end{aligned}
 \end{equation*}
 where $u=(\pi,\kappa)$, $v=(\theta_1,\theta_2,\theta_3,\theta_4) $, $p=\begin{pmatrix}p_{1}\\p_{2}\end{pmatrix}$, $q=\begin{pmatrix}q_{11}&q_{12}&q_{13}&q_{14}\\q_{21}&q_{22}&q_{23}&q_{24}\end{pmatrix}$, and $\mathbb{A}$ represents the set of all the matrices $(q_{ij})_{2\times 4}$ with the elements in the first two columns being real numbers and the elements in the last two columns being functions from $\mathbb{R}_0$ to $\mathbb{R}$ such that $\int_{\mathbb{R}_0} ( |q_{13}(z)|^2+|q_{23}(z)|^2)G_{\cal H}^1(t,\mathrm{d}z)<\infty$ and $\int_{\mathbb{R}_0}(| q_{14}(z)|^2+|q_{24}(z)|^2)G_{\cal H}^2(t,\mathrm{d}z)<\infty$. It is obvious that $H$ is differentiable with respect to $x$ and $\varepsilon$, and Fr\'{e}chet differentiable with respect to $u$ and $v$. The associated BSDE system for the adjoint pair $(p_t,q(t,z))$ is given by (see \cite{Yong99,Oksendal19})
  \begin{eqnarray}\label{adjoint0}
  \left\{ \begin{aligned}\mathrm{d}p_1(t)=&-\frac{\partial H}{\partial x}(t)\mathrm{d}t+q_{11}(t)\mathrm{d}W^1_{\cal H}(t)+q_{12}(t)\mathrm{d}W^2_{\cal H}(t)\\
  &  +\int_{\mathbb{R}_0}q_{13}(t,z)\tilde N_{\cal H}^1(\mathrm{d}t,\mathrm{d}z)+\int_{\mathbb{R}_0}q_{14}(t,z)\tilde N_{\cal H}^2(\mathrm{d}t,\mathrm{d}z), \quad 0\le t\le T,
  \\p_1(T)=&\varepsilon^{v}_TU'(X^{u}_T) ,\end{aligned} \right.   
 \end{eqnarray}
 and
 \begin{eqnarray}\label{adjoint}
  \left\{ \begin{aligned}\mathrm{d}p_2(t)=&-\frac{\partial H}{\partial \varepsilon}(t)\mathrm{d}t+q_{21}(t)\mathrm{d}W^1_{\cal H}(t)+q_{22}(t)\mathrm{d}W^2_{\cal H}(t)\\
  &  +\int_{\mathbb{R}_0}q_{23}(t,z)\tilde N_{\cal H}^1(\mathrm{d}t,\mathrm{d}z)+\int_{\mathbb{R}_0}q_{24}(t,z)\tilde N_{\cal H}^2(\mathrm{d}t,\mathrm{d}z), \quad 0\le t\le T,
  \\p_2(T)=&U(X^{u}_T) ,\end{aligned} \right.   
 \end{eqnarray}
 where $p_{i}(t)$ is an ${\cal H}_t$-special semimartingale, and $q_{ij}(t)$ is an ${\cal H}_t$-predictable process with the following integrability
  \begin{eqnarray*}
&\int_0^T  \left[\big|\frac{\partial H}{\partial x}(t)\big|+\big|\frac{\partial H}{\partial \varepsilon}(t)\big|+|q_{i1}(t)|^2+|q_{i2}(t)|^2\right]\mathrm{d}t+\int_0^T\int_{\mathbb{R}_0}  \left[|q_{i3}(t,z)|^2G^1_{\cal H}(t,\mathrm{d}z)+|q_{i4}(t,z)|^2 G^2_{\cal H}(t,\mathrm{d}z)\right] \mathrm{d}t <\infty,
  \end{eqnarray*}
$j=1,2,3,4$, $ i=1,2$. Here and in the following the abbreviated notation $  H(t):= H(t,X_t^u,\varepsilon^v_t,u_t,v_t,p_t,q_t,\omega) $, etc., are taken.

We give a necessary maximum principle to characterize the optimal pair $(u^*,v^*)\in {\cal A}_1'\times{\cal A}_2'$.

   \begin{theorem}\label{mainth4}
Suppose $(u^*,v^*)\in {\cal A}_1'\times{\cal A}_2'$ is optimal for Problem \ref{sdg} under the conditions of Theorem \ref{mainth2} and Assumptions \ref{assump3}-\ref{assump4}, and $(p^*,q^*)$ is the associated adjoint pair satisfying BSDEs (\ref{adjoint0}) and (\ref{adjoint}). Then $(u^*,v^*)$ solves the following equations (the Hamiltonian system)
   \begin{equation} \label{half_equ3_0}
 \begin{aligned}
\nabla _u H^*(t)&=\left(\frac{\partial H^*}{\partial \pi}(t), \frac{\partial H^*}{\partial \kappa}(t)\right)=(0,0),\quad 0\le t\le T,
\end{aligned}
 \end{equation}
and
  \begin{equation} \label{half_equ3}
 \begin{aligned}
\nabla _v H^*(t)&=\left(\frac{\partial H^*}{\partial \theta_1}(t), \frac{\partial H^*}{\partial \theta_2}(t),\frac{\partial H^*}{\partial \theta_3}(t),\frac{\partial H^*}{\partial \theta_4}(t)\right)=(0,0,0,0),\quad 0\le t\le T,
\end{aligned}
 \end{equation}
 given the following integrability conditions 
   \begin{equation*} 
 \begin{aligned}
\mathbb{E}\Bigg\{&\int_0^T \left(\tilde\psi^{u^*}_1(t-)^2+\tilde\psi^{u^*}_2(t-)^2\right)\mathrm{d}\big[p^*_{1}\big]_t +\int_0^T p^*_1(t-)^2\mathrm{d}\big[\tilde\psi^{u^*}_1\big]_t +\int_0^T p^*_1(t-)^2\mathrm{d}\big[\tilde\psi^{u^*}_2\big]_t \\
&+\int_0^T\int_{\mathbb{R}_0}\big|  \left(X^{u^*}_{t-} \gamma_1(t,z)\alpha_1(t)+\tilde\psi_1^{u^*}(t-)\pi^*_t\gamma_1(t,z)\right)q_{13}^*(t,z) \big |^2G^{1}_{\cal H}(t,\mathrm{d}z)\mathrm{d}t \\
&  +\int_0^T\int_{\mathbb{R}_0}| \tilde\psi_1^{u^*}(t-)\kappa^*_t\gamma_2(t,z)  q^*_{14}(t,z)|^2G^{2}_{\cal H}(t,\mathrm{d}z)\mathrm{d}t     +\int_0^T\int_{\mathbb{R}_0}| \tilde\psi_2^{u^*}(t-)\pi^*_t\gamma_1(t,z)  q^*_{13}(t,z)|^2G^{1}_{\cal H}(t,\mathrm{d}z)\mathrm{d}t\\
&+\int_0^T\int_{\mathbb{R}_0}\big|  \left(X^{u^*}_{t-} \gamma_2(t,z)\alpha_2(t)+\tilde\psi_2^{u^*}(t-)\kappa^*_t\gamma_2(t,z)\right)q_{14}^*(t,z) \big |^2G^{2}_{\cal H}(t,\mathrm{d}z)\mathrm{d}t \Bigg\} <\infty, 
\end{aligned}
 \end{equation*}
 and
  \begin{equation*} 
 \begin{aligned}
\mathbb{E}\Bigg\{&\int_0^T \psi^{v^*}_i(t-)^2\mathrm{d}\big[p^*_2\big]_t +\int_0^T p^*_2(t-)^2\mathrm{d}\big[\psi^{v^*}_i\big]_t +\int_0^T\int_{\mathbb{R}_0}|\psi_i^{v^*}(t-)  \theta^*_j(t) q^*_{2j}(t,z)|^2G^{j-2}_{\cal H}(t,\mathrm{d}z)\mathrm{d}t   \\&+\int_0^T\int_{\mathbb{R}_0}| \varepsilon^{v^*}_{t-}\beta_j(t)  q^*_{2j}(t,z)|^2G^{j-2}_{\cal H}(t,\mathrm{d}z)\mathrm{d}t \Bigg\} <\infty, 
\end{aligned}
 \end{equation*}
  for all bounded $(\alpha,\beta)\in{\cal A}_1'\times{\cal A}_2'$, $i=1,2,3,4$, and $j=3,4$. Here, $H^*(t):=H(t,X^{u^*}_t,\varepsilon^{v^*}_t,u^*_t,v^*_t,p^*_t,q^*_t,\omega)$, etc.
\end{theorem}
\begin{proof}
 Suppose that the pair $(u^*,v^*)\in {\cal A}_1'\times{\cal A}_2'$ is optimal. Then for any bounded $(\beta_1,0) \in {\cal A}_2' $ and $|y|<\delta$, we have $J(u^*,v^*+y(\beta_1,0,0,0))\ge J(u^*,v^*) $, which implies that $y=0$ is a minimum point of the function $y\mapsto J(u^*,v^*+y(\beta_1,0,0,0))$. By Assumptions \ref{assump3} and \ref{assump3_2}, It\^{o} formula for It\^{o} integrals (see \cite{He92}), and Remark \ref{integral_ito_rm1}, we have 
  \begin{equation} 
 \begin{aligned}
\frac{\mathrm{d}}{\mathrm{d}y}J(u^*,v^*+y(\beta_1,0,0,0)){|}_{y=0}&=\mathbb{E}\Bigg[ \psi^{v^*}_1(T)U(X_T^{u^*})+\int_0^T \psi^{v^*}_1(s)g(s,v^*_s)\mathrm{d}s+\int_0^T\varepsilon^{v^*}_s\frac{\partial}{\partial \theta_1}g(s,v^*_s)\beta_1(s)\mathrm{d}s
\Bigg]\\
&=\mathbb{E}\Bigg[ \psi^{v^*}_1(T)p^*_{2}(T)+\int_0^T \psi^{v^*}_1(s)g(s,v^*_s)\mathrm{d}s+\int_0^T\varepsilon^{v^*}_s\frac{\partial}{\partial \theta_1}g(s,v^*_s)\beta_1(s)\mathrm{d}s
\Bigg]\\
&=\mathbb{E}\Bigg[ \int_0^T\psi^{v^*}_1(s-)\mathrm{d}p^*_2(s)+\int_0^Tp^*_2(s-)\mathrm{d}\psi^{v^*}_1(s)+\big[p^*_2,\psi^{v^*}_1\big]_T\\
&\hspace{2em}\  +\int_0^T \psi^{v^*}_1(s)g(s,v^*_s)\mathrm{d}s+\int_0^T\varepsilon^{v^*}_s\frac{\partial}{\partial \theta_1}g(s,v^*_s)\beta_1(s)\mathrm{d}s
\Bigg]
 \\
&=\mathbb{E}\Bigg[  \int_0^Tq^*_{21}(s)\varepsilon^{v^*}_{s}\beta_1(s)\mathrm{d}s+\int_0^T\varepsilon^{v^*}_s\frac{\partial}{\partial \theta_1}g(s,v^*_s)\beta_1(s)\mathrm{d}s\Bigg]\\
&=\mathbb{E}\Bigg[  \int_0^T \frac{\partial H^*}{\partial \theta_1}(s) \beta_1(s)\mathrm{d}s\Bigg]\\
&=0.
 \end{aligned}
 \end{equation}
By Assumption \ref{assump4} and the same procedure in Theorem \ref{mainth1}, we can deduce that $ \frac{\partial H^*}{\partial \theta_1}(t) =0$, $t\in[0,T]$. By similar arguments, we can conclude that $ \frac{\partial H^*}{\partial \theta_i}(t) =0$, $t\in[0,T]$, $i=2,3,4$, and also $\nabla_u H^*(t)=(0,0)$, $t\in[0,T]$.
 
\end{proof}

\begin{remark}
The integrability in Theorem \ref{mainth4} can be weakened using the localization technique, i.e., choosing a sequence of ${\cal H}_t$-stopping times such that It\^{o} integrals of the stopped processes are ${\cal H}_t$-square-integrable martingales. We refer to \cite{Oksendal19} for more details.
\end{remark}

Combining Theorem \ref{mainth4} with the conclusion in Section \ref{half_section}, we give the total characterization of the optimal pair $(u^*,v^*)$ as the following theorem.

   \begin{theorem}\label{mainth5}
Suppose $(u^*,v^*)\in {\cal A}_1'\times{\cal A}_2'$ is optimal for Problem \ref{sdg} (with the associated pair $(p^*,q^*)$  satisfying BSDEs (\ref{adjoint0}) and (\ref{adjoint})) under the conditions of Theorem \ref{mainth4}. Then $(u^*,v^*)$ solves equations (\ref{halfequ_1}), (\ref{halfequ_2}), (\ref{half_equ3_0}) and (\ref{half_equ3}). 
 \end{theorem}

\begin{remark}
In fact, equations (\ref{halfequ_1}), (\ref{halfequ_2}) and (\ref{half_equ3}) are enough to obtain the optimal pair $(u^*,v^*)$. This combined method (rather than only using the Hamiltonian system  (\ref{half_equ3_0})-(\ref{half_equ3})) can always give a better characterization of $(u^*,v^*)$. Moreover, when the mean rate of return $\mu$ is dependent on $\pi^*$, i.e., a large insurer is consider, it is very hard to obtain the solution $(u^*,v^*)$ by just using (\ref{half_equ3_0})-(\ref{half_equ3}) since the dynamic of $X^u$ is not homogeneous in this situation (see Remark \ref{explain_log_half} in Section \ref{sec_example2}), while it is not by combining the Hamiltonian system with (\ref{halfequ_1}) and (\ref{halfequ_2}). We will give examples in Section \ref{sec_example1} to illustrate this when the logarithmic utility is considered. However,  since (\ref{halfequ_1}) and (\ref{halfequ_2}) are complicated for general utility, we turn to the Hamiltonian system (\ref{half_equ3_0})-(\ref{half_equ3}) for the optimal pair $(u^*,v^*)$ in general cases, which will be illustrated in Section \ref{sec_example2}.
\end{remark}

\section{The small insurer case: maximum principle}
 \label{sec_example2}

  \subsection{Without jumps}
  \label{sec_exampleg_sub1}
  
Suppose that $G^1(\mathrm{d}z)=G^2(\mathrm{d}z)=0$, i.e., there are no jumps in the risky asset process and the insurer's risk process. Assume that the mean rate of return $\mu(t,x)=\mu_0(t)+\varrho_tx$ for some ${\cal G}^1_t$-adapted c\`{a}gl\`{a}d processes $\mu_0(t)$ and $\varrho_t$ with $0\le \varrho_t<\frac{1}{2}\sigma_t^2$, and $b\ge \epsilon>0$ for some positive constant $\epsilon$. Put $\iota_t=\frac{\mu_0(t)-r_t}{\sigma_t}$, $\tilde\sigma_t=\sigma_t-\frac{2\varrho_t}{\sigma_t}$, $\tilde \phi_1(t)=\iota_t+\phi_1(t)$, and $\tilde \phi_2(t)=\frac{\lambda_t-a_t+\rho b_t\iota_t}{\sqrt{1-\rho^2}b_t}-\phi_2(t)$. Assume further the penalty function $g$ is of the quadratic form, i.e., $g(s,v)=g(v)=\frac{1}{2}(\theta_1^2+\theta_2^2)$. Then we have by the Girsanov theorem (see Theorem \ref{girsanov2}) that
     \begin{equation*} 
 \begin{aligned}
\mathbb{E}\left[  \int_0^T\varepsilon_s^{v^*}  g(v^*_s)\mathrm{d}s \right]
= \mathbb{E}_{{\cal{Q}}^{v^*}}\left[\int_0^Tg(v^*_s)\mathrm{d}s   \right]
&=\mathbb{E}_{{\cal{Q}}^{v^*}}\left[  \int_0^T\theta_1^*(s)\mathrm{d}W_{\cal H}^1(s)+   \int_0^T\theta_2^*(s)\mathrm{d}W_{\cal H}^2(s)  -\ln \varepsilon^{v^*}_T  \right]\\
&=\mathbb{E}_{{\cal{Q}}^{v^*}}\left[  \int_0^T\left(\theta_1^*(s)^2+\theta_2^*(s)^2\right)\mathrm{d}s -\ln \varepsilon^{v^*}_T  \right]\\
&=2\mathbb{E}_{{\cal{Q}}^{v^*}}\left[\int_0^Tg(v^*_s)\mathrm{d}s   \right]-\mathbb{E}_{{\cal{Q}}^{v^*}}\left[  \ln \varepsilon^{v^*}_T  \right],
\end{aligned}
 \end{equation*}
 which implies that 
 \begin{equation}\label{part1ofJ}
 \mathbb{E}\left[  \int_0^T\varepsilon_s^{v^*}  g(v^*_s)\mathrm{d}s \right]
= \mathbb{E}_{{\cal{Q}}^{v^*}}\left[  \ln \varepsilon^{v^*}_T  \right]= \mathbb{E}\left[\varepsilon^{v^*}_T  \ln \varepsilon^{v^*}_T  \right].
\end{equation} 

We make the following assumption before our procedure.

\begin{assumption}\label{extraa1}
Suppose the coefficients satisfy the following integrability
\begin{equation}
\int_0^T\left(|\tilde\phi_1(t)|^2+|\tilde\phi_2(t)|^2\right)\mathrm{d}t<\infty.
\end{equation}
\end{assumption}

By the Hamiltonian system (\ref{half_equ3}) in Theorem \ref{mainth4}, we have
 \begin{eqnarray} 
\nabla_vH^*(t)=\left(\left(\theta_1^*(t)+q_{21}^*(t)\right)\varepsilon^{v^*}_t,\left(\theta_2^*(t)+q_{22}^*(t)\right)\varepsilon^{v^*}_t   \right)=(0,0),
 \end{eqnarray}
which implies that
  \begin{equation}\label{logconti_v}
  \begin{aligned}
  &\theta_1^*(t)+q_{21}^*(t)=0 ,
  \\&\theta_2^*(t)+q_{22}^*(t)=0. 
  \end{aligned}
 \end{equation}
 Substituting (\ref{logconti_v}) into the adjoint BSDE (\ref{adjoint}) with respect to $p_2^*(t)$ we have
 \begin{eqnarray}\label{equeg1_39}
  \left\{ \begin{aligned}& \mathrm{d}p_2^*(t)=\frac{\theta_1^*(t)^2+\theta_2^*(t)^2}{2} \mathrm{d}t -\theta_1^*(t)\mathrm{d}W^1_{\cal H}(t)-\theta_2^*(t)\mathrm{d}W^2_{\cal H}(t), \quad 0\le t\le T,
  \\&p_2^*(T)=U( X^{u^*}_T).\end{aligned} \right.   
 \end{eqnarray}
The SDE (\ref{half_equ2}) of $\varepsilon^{v^*}_t$ combined with Theorem \ref{novikov} implies that
    \begin{equation} \label{equeg1_40}
  \begin{aligned}
  \mathrm{d} \ln\varepsilon^{v^*}_t =-\frac{\theta_1^*(t)^2+\theta_2^*(t)^2}{2}\mathrm{d}t+  \theta_1^*(t)\mathrm{d}W_{\cal H}^1(t)+\theta_2^*(t)\mathrm{d}W_{\cal H}^2(t).
   \end{aligned}
 \end{equation}
 By comparing (\ref{equeg1_39}) with (\ref{equeg1_40}), the solution of the BSDE (\ref{equeg1_39}) can be expressed as
     \begin{equation}  \label{equeg1_41}
  \begin{aligned}
 p_2^*(t)=p^*_2(0)-\ln \varepsilon^{v^*}_t.
 \end{aligned}
 \end{equation}
Denote the ${\cal H}_0$-measurable random variable $p^*_2(0)$ by $c^*_2$. Substituting the terminal condition in (\ref{equeg1_39}), i.e., $p_2^*(T)=U(X^{u^*}_T)$, into (\ref{equeg1_41}) with $t=T$, we have
     \begin{equation}  \label{equeg1_42fb}
  \begin{aligned}
 \ln \varepsilon^{v^*}_T+U(X_T^{u^*} )=c^*_2.
 \end{aligned}
 \end{equation} 

Since $\varepsilon^{v^*}_t$ is an ${\cal H}_t$-martingale, we have $\varepsilon^{v^*}_t=\mathbb{E}\left[ \varepsilon^{v^*}_T  |{\cal H}_t\right]=\mathbb{E}\left[ e^{c^*_2-U(X^{u^*}_T)}   \big{|}{\cal H}_t\right]$ by (\ref{equeg1_42fb}). Using that $\varepsilon^{v^*}_0=1$ we obtain
 \begin{equation}  \label{anequaboutc2}
  \begin{aligned}
  e^{c^*_2}= \left(\mathbb{E}\left[  e^{-U(X^{u^*}_T)}  \big|{\cal H}_0\right]\right)^{-1}.
   \end{aligned}
 \end{equation} 
Thus, we can give the expression of $\varepsilon^{v^*}_t$ as follows
 \begin{equation}  \label{anequaboutc22}
  \begin{aligned}
  \varepsilon^{v^*}_t=\mathbb{E}\left[ \left(\mathbb{E}\left[  e^{-U(X^{u^*}_T)}\big |{\cal H}_0\right] e^{U(X^{u^*}_T)} \right)^{-1} \big{|}{\cal H}_t\right].
   \end{aligned}
 \end{equation} 
Moreover, substituting (\ref{anequaboutc2}) into (\ref{equeg1_42fb}) we obtain
\begin{equation}  \label{anequaboutc222}
  \begin{aligned}
  \varepsilon^{v^*}_T=\left(\mathbb{E}\left[  e^{-U(X^{u^*}_T)}  \big|{\cal H}_0\right] e^{U(X^{u^*}_T)}\right)^{-1}.
   \end{aligned}
 \end{equation}

On the other hand, by the Hamiltonian system (\ref{half_equ3_0}) in Theorem \ref{mainth4}, we have
 \begin{equation}\begin{aligned} 
\nabla_uH^*(t)=\Bigg( & X^{u^*}_t\left(\frac{\partial \mu}{\partial x}(t,\pi^*_t)+(\mu(t,\pi^*_t)-r_t)+\sigma_t\phi^*_1(t)  \right )p_1^*(t) +X^{u^*}_t\sigma_tq^*_{11}(t)  , \\
&   X^{u^*}_t\left(\lambda_t-a_t-\rho b_t\phi_1(t)-\sqrt{1-\rho^2}b_t\phi_2(t)  \right )p_1^*(t) -X^{u^*}_t\rho b_tq_{11}^*(t) -X^{u^*}_t\sqrt{1-\rho^2}b_tq_{12}^*(t) \Bigg)=(0,0),
\end{aligned}  \end{equation} 
which implies that
  \begin{equation}\label{generalconti_v}
  \begin{aligned}
  &  \left(\frac{\partial \mu}{\partial x}(t,\pi^*_t)+(\mu(t,\pi^*_t)-r_t)+\sigma_t\phi^*_1(t)  \right )p_1^*(t) + \sigma_tq^*_{11}(t) =0 ,
  \\&  \left(\lambda_t-a_t-\rho b_t\phi_1(t)-\sqrt{1-\rho^2}b_t\phi_2(t)  \right )p_1^*(t) - \rho b_tq_{11}^*(t) - \sqrt{1-\rho^2}b_tq_{12}^*(t) =0. 
  \end{aligned}
 \end{equation}
 Substituting (\ref{generalconti_v}) into the adjoint BSDE (\ref{adjoint0}) with respect to $p_1^*(t)$ yields
 \begin{eqnarray}\label{idontknow0}
   \left\{ \begin{aligned}&   \mathrm{d}p_1^*(t)=-\left[r_t-\frac{(\sigma_t-\tilde\sigma_t)\sigma_t\pi^*_t}{2} \right]p_1^*(t)\mathrm{d}t-\left[\frac{\sigma_t-\tilde\sigma_t}{2}(1+\pi^*_t)+\tilde\phi_1(t)\right]p_1^*(t)\mathrm{d}W^1_{\cal H}(t)\\
   &\hspace{3em} \ \ +\left[  \tilde\phi_2(t)+\frac{\rho (\sigma_t-\tilde\sigma_t)  (1+\pi^*_t)}{2\sqrt{1-\rho^2} }  \right ]p_1^*(t)\mathrm{d}W^2_{\cal H}(t), \quad 0\le t\le T,
  \\&  p_1^*(T)=\varepsilon^{v^*}_TU'(X_T^{u^*}).\end{aligned} \right.  
 \end{eqnarray}
Suppose $\varrho_t\equiv 0$, i.e., $\sigma_t-\tilde\sigma_t\equiv0$, then (\ref{idontknow0}) degenerates to
  \begin{eqnarray}\label{idontknow}
   \left\{ \begin{aligned}&   \mathrm{d}p_1^*(t)=- r_t p_1^*(t)\mathrm{d}t- \tilde\phi_1(t) p_1^*(t)\mathrm{d}W^1_{\cal H}(t)+  \tilde\phi_2(t) p_1^*(t)\mathrm{d}W^2_{\cal H}(t), \quad 0\le t\le T,
  \\&  p_1^*(T)=\varepsilon^{v^*}_TU'(X_T^{u^*}),\end{aligned} \right.  
 \end{eqnarray}
which implies that all coefficients in (\ref{idontknow}) are independent of $\pi^*$. Then by Theorem \ref{novikov}, the unique solution of (\ref{idontknow}) is given by
   \begin{equation}\label{ireallydontknow} 
  \begin{aligned}
 p_1^*(t)=c^*_1 \Pi^*(0,t),
  \end{aligned}
 \end{equation}
 where $c_1^*:=p^*_1(0)$ is an ${\cal H}_0$-measurable random variable, and $\Pi^*(0,t)$, $t\in[0,T]$, is defined as
    \begin{equation} \label{definitionofPi}
  \begin{aligned}
\Pi^*(0,t):= \exp\Bigg\{ & -\int_{0}^{t} r_s  \mathrm{d}s
  -\int_{0}^{t} \tilde\phi_1(s)  \mathrm{d}W^1_{\cal H}(s)
  +  \int_{0}^{t}   \tilde\phi_2(s)\mathrm{d}W^2_{\cal H}(s)  -\frac{1}{2}\int_{0}^{t}  \left(  \tilde\phi_1(s) ^2+   \tilde\phi_2(s) ^2   \right)\mathrm{d}s \Bigg\}.
  \end{aligned}
 \end{equation}
 Since $U(x)$ has a strictly decreasing derivative $U'(x)$, we can denote the inverse function of $U'(x)$ by $I(x)$. Substituting (\ref{ireallydontknow}) into (\ref{idontknow}) with $t=T$ we have 
    \begin{equation} \label{dffunctionofu}
  \begin{aligned}
 X^{u^*}_T=I\left(\frac{c_1^*\Pi^*(0,T)}{\varepsilon^{v^*}_T}\right).
  \end{aligned}
 \end{equation}
 Combining (\ref{dffunctionofu}) with (\ref{anequaboutc222}) we have
  \begin{equation} \label{dffunctionofu222}
  \begin{aligned}
 X^{u^*}_T=I\left(  c_3^*\Pi^*(0,T) e^{U(X^{u^*}_T)} \right),
  \end{aligned}
 \end{equation}
 where $c_3^*:=c^*_1\mathbb{E}\left[ e^{-U(X^{u^*}_T)}  \big|{\cal H}_0\right]$ is also an ${\cal H}_0$-measurable random variable. Since $I$ is also strictly decreasing, there is a unique fixed point $x^*=\tilde I(y)$ of the equation $x=I(ye^{U(x)})$ for every $y>0$. We also call $\tilde I(y)$, $y>0$, the fixed point function of $I$ (see \cite{Mejia05}). Thus, the solution of (\ref{dffunctionofu222}) is given by
   \begin{equation} \label{dffunctionofu2222}
  \begin{aligned}
 X^{u^*}_T=\tilde I\left(  c_3^*\Pi^*(0,T)   \right).
  \end{aligned}
 \end{equation}

 \begin{remark}
 Here, we introduce the new ${\cal H}_0$-measurable random variable $c^*_3$ and use the technique of the fixed point function to separate $\varepsilon^{v^*}_T$ from the terminal condition of $X_T^{u^*}$, which leads to the non-nested linear BSDE (\ref{final_gex_1}) of $X_t^{u^*}$ below (while the traditional method might lead to a nested linear BSDE, see \cite{Peng21,Oksendal19}). Thus, we could solve the BSDE by traditional methods and obtain the formula for $X_t^{u^*}$ under mild conditions (see (\ref{linearbsde_bdg_2}) below).
 \end{remark}

Put $  z^*_t= (  z^*_1(t),  z^*_2(t))=\left(  (\sigma_t\pi^*_t-\rho b_t\kappa^*_t)X^{u^*}_t, -\sqrt{1-\rho^2}b_t\kappa^*_t X^{u^*}_t \right)$. Then we have
  \begin{equation}\label{bsde_barzt1}
   \begin{aligned}
 &\pi_t^*=\frac{  z_1^*(t)}{\sigma_tX^{u^*}_t}-\frac{\rho  z^*_2(t)}{\sqrt{1-\rho^2}\sigma_tX^{u^*}_t},
  \\&\kappa_t^*=-\frac{  z^*_2(t)}{\sqrt{1-\rho^2}b_tX^{u^*}_t}. 
\end{aligned} 
 \end{equation}
The SDE (\ref{wealthsde2}) of $X_t^{u^*}$ combined with (\ref{bsde_barzt1}) leads to the following linear BSDE
  \begin{eqnarray}\label{final_gex_1}
  \left\{ \begin{aligned}&\mathrm{d}X^{u^*}_t=-f_{\text{L}}(t,X^{u^*}_t,z^*_t,\omega)\mathrm{d}t+z^*_t\mathrm{d}W_{\cal H}(t), \quad 0\le t\le T,
  \\& X^{u^*}_T=\tilde I\left(  c_3^*\Pi^*(0,T)   \right),\end{aligned} \right.   
 \end{eqnarray}
where $W_{\cal H}(t)=\left(W_{\cal H}^1(t), W_{\cal H}^2(t)\right)^\tau$, and the generator (or the driver) $f_{\text{L}}:[0,T]\times \mathbb{R}\times \mathbb{R}^2\times\Omega\rightarrow \mathbb{R}$  is given by
 \begin{equation*} 
\begin{aligned}
f_{\text{L}}(t,x,z,\omega)=-r_tx-\tilde\phi_1(t)z_1  +\tilde\phi_2(t)z_2.
\end{aligned} 
 \end{equation*}
 
 \begin{remark} \label{explain_log_half}
If $\varrho_t\neq 0$, the terminal condition in BSDE (\ref{final_gex_1}) will depend on $z^*_t$  by (\ref{idontknow0}), which makes the BSDE (\ref{final_gex_1}) irregular and very hard to solve. The reason is that the SDE (\ref{wealthsde2}) of $X^{u^*}$ is not homogeneous in this situation. However, once the relationship (equations (\ref{halfequ_1}) and (\ref{halfequ_2}) in Theorem \ref{mainth3}) of $u^* $ and $v^*$ is solved, we can then overcome this situation by a combined method in Section \ref{sec_example1}.
 \end{remark}
 
By the It\^{o} formula for It\^{o} integrals, we have
 \begin{equation}\label{solutionoflbsde}
 \begin{aligned}
\mathrm{d}\left(\Pi^*(0,t)X^{u^*}_t\right)&=\Pi^*(0,t)\mathrm{d}X^{u^*}_t+X^{u^*}_{t}\mathrm{d}\Pi^*(0,t)+\mathrm{d}\big\langle X^{u^*},\Pi^*(0,\cdot)\big\rangle_t\\
&=\Pi^*(0,t) \left( z^*_1(t)-\tilde\phi_1(t)X^{u^*}_t  \right)\mathrm{d}W^1_{\cal H}(t) \\&+\Pi^*(0,t) \left( z^*_2(t)+\tilde\phi_2(t)X^{u^*}_t  \right)\mathrm{d}W^2_{\cal H}(t) .
\end{aligned}
\end{equation}
Suppose the following integrability condition holds
 \begin{equation}\label{linearbsde_bdg}
 \begin{aligned}
\mathbb{E}\Bigg[\left(\int_0^T  \Pi^*(0,t)^2 \left(z^*_1(t)-\tilde\phi_1(t)X^{u^*}_t\right)^2\mathrm{d}t \right)^{\frac{1}{2}} + \left(\int_0^T  \Pi^*(0,t)^2 \left(z^*_2(t)+\tilde\phi_2(t)X^{u^*}_t\right)^2\mathrm{d}t \right)^{\frac{1}{2}}+ (X^{u^*}_T)^2 \Bigg]<\infty.
\end{aligned}
\end{equation}
Then by the Burkholder-Davis-Gundy inequality (see \cite{Karatzas91}), $\Pi^*(0,t)X^{u^*}_t$, $t\in[0,T]$, is an ${\cal H}_t$-martingale. Thus we have
 \begin{equation}\label{linearbsde_bdg_2}
 \begin{aligned}
 X^{u^*}_t= \mathbb{E} \left[   \Pi^*(t,T)\tilde I\left(  c_3^*\Pi^*(0,T)   \right) | {\cal H}_t\right],
\end{aligned}
\end{equation}
where $\Pi^*(t,T):=\Pi^*(0,T)/\Pi^*(0,t)$. By (\ref{linearbsde_bdg_2}) and the initial value condition $X^{u^*}_0=X_0$, the ${\cal H}_0$-random variable $c^*_3$ can be (implicitly) determined by
 \begin{equation} \label{c3determined}
 \begin{aligned}
X_0=\mathbb{E} \left[   \Pi^*(0,T)\tilde I\left(  c_3^*\Pi^*(0,T)   \right) | {\cal H}_0\right].
\end{aligned}
\end{equation}

\begin{remark}
Here we could obtain the formula for $X^{u^*}_t$ without the traditional assumptions for the existence and the uniqueness of the solution to the linear BSDE (\ref{final_gex_1}) (see \cite{Pham09,Peng21}). The reason is that there is already a priori hypothesis for the existence of $X^{u^*}_t$ in the necessary maximum principle (see Theorem \ref{mainth4}), and the formula for $X^{u^*}_t$ is only based on the It\^{o} formula and the integrability condition (\ref{linearbsde_bdg}). Moreover, the existence and the uniqueness of the solution to the BSDE (\ref{final_gex_1}) need more assumptions (such as the assumption for the filtration $\{{\cal H}_t\}_{0\le t\le T}$) below to obtain the existence and the uniqueness of the solution to the BSDE (\ref{final_gex_1}). While those assumptions are very strong and can conversely ensures the condition (\ref{linearbsde_bdg}) ( see \cite{Pham09}).
\end{remark}

Moreover, substituting (\ref{dffunctionofu2222}) into (\ref{anequaboutc22}) and (\ref{anequaboutc222}) we obtain
 \begin{equation}  \label{anequaboutc22_final}
  \begin{aligned}
  \varepsilon^{v^*}_t=\mathbb{E}\left[ \left(\mathbb{E}\left[   e^{ -U\left(\tilde I\left(  c_3^*\Pi^*(0,T)   \right) \right )   }\big |{\cal H}_0\right]   e^{ U\left(\tilde I\left(  c_3^*\Pi^*(0,T)   \right) \right )   }  \right)^{-1} \big{|}{\cal H}_t\right],
   \end{aligned}
 \end{equation} 
and
\begin{equation}  \label{anequaboutc222_final}
  \begin{aligned}
  \varepsilon^{v^*}_T=\left(\mathbb{E}\left[   e^{ -U\left(\tilde I\left(  c_3^*\Pi^*(0,T)   \right) \right )   }\big |{\cal H}_0\right]   e^{ U\left(\tilde I\left(  c_3^*\Pi^*(0,T)   \right) \right )   }  \right)^{-1}.
   \end{aligned}
 \end{equation}

Further, if the filtration $\{{\cal H}_t\}_{0\le t\le T}$ is the augmentation of the natural filtration of $W^1_{\cal H}(t)$ and $W^2_{\cal H}(t)$, which was also assumed in \cite{Biagini05} in the continuous case, then $c_3^*$ is a constant since ${\cal H}_0$ is generated by the trivial $\sigma$-algebra and all $\mathbb{P}$-negligible sets. Suppose that $ \mathbb{E} \tilde I(c_3^*(\Pi^*(0,T)))^2 <\infty$, and $r$, $\tilde \phi_1$ and $\tilde\phi_2$ are bounded. Then by \cite[Theorem 4.8]{Oksendal19}, the linear BSDE (\ref{final_gex_1}) has a unique strong solution $(X^{u^*},z^*_1,z^*_2)$, i.e., $X^{u^*}_t$ is a continuous ${\cal H}_t$-adapted process with $\mathbb{E}\left[\sup_{0\le t\le T}|X^{u^*}_t|^2 \right]<\infty$, $z^*_i(t)$ is an ${\cal H}_t$-progressively measurable process with $\mathbb{E}\left[ \int_0^T|z^*_i(t)|^2\mathrm{d}t \right]<\infty$, $i=1,2$, and $(X^{u^*},z^*_1,z^*_2)$ satisfies the BSDE (\ref{final_gex_1}). 

\begin{remark}
Under mild conditions, we can obtain the formulae for $  z^*_1(t) $ and $ z^*_2(t))$ as follows (see \cite[Proposition 3.5.1]{Delong13})
 \begin{eqnarray}\label{final_ex_bmoofz1}
  \begin{aligned}& z^*_1(t)=D^1_t X_t^{u^*},
  \\&  z^*_2(t)= D^2_t X_t^{u^*},
  \end{aligned}  
 \end{eqnarray}
 where $D_t=(D^1_t,D^2_t)$ is the Malliavin derivative operator from the Sobolev space $D^{1,2}(\Omega)$ to $L^2(\Omega; L^2([0,T]))$. We refer to \cite{Nualart06,Huang00} for the theory of Malliavin calculus and relevant definitions.
\end{remark}

Moreover, $\varepsilon^{v^*}_t$ is an ${\cal H}_t$-square-integrable martingale by virtue of Remark \ref{aremarkofl2m}. Thus, by (\ref{anequaboutc222_final}) and the generalized Clark-Ocone representation theorem in Malliavin calculus (see \cite{Bermin02}), we have
\begin{equation}\label{c-oth1}
\begin{aligned}
\varepsilon^{v^*}_T=&1+\int_0^T\mathbb{E}\left[   D^1_t \left(\mathbb{E}\left[   e^{ -U\left(\tilde I\left(  c_3^*\Pi^*(0,T)   \right) \right )   }\big |{\cal H}_0\right]   e^{ U\left(\tilde I\left(  c_3^*\Pi^*(0,T)   \right) \right )   }  \right)^{-1} \big{|} {\cal H}_t  \right] \mathrm{d}
W^1_{\cal H}(t)\\
&+\int_0^T\mathbb{E}\left[   D^2_t \left(\mathbb{E}\left[   e^{ -U\left(\tilde I\left(  c_3^*\Pi^*(0,T)   \right) \right )   }\big |{\cal H}_0\right]   e^{ U\left(\tilde I\left(  c_3^*\Pi^*(0,T)   \right) \right )   }  \right)^{-1} \big{|} {\cal H}_t  \right] \mathrm{d}
W^2_{\cal H}(t),
\end{aligned}
\end{equation}
where $D_t=(D^1_t,D^2_t)$ is the extended Malliavin derivative operator from the generalized Sobolev space $D^{0,2}(\Omega)=L^2(\Omega)$ to $D^{-1,2}(\Omega;L^2([0,T]))$, and $\mathbb{E}[\cdot|{\cal H}_t]$ is interpreted as the generalized conditional expectation. We refer to \cite{Nualart06,Huang00} for more details. Comparing (\ref{c-oth1}) with (\ref{half_equ2}) we obtain
\begin{equation}\label{c-oth1-theta}
\begin{aligned}
&\theta_1^*(t)= \frac{\mathbb{E}\left[   D^1_t \left(\mathbb{E}\left[   e^{ -U\left(\tilde I\left(  c_3^*\Pi^*(0,T)   \right) \right )   }\big |{\cal H}_0\right]   e^{ U\left(\tilde I\left(  c_3^*\Pi^*(0,T)   \right) \right )   }  \right)^{-1}\big{|} {\cal H}_t  \right]}
{\mathbb{E}\left[ \left(\mathbb{E}\left[   e^{ -U\left(\tilde I\left(  c_3^*\Pi^*(0,T)   \right) \right )   }\big |{\cal H}_0\right]   e^{ U\left(\tilde I\left(  c_3^*\Pi^*(0,T)   \right) \right )   }  \right)^{-1}\big{|}{\cal H}_t\right]} \\
&\theta_2^*(t)=\frac{\mathbb{E}\left[   D^2_t \left(\mathbb{E}\left[   e^{ -U\left(\tilde I\left(  c_3^*\Pi^*(0,T)   \right) \right )   }\big |{\cal H}_0\right]   e^{ U\left(\tilde I\left(  c_3^*\Pi^*(0,T)   \right) \right )   }  \right)^{-1} \big{|} {\cal H}_t  \right]}
{\mathbb{E}\left[ \left(\mathbb{E}\left[   e^{ -U\left(\tilde I\left(  c_3^*\Pi^*(0,T)   \right) \right )   }\big |{\cal H}_0\right]   e^{ U\left(\tilde I\left(  c_3^*\Pi^*(0,T)   \right) \right )   }  \right)^{-1}\big{|}{\cal H}_t\right]}.
\end{aligned}
\end{equation}

To sum up, we give the following theorem.
 
    \begin{theorem} \label{ex_mainth_idontknow}
Assume that $G^1(\mathrm{d}z)=G^2(\mathrm{d}z)=0$, $\mu(t,x)=\mu_0(t) $ for some ${\cal G}^1_t$-adapted c\`{a}gl\`{a}d process $\mu_0(t)$, $b\ge \epsilon>0$ for some positive constant $\epsilon$, and $g(s,v)=\frac{1}{2} (\theta_1^2+\theta_2^2)$. Suppose $(u^*,v^*)\in {\cal A}_1'\times{\cal A}_2'$ is optimal for Problem \ref{sdg} under the conditions of Theorem \ref{mainth4}, and Assumption \ref{extraa1} and the integrability condition (\ref{linearbsde_bdg}) hold. Then $u^*$, $v^*$, $X^{u^*}_t$ and $\varepsilon^{v^*}_t$ are given by (\ref{bsde_barzt1}), (\ref{halfequ_1})-(\ref{halfequ_2}), (\ref{linearbsde_bdg_2}) and (\ref{anequaboutc22_final}), respectively, where $\Pi^*$ is given by (\ref{definitionofPi}), $c^*_3$ is determined by (\ref{c3determined}), and $(X^{u^*},  z^*_1, z^*_2)$ solves the linear BSDE (\ref{final_gex_1}). Furthermore, if $\{{\cal H}_t\}_{0\le t\le T}$ is the augmentation of the natural filtration of $W^1_{\cal H}(t)$ and $W^2_{\cal H}(t)$, $\mathbb{E} \tilde I(c_3^*(\Pi^*(0,T)))^2 <\infty$, and $r$, $\tilde \phi_1$ and $\tilde\phi_2$ are bounded, then the linear BSDE (\ref{final_gex_1}) has a unique strong solution, and $v^*$ is given by (\ref{c-oth1-theta}).
 \end{theorem}

\begin{remark}
If the filtration $\{{\cal H}_t\}_{0\le t\le T}$ in Theorem \ref{ex_mainth_idontknow} is not  the augmentation of the natural filtration of $W^1_{\cal H}(t)$ and $W^2_{\cal H}(t)$, or the coefficients of the generator $f_{\text{L}}$ is not necessarily bounded, we refer to \cite{Eyraoud05,Li06,Lu13,Wang07} for further results. In those cases, the existence and uniqueness of the solution to the BSDE (\ref{final_gex_1}) still hold under mild conditions when a general martingale representation property was assumed, or a transposition solution was considered, or a stochastic Lipschitzs condition was considered. 
\end{remark}

When the utility function is of the logarithmic form, i.e., $U(x)=\ln x $, we can easily calculate the fixed point function $\tilde I(y)$ as follows
\begin{equation}\label{fixedpln}
\tilde I(y)=\frac{1}{\sqrt{y}},\quad y>0.
\end{equation}

Combining (\ref{linearbsde_bdg_2}) with (\ref{c3determined}), we have
\begin{equation} \begin{aligned}\label{linearbsde_bdg_2_iit}
 X^{u^*}_t=\frac{X_0\mathbb{E} \left[   \sqrt{  \Pi^*(t,T)}  | {\cal H}_t\right]}{\mathbb{E} \left[  \sqrt{  \Pi^*(0,T)} | {\cal H}_0\right]\sqrt{  \Pi^*(0,t)}},
\end{aligned}\end{equation}
and
 \begin{equation}\label{linearbsde_bdg_2_iiT}
 \begin{aligned}
 X^{\pi^*}_T=\frac{X_0 }{\mathbb{E} \big[  \sqrt{  \Pi^*(0,T)} | {\cal H}_0\big]\sqrt{  \Pi^*(0,T)}}.
\end{aligned}
\end{equation}
Thus the BSDE (\ref{final_gex_1}) can be rewritten as
 \begin{eqnarray}\label{final_gex_1f}
  \left\{ \begin{aligned}&\mathrm{d}X^{u^*}_t=-f_{\text{L}}(t,X^{u^*}_t,z^*_t,\omega)\mathrm{d}t+z^*_t\mathrm{d}W_{\cal H}(t), \quad 0\le t\le T,
  \\& X^{u^*}_T=\frac{X_0 }{\mathbb{E} \big[  \sqrt{  \Pi^*(0,T)} | {\cal H}_0\big]\sqrt{  \Pi^*(0,T)}}.\end{aligned} \right.   
 \end{eqnarray}
where $f_{\text{L}} $  is given by
 \begin{equation*} 
\begin{aligned}
f_{\text{L}}(t,x,z,\omega)=-r_tx-\tilde\phi_1(t)z_1  +\tilde\phi_2(t)z_2.
\end{aligned} 
 \end{equation*}

By (\ref{part1ofJ}), (\ref{equeg1_42fb}), (\ref{anequaboutc2}) and (\ref{linearbsde_bdg_2_iiT}), we can calculate the value of Problem \ref{sdg} as follows
\begin{equation}  \label{dontknowv}
  \begin{aligned}
 V&=\mathbb{E}\left[\varepsilon_T^{v^*}\ln \left( \varepsilon^{v^*}_TX^{u^*}_T  \right)    \right]  
 =\mathbb{E}\left[\varepsilon_T^{v^*} c^*_2 \right] =\mathbb{E}\left[-\varepsilon_T^{v^*}\ln \mathbb{E} [  (X_T^{u^*})^{-1}|{\cal H}_0  ] \right]\\
 &=-\mathbb{E}\left[ \mathbb{E}\left[\varepsilon_T^{v^*}\ln \mathbb{E} [  (X_T^{u^*})^{-1}|{\cal H}_0  ]\big{|} {\cal H}_0 \right]  \right]=-\mathbb{E}\left[ \ln \mathbb{E} [  (X_T^{u^*})^{-1}|{\cal H}_0  ]\mathbb{E}\left[\varepsilon_T^{v^*} \big{|} {\cal H}_0 \right]  \right]\\
 &=-\mathbb{E}\left[ \ln \mathbb{E} [  (X_T^{u^*})^{-1}|{\cal H}_0  ]  \right]=\ln X_0-2\mathbb{E}\left[ \ln \mathbb{E}\big[ \sqrt{\Pi^*(0,T)}  |{\cal H}_0\big] \right].
   \end{aligned}
 \end{equation}
 
    \begin{corollary} \label{corofvalue}
Assume that $U(x)=\ln x$, $G^1(\mathrm{d}z)=G^2(\mathrm{d}z)=0$, $\mu(t,x)=\mu_0(t) $ for some ${\cal G}^1_t$-adapted c\`{a}gl\`{a}d process $\mu_0(t)$, $b\ge \epsilon>0$ for some positive constant $\epsilon$, and $g(s,v)=\frac{1}{2} (\theta_1^2+\theta_2^2)$. Suppose $(u^*,v^*)\in {\cal A}_1'\times{\cal A}_2'$ is optimal for Problem \ref{sdg} under the conditions of Theorem \ref{mainth4}, and  Assumption \ref{extraa1} and the integrability condition (\ref{linearbsde_bdg}) hold. Then $(u^*,v^*)$ and $V$ are given by (\ref{bsde_barzt1}), (\ref{halfequ_21})-(\ref{halfequ_22}) and (\ref{dontknowv}), respectively, where $\Pi^*$ is given by (\ref{definitionofPi}), and $(X^{u^*},  z^*_1, z^*_2)$ solves the linear BSDE (\ref{final_gex_1f}). 
 \end{corollary}

When the insurer has no insider information, i.e., ${\cal H}_t={\cal F}_t$, we have $\phi_1=\phi_2=0$. Assume further that all the parameter processes are deterministic functions. By (\ref{linearbsde_bdg_2_iit}) we have
\begin{equation} \begin{aligned}
 X^{u^*}_t=\frac{X_0  \mathbb{E} \left[  \sqrt{  \Pi^*(t,T)} | {\cal F}_t\right]}{\mathbb{E}  \sqrt{  \Pi^*(0,T)}  \sqrt{  \Pi^*(0,t)}}.
\end{aligned}\end{equation}
By (\ref{final_ex_bmoofz1}) and Malliavin calculus, we can easily obtain
   \begin{equation}  
 \begin{aligned}
z_1^*(t)=D^1_tX^{u^*}_t&=\frac{X_0\mathbb{E}\left[ D^1_t\sqrt{\Pi^*(t,T)  } |{\cal F}_t   \right]}{  \mathbb{E}\sqrt{\Pi^*(0,T) }\sqrt{\Pi^*(0,t)}}-\frac{1}{2} \frac{X_0\mathbb{E}\left[ \sqrt{\Pi^*(t,T)  } |{\cal F}_t   \right]}{  \mathbb{E}\sqrt{\Pi^*(0,T) }\Pi^*(0,t)^{\frac{3}{2}}}D^1_t\Pi^*(0,t)\\
&= \frac{1}{2} \frac{X_0\mathbb{E}\left[ \sqrt{\Pi^*(t,T)   } |{\cal F}_t   \right]}{  \mathbb{E}\sqrt{\Pi^*(0,T) }\sqrt{\Pi^*(0,t)}}\iota_t,
\end{aligned}
\end{equation}
 and
    \begin{equation}  
 \begin{aligned}
z_2^*(t)=D^2_tX^{u^*}_t=&\frac{X_0\mathbb{E}\left[ D^2_t\sqrt{\Pi^*(t,T)  } |{\cal F}_t   \right]}{  \mathbb{E}\sqrt{\Pi^*(0,T) }\sqrt{\Pi^*(0,t)}}-\frac{1}{2} \frac{X_0\mathbb{E}\left[ \sqrt{\Pi^*(t,T)   } |{\cal F}_t   \right]}{  \mathbb{E}\sqrt{\Pi^*(0,T) }\Pi^*(0,t)^{\frac{3}{2}}}D^2_t\Pi^*(0,t)\\
=&-\frac{1}{2} \frac{X_0\mathbb{E}\left[ \sqrt{\Pi^*(t,T)   } |{\cal F}_t   \right]}{  \mathbb{E}\sqrt{\Pi^*(0,T) }\sqrt{\Pi^*(0,t)}}\frac{\lambda_t-a_t+\rho b_t\iota_t}{\sqrt{1-\rho^2}b_t}.
\end{aligned}
\end{equation}

 We can obtain the robust optimal strategy by (\ref{bsde_barzt1}) as follows
  \begin{equation}  \label{conti_noinsider}
 \begin{aligned}
 \pi^*_t&=  \frac{\iota_t}{2\sigma_t}+\frac{\rho(\lambda_t-a_t+\rho b_t\iota_t)}{2(1-\rho^2)\sigma_t b_t} ,
 \\ \kappa^*_t&= \frac{\lambda_t-a_t+\rho b_t\iota_t}{2(1-\rho^2)b^2_t}  .
\end{aligned}
\end{equation}

By the equations (\ref{halfequ_21}) and (\ref{halfequ_22}) in Corollary \ref{mainth6}, we can also obtain
 \begin{equation}\label{bsde_barzt1_donsktheta_0}
   \begin{aligned}
 &\theta_1^*(t)=-\frac{ \iota_t}{2} ,   
  \\&\theta_2^*(t)= \frac{\lambda_t-a_t+\rho b_t\iota_t}{2\sqrt{1-\rho^2}b_t}. 
\end{aligned} 
 \end{equation}
 
 The value of Problem \ref{sdg} can be obtained by (\ref{dontknowv}) as follows
\begin{equation}  \label{value1}
  \begin{aligned}
 V&=\ln X_0-\ln  \left(\mathbb{E}\sqrt{\Pi^*(0,T)} \right)^2 \\
 &=\ln X_0+\int_0^Tr_t\mathrm{d}t+\frac{1}{4}\int_0^T\left(\iota_t^2+\frac{(\lambda_t-a_t+\rho b_t\iota_t)^2}{(1-\rho^2)b_t^2} \right)\mathrm{d}t.
   \end{aligned}
 \end{equation} 

Thus we obtain the following corollary.

    \begin{corollary} \label{corex_mainth_idontknow_donsker}
Assume that $U(x)=\ln x$, $G^1(\mathrm{d}z)=G^2(\mathrm{d}z)=0$, $\mu(t,x)=\mu_0(t) $ for some c\`{a}gl\`{a}d function $\mu_0(t)$, $b\ge \epsilon>0$ for some positive constant $\epsilon$, and $g(s,v)=\frac{1}{2} (\theta_1^2+\theta_2^2)$. Assume further that ${\cal H}_t={\cal F}_t$ and all parameter processes are deterministic functions. Suppose $(u^*,v^*)\in {\cal A}_1'\times{\cal A}_2'$ is optimal for Problem \ref{sdg} under the conditions of Theorem \ref{mainth4}. Then $(u^*,v^*)$ is given by (\ref{conti_noinsider}) and (\ref{bsde_barzt1_donsktheta_0}), and the value $V$ is given by (\ref{value1}).
 \end{corollary}

Next, we concentrate on a special situation without model uncertainty. The stochastic differential game problem (\ref{sdg}) degenerates to the following anticipating stochastic control problem.
\begin{problem}\label{sdg2fb}
 Select $u^*\in {\cal A}_1' $ such that
\begin{equation}  
\tilde V:=\tilde J(u^*)= \sup_{u\in {\cal A}_1'}   \tilde J(u),
\end{equation}  
where $\tilde J(u):=\mathbb{E}\left[  U( X_T^u )\right]$. We call $\tilde V$ the value (or the optimal expected utility) of Problem \ref{sdg2fb}.
\end{problem}

We still assume that $U(x)=\ln x$, ${\cal H}_t={\cal F}_t$ and all parameter processes are deterministic functions. Then the terminal value condition of $X^{u^*}_t$ in (\ref{final_gex_1f}) is replaced by
\begin{equation} 
X^{u^*}_T=\frac{X_0}{\Pi^*(0,T)}.
\end{equation}  
By the similar procedure, we obtain the following result without insider information.

    \begin{proposition}  \label{co2rex_mainth_idontknow_donsker}
Assume that $U(x)=\ln x$, $G^1(\mathrm{d}z)=G^2(\mathrm{d}z)=0$, $\mu(t,x)=\mu_0(t) $ for some c\`{a}gl\`{a}d function $\mu_0(t)$, $b\ge \epsilon>0$ for some positive constant $\epsilon$, and no model certainty is considered. Assume further that ${\cal H}_t={\cal F}_t$ and all parameter processes are deterministic functions. Suppose $u^*\in {\cal A}_1'$ is optimal for Problem \ref{sdg2fb} under the conditions of Theorem \ref{mainth4} (with ${\cal A}_2'=\{(0,0)\}$). Then $u^* $ and $\tilde V$ are given by
\begin{equation} \begin{aligned}
&\pi^*_t =  \frac{\iota_t}{\sigma_t}+\frac{\rho(\lambda_t-a_t+\rho b_t\iota_t)}{(1-\rho^2)\sigma_t b_t} ,\\
&\kappa^*_t =\frac{\lambda_t-a_t+\rho b_t\iota_t}{(1-\rho^2)b^2_t}   ,\\
&\tilde V =\ln X_0+\int_0^Tr_t\mathrm{d}t+\frac{1}{2}\int_0^T \left(\iota_t^2 +\frac{(\lambda_t-a_t+\rho b_t\iota_t)^2}{(1-\rho^2)b^2_t}  \right)\mathrm{d}t.
\end{aligned}\end{equation} 
 \end{proposition}

\begin{remark}\label{compare1}
By comparing Corollary \ref{corex_mainth_idontknow_donsker} with Proposition \ref{co2rex_mainth_idontknow_donsker}, the insurer should cut both investment and insurance in half once she is ambiguity averse. The reason is that ambiguity aversion makes her strategy more conservative. Moreover, the difference between the optimal expected utility under the worst-case probability and that under the original reference probability delivers aversion to such ambiguity, and can be characterized by $V-\tilde V=-\frac{1}{4}\int_0^T \left(\iota_t^2 +\frac{(\lambda_t-a_t+\rho b_t\iota_t)^2}{(1-\rho^2)b^2_t}  \right)\mathrm{d}t$, which is nonpositive. 
\end{remark}

\subsubsection{A particular case}
\label{subsubconti1}
Next, we give a particular case to derive the closed form of the robust optimal strategy. Assume that $U(x)=\ln x $, the insider information is related to the future value of risk in the insurance market, i.e.,  
\begin{equation}\label{donsker_filtration}
{\cal H}_t=\bigcap_{s>t}({\cal F}_s\vee Y_0):=\bigcap_{s>t}\left({\cal F}_s\vee \int_0^{T_0}\varphi({s'})\mathrm{d}\bar W_{s'}\right),\quad 0\le t\le T,
\end{equation}
for some $T_0>T$, and all the parameter processes are assumed to be deterministic functions. Here, $\varphi_t $ is some deterministic function satisfying $\Vert \varphi\Vert^2_{[s,t]}:=\int_s^{t}\varphi^2(s')    \mathrm{d}s'<\infty$ for all $0\le s\le t\le T_0$, and $\Vert \varphi\Vert^2_{[T,T_0]}>0$. 

In this situation, each ${\cal H}_t$-adapted process $x_t$, $t\in[0,T]$, has the form $x_t=x_1(t,Y_0,\omega)$ for some function $x_1: [0,T]\times\mathbb{R}\times\Omega\rightarrow \mathbb{R}$ such that $x_1(t,y)$ is ${\cal F}_t$-adapted for every $y\in\mathbb{R}$.  For simplicity, we write $x$ instead of $x_1$ in the sequel. To solve the anticipating linear BSDE (\ref{final_gex_1}), we need to introduce some white noise techniques in Malliavin calculus (see \cite{Draouil15,Huang00,DiNunno09}).

\begin{definition}[Donsker $\delta$ functional]\label{definitionofdonsker}
Let $Y:\Omega\rightarrow \mathbb{R}$ be a random variable which belongs to the distribution space $({\cal S})^{-1}$ (see \cite{Huang00} for the definition). Then a  continuous linear operator $\delta_{\cdot}(Y):\mathbb{R}\rightarrow ({\cal S})^{-1}$ is called a Donsker $\delta$ functional of $Y$ if it has the property that
\begin{equation*}
 \int_\mathbb{R}f(y)\delta_y(Y)\mathrm{d}y=f(Y)
\end{equation*}
for all Borel measurable functions $f:\mathbb{R}\rightarrow\mathbb{R}$ such that the integral converges in $({\cal S})^{-1}$.
\end{definition}

The following lemma gives a sufficient condition for the existence of the Donsker $\delta$ functional. The proof can be found in \cite{DiNunno09}.

\begin{lemma} \label{lemmaofdonsker}
Let $Y: \Omega\rightarrow \mathbb{R}$ be a Gaussian random variable with mean $\bar \mu$ and variance $\bar\sigma^2>0$. Then its Donsker $\delta$ functional $\delta_{y}(Y)$ exists and is uniquely given by
\begin{equation*}
\delta_y(Y)=\frac{1}{\sqrt{2\pi\bar\sigma^2}}\exp^{\diamond}\left\{-\frac{(y-Y)^{\diamond 2}}{2\bar\sigma^2}  \right\}\in ({\cal S})'\subset ({\cal S})^{-1},
\end{equation*}
\end{lemma}
where $({\cal S})'$ is the Hida distribution space, and $\diamond$ denotes the Wick product. We refer to \cite{Huang00} for relevant definitions.

Moreover, we can obtain the explicit expressions of $\phi_1$ and $\phi_2$ as the following lemma, which was first proved in \cite[Theorem 2.2]{Draouil16}.

\begin{lemma}[Enlargement of filtration] \label{lemmaofdonskersemi}
Suppose $Y$ is an ${\cal F}_{T_0}$-measurable random variable for some $T_0> T$ and belongs to $({\cal S})'$. The Donsker $\delta$ functional of $Y$ exists and satisfies $\mathbb{E}[ \delta_\cdot(Y) |{\cal F}_t]\in L^2(\mathrm{m}\times\mathbb{P})$ and $\mathbb{E}[ D^i_t\delta_\cdot(Y) |{\cal F}_t]\in L^2(\mathrm{m}\times\mathbb{P})$, where $D^i_t$ is the (extended) Hida-Malliavin derivative (see \cite{DiNunno09}), $i=1, 2$. Assume further that ${\cal H}_t=\bigcap_{s>t} ({\cal F}_s\vee Y)$, and $W^i$ is an ${\cal H}_t$-semimartingale with the decomposition (\ref{semi_decomp}), $i=1, 2$. Then we have
\begin{equation*}\begin{aligned}
\phi_1(t)&=\frac{\mathbb{E}[D^1_t\delta_{y}(Y)|{\cal F}_t]|_{y=Y}}{\mathbb{E}[\delta_y(Y)|{\cal F}_t]|_{y=Y}},\\
\phi_2(t)&=\frac{\mathbb{E}[D^2_t\delta_{y}(Y)|{\cal F}_t]|_{y=Y}}{\mathbb{E}[\delta_y(Y)|{\cal F}_t]|_{y=Y}}.
\end{aligned}\end{equation*}
\end{lemma}

By Lemma \ref{lemmaofdonsker} and the L\'{e}vy theorem, the Donsker $\delta$ functional of $Y_0$ in (\ref{donsker_filtration}) is given by
\begin{equation}\begin{aligned}
\delta_y(Y_0)=\frac{1}{\sqrt{2\pi\Vert \varphi\Vert_{[0,T_0]}^2}}\exp^{\diamond}\Bigg\{-\frac{(y-Y_0)^{\diamond 2}}{2\Vert\varphi \Vert_{[0,T_0]}^2}  \Bigg\},
\end{aligned}\end{equation}
and we have
\begin{equation}\label{defofg}
\mathbb{E}[\delta_y(Y_0)|{\cal F}_t]=\frac{1}{\sqrt{2\pi\Vert \varphi\Vert_{[t,T_0]}^2}}\exp\Bigg\{-\frac{(y-\int_0^t\varphi_s\mathrm{d}\bar W_s)^{ 2}}{2\Vert\varphi \Vert_{[t,T_0]}^2}  \Bigg\},\quad 0\le t\le T.
\end{equation}
 Moreover, when the filtration $\{{\cal H}_t\}_{0\le t\le T}$ is of the form (\ref{donsker_filtration}), we have by Lemma \ref{lemmaofdonskersemi} that
\begin{equation}\label{weknowphi1}
\begin{aligned}
&\phi_1(t)=\phi_1(t,Y_0)=\frac{Y_{0}-\int_0^t\varphi_s\mathrm{d}\bar W_s}{\Vert \varphi\Vert^2_{[t,T_0]}}\rho \varphi_t,
\\&  \phi_2(t)=\phi_2(t,Y_0)=\frac{Y_{0}-\int_0^t\varphi_s\mathrm{d}\bar W_s}{\Vert \varphi\Vert^2_{[t,T_0]}}\sqrt{1-\rho^2}\varphi_t .
  \end{aligned}\end{equation}
Substituting (\ref{weknowphi1}) into (\ref{definitionofPi}) and using the It\^{o} formula we can rewrite the expression of $\Pi^*(0,t)$ as follows
    \begin{equation} \label{redefinitionofPi}
  \begin{aligned}
\Pi^*(0,t)&=\Pi^*(0,t,y)|_{y=Y_0}\\
&= \exp\left\{-\int_0^t\phi_1(s,Y_0)\mathrm{d}W^1_s-\int_0^t\phi_2(s,Y_0)\mathrm{d}W^2_s+\frac{1}{2}\int_0^t \left(\phi_1^2(s,Y_0)+\phi^2_2(s,Y_0)\right)\mathrm{d}y   \right\}\Pi^*_a(0,t) \\
&=\mathbb{E}[\delta_y(Y_0)|{\cal F}_0]\frac{\Pi^*_a(0,t)}{\mathbb{E}[\delta_y(Y_0)|{\cal F}_t]}\big|_{y=Y_0},
  \end{aligned}
 \end{equation}
where
    \begin{equation}  \label{pi_adef}
  \begin{aligned}
\Pi^*_a(0,t):= \exp\Bigg\{  -\int_{0}^{t} r_s  \mathrm{d}s
  -\int_{0}^{t} \iota_s  \mathrm{d}W^1_s
  +  \int_{0}^{t}  \frac{\lambda_s-a_s+\rho b_s\iota_s}{\sqrt{1-\rho^2}b_s}\mathrm{d}W^2_s  -\frac{1}{2}\int_{0}^{t}  \left(  \iota_s ^2+  \frac{(\lambda_s-a_s+\rho b_s\iota_s)^2}{ (1-\rho^2)b^2_s}   \right)\mathrm{d}s \Bigg\}
  \end{aligned}
 \end{equation}
 is an ${\cal F}_t$-adapted process.
 
In the case of the logarithmic utility, the fixed point function $\tilde I(y)$ is given by (\ref{fixedpln}). Thus, the terminal value condition of $X^{u^*}_T$ in (\ref{final_gex_1f}) leads to
\begin{equation}\label{c3determined_donsker_iii}\begin{aligned}
\\
X^{u^*}_T &=X^{u^*}_T(y)|_{y=Y_0}=\frac{X_0}{\mathbb{E}\big[\sqrt{\Pi^*(0,T)} | {\cal H}_0\big](y)}\sqrt{\frac{ \mathbb{E}[\delta_y(Y_0)|{\cal F}_T]}{   \mathbb{E}[\delta_y(Y_0)|{\cal F}_0]   \Pi^*_a(0,T)}} \Bigg|_{y=Y_0}
=\tilde c^*_3(y)  \sqrt{\frac{ \mathbb{E}[\delta_y(Y_0)|{\cal F}_T]}{    \Pi^*_a(0,T)}}\Bigg|_{y=Y_0},\end{aligned}
\end{equation}
where $\tilde c^*_3(y):=\frac{X_0}{\mathbb{E}\big[\sqrt{\Pi^*(0,T)} | {\cal H}_0\big](y) \sqrt{\mathbb{E}[\delta_y(Y_0)|{\cal F}_0]}  }$ is a Borel measurable function with respect to $y$.

By Definition \ref{definitionofdonsker}, the BSDE (\ref{final_gex_1f}) can be rewritten as
  \begin{equation}\label{final_gex_1_donsker}\begin{aligned}
\int_{\mathbb{R}} X^{u^*}_t(y)\delta_y(Y_0)\mathrm{d}y=& \int_{\mathbb{R}}  X_T^{u^*}(y)\delta_y(Y_0)\mathrm{d}y+ \int_{\mathbb{R}} \int_t^T\left[-r_sX_s^{u^*}(y) -\iota_s z^*_1(s,y)+\frac{\lambda_s-a_s+\rho b_s\iota_s}{\sqrt{1-\rho^2}b_s}z^*_2(s,y) \right] \mathrm{d}s \delta_y(Y_0) \mathrm{d}y\\
&-\int_{\mathbb{R}}\int_t^Tz^*_s(y)\mathrm{d}^-W_s\delta_y(Y_0) \mathrm{d}y, \quad 0\le t\le T,
\end{aligned} \end{equation}
where $W_t=(W_1(t),W_2(t))^\tau$. Since $z^*_t(y)$ is ${\cal F}_t$-adapted for each $y$, the forward integral in (\ref{final_gex_1_donsker}) is just the It\^{o} integral. Then $(\ref{final_gex_1f})$ holds if and only if we choose $(X^{u^*}_t(y),z^*_t(y))$ for each $y$ as the solution of the following classical linear BSDE with respect to the filtration $\{{\cal F}_t\}_{0\le t\le T}$
  \begin{eqnarray}\label{final_gex_1_donsker2}
  \left\{ \begin{aligned}&\mathrm{d}X^{u^*}_t(y)=-\bar f_{\text{L}}(t,X^{u^*}_t(y),z^*_t(y))\mathrm{d}t+z^*_t(y)\mathrm{d}W_t, \quad 0\le t\le T,
  \\& X^{u^*}_T(y)=\tilde c^*_3(y) \sqrt{\frac{\mathbb{E}[\delta_y(Y_0)|{\cal F}_T]}{     \Pi^*_a(0,T)}} ,\end{aligned} \right.   
 \end{eqnarray}
 where the generator $\bar f_{\text{L}}:[0,T]\times \mathbb{R}\times \mathbb{R}^2 \rightarrow \mathbb{R}$  is given by
  \begin{equation*} 
\begin{aligned}
\bar f_{\text{L}}(t,x,z )=-r_tx-\iota_tz_1  +\frac{\lambda_t-a_t+\rho b_t\iota_t}{\sqrt{1-\rho^2}b_t}z_2.
\end{aligned} 
 \end{equation*}
 
By \cite[Theorem 4.8]{Oksendal19}, the unique strong solution of (\ref{final_gex_1_donsker2}) is given by
  \begin{equation}\label{linearbsde_bdg_2_donsker}
 \begin{aligned}
 X^{u^*}_t(y)= \mathbb{E} \left[   \Pi_a^*(t,T) \tilde c^*_3(y)\sqrt{\frac{ \mathbb{E}[\delta_y(Y_0)|{\cal F}_T]}{     \Pi^*_a(0,T)}} \bigg | {\cal F}_t\right],
\end{aligned}
\end{equation}
where $\Pi_a^*[t,T]:=\Pi_a^*[0,T]/\Pi_a^*[0,t]$. By (\ref{linearbsde_bdg_2_donsker}) and the initial value condition $X^{u^*}_0(y)=X_0$, the Borel measurable function $\tilde c^*_3(y)$ is given by
 \begin{equation} \label{c3determined_donsker}
 \begin{aligned}
\tilde c^*_3(y) =\frac{X_0}{\mathbb{E}     \sqrt{  \Pi_a^*(0,T) \mathbb{E}[\delta_y(Y_0)|{\cal F}_T]} }=\frac{X_0}{\mathbb{E}\big[\sqrt{\Pi^*(0,T)} | {\cal H}_0\big](y) \sqrt{\mathbb{E}[\delta_y(Y_0)|{\cal F}_0]}  }.
\end{aligned}
\end{equation}
 The last equation in (\ref{c3determined_donsker}) is by the definition of $\tilde c_3^*(y)$. Substituting (\ref{c3determined_donsker}) into (\ref{linearbsde_bdg_2_donsker}) we obtain
   \begin{equation}\label{linearbsde_bdg_2_donsker_2}
 \begin{aligned}
 X^{u^*}_t(y)= \frac{X_0\mathbb{E}\left[ \sqrt{\Pi^*_a(t,T)\mathbb{E}[\delta_y(Y_0)|{\cal F}_T]  } |{\cal F}_t   \right]}{  \mathbb{E}\sqrt{\Pi^*_a(0,T)\mathbb{E}[\delta_y(Y_0)|{\cal F}_T]}\sqrt{\Pi^*_a(0,t)}}.
\end{aligned}
\end{equation}
By \cite[Proposition 3.5.1]{Delong13} and Malliavin calculus, we have
   \begin{equation} \label{iknow1}
 \begin{aligned}
z_1^*(t,y)&=D^1_tX^{u^*}_t(y)\\
&=\frac{X_0\mathbb{E}\left[ D^1_t\sqrt{\Pi^*_a(t,T)\mathbb{E}[\delta_y(Y_0)|{\cal F}_T]  } |{\cal F}_t   \right]}{  \mathbb{E}\sqrt{\Pi^*_a(0,T)\mathbb{E}[\delta_y(Y_0)|{\cal F}_T]}\sqrt{\Pi^*_a(0,t)}}-\frac{1}{2} \frac{X_0\mathbb{E}\left[ \sqrt{\Pi^*_a(t,T)\mathbb{E}[\delta_y(Y_0)|{\cal F}_T]  } |{\cal F}_t   \right]}{  \mathbb{E}\sqrt{\Pi^*_a(0,T)\mathbb{E}[\delta_y(Y_0)|{\cal F}_T]}\Pi^*_a(0,t)^{\frac{3}{2}}}D^1_t\Pi^*_a(0,t)\\
&=\frac{1}{2}\frac{X_0\mathbb{E}\left[ \sqrt{ \Pi^*_a(t,T)\mathbb{E}[\delta_y(Y_0)|{\cal F}_T]   } \left(y-\int_0^T\varphi_s\mathrm{d}\bar W_s\right)   \big|{\cal F}_t   \right]}{ \Vert \varphi\Vert^2_{[T,T_0]} \mathbb{E}\sqrt{\Pi^*_a(0,T)\mathbb{E}[\delta_y(Y_0)|{\cal F}_T]}\sqrt{\Pi^*_a(0,t)}}\rho\varphi_t+\frac{1}{2} \frac{X_0\mathbb{E}\left[ \sqrt{\Pi^*_a(t,T)\mathbb{E}[\delta_y(Y_0)|{\cal F}_T]  } |{\cal F}_t   \right]}{  \mathbb{E}\sqrt{\Pi^*_a(0,T)\mathbb{E}[\delta_y(Y_0)|{\cal F}_T]}\sqrt{\Pi^*_a(0,t)}}\iota_t,
\end{aligned}
\end{equation}
 and
    \begin{equation} \label{iknow2}
 \begin{aligned}
z_2^*(t,y)&=D^2_tX^{u^*}_t(y)\\
&=\frac{X_0\mathbb{E}\left[ D^2_t\sqrt{\Pi^*_a(t,T)\mathbb{E}[\delta_y(Y_0)|{\cal F}_T]  } |{\cal F}_t   \right]}{  \mathbb{E}\sqrt{\Pi^*_a(0,T)\mathbb{E}[\delta_y(Y_0)|{\cal F}_T]}\sqrt{\Pi^*_a(0,t)}}-\frac{1}{2} \frac{X_0\mathbb{E}\left[ \sqrt{\Pi^*_a(t,T)\mathbb{E}[\delta_y(Y_0)|{\cal F}_T]  } |{\cal F}_t   \right]}{  \mathbb{E}\sqrt{\Pi^*_a(0,T)\mathbb{E}[\delta_y(Y_0)|{\cal F}_T]}\Pi^*_a(0,t)^{\frac{3}{2}}}D^2_t\Pi^*_a(0,t)\\
&=\frac{1}{2}\frac{X_0\mathbb{E}\left[ \sqrt{ \Pi^*_a(t,T)\mathbb{E}[\delta_y(Y_0)|{\cal F}_T]   } \left(y-\int_0^T\varphi_s\mathrm{d}\bar W_s\right)   \big|{\cal F}_t   \right]}{ \Vert \varphi\Vert^2_{[T,T_0]} \mathbb{E}\sqrt{\Pi^*_a(0,T)\mathbb{E}[\delta_y(Y_0)|{\cal F}_T]}\sqrt{\Pi^*_a(0,t)}}\sqrt{1-\rho^2}\varphi_t\\
&\quad -\frac{1}{2} \frac{X_0\mathbb{E}\left[ \sqrt{\Pi^*_a(t,T)\mathbb{E}[\delta_y(Y_0)|{\cal F}_T]  } |{\cal F}_t   \right]}{  \mathbb{E}\sqrt{\Pi^*_a(0,T)\mathbb{E}[\delta_y(Y_0)|{\cal F}_T]}\sqrt{\Pi^*_a(0,t)}}\frac{\lambda_t-a_t+\rho b_t\iota_t}{\sqrt{1-\rho^2}b_t}.
\end{aligned}
\end{equation}

Substituting (\ref{iknow1})  into (\ref{bsde_barzt1}) we obtain the robust optimal investment strategy
 \begin{equation}\label{bsde_barzt1_donsk}
   \begin{aligned}
 \pi_t^*&= \frac{\iota_t}{2\sigma_t} +  \frac{\rho(\lambda_t-a_t+\rho b_t\iota_t)}{2(1-\rho^2)\sigma_tb_t} .
 \end{aligned} 
 \end{equation}
 Combining (\ref{bsde_barzt1}) with (\ref{iknow2}) , the robust optimal insurance strategy can also be given as
  \begin{equation}\label{bsde_barzt1_donsk02}
   \begin{aligned}
  \kappa_t^*&= \kappa_t^*(y)|_{y=Y_0} \\
  &=   \frac{\lambda_t-a_t+\rho b_t\iota_t}{2(1-\rho^2)b_t^2}- \frac{\mathbb{E}\left[ \tilde\Pi^*_a(t,T)\sqrt{ \mathbb{E}[\delta_y(Y_0)|{\cal F}_T]   } \left(y-\int_0^T\varphi_s\mathrm{d}\bar W_s\right)   \big|{\cal F}_t   \right]}{2 b_t\Vert \varphi\Vert^2_{[T,T_0]}  \mathbb{E}\left[\tilde\Pi^*_a(t,T) \sqrt{\mathbb{E}[\delta_y(Y_0)|{\cal F}_T]  } |{\cal F}_t   \right]   } \varphi_t  \Bigg{|}_{y=Y_0}
  ,
\end{aligned} 
 \end{equation}
  where
  \begin{equation}  
  \begin{aligned}
\tilde\Pi^*_a(0,t):= \exp\Bigg\{  -\int_{0}^{t} \frac{\iota_s}{2}  \mathrm{d}W^1_s
  +  \int_{0}^{t}  \frac{\lambda_s-a_s+\rho b_s\iota_s}{2\sqrt{1-\rho^2}b_s}\mathrm{d}W^2_s  -\frac{1}{8}\int_{0}^{t}  \left(  \iota_s ^2+  \frac{(\lambda_s-a_s+\rho b_s\iota_s)^2}{ (1-\rho^2)b^2_s}   \right)\mathrm{d}s \Bigg\}.
  \end{aligned}
 \end{equation}
Then by the Girsanov theorem (see Theorem \ref{girsanov2}),  $W^1_{\mathbb Q}(t):=W^1_t+\int_0^t\frac{\iota_s}{2}\mathrm{d}s$ and $W^2_{\mathbb Q}(t):=W^2_t- \int_{0}^{t}  \frac{\lambda_s-a_s+\rho b_s\iota_s}{2\sqrt{1-\rho^2}b_s}\mathrm{d}s$ are two independent ${\cal F}_t$-Brownian motions under the new equivalent probability measure $\mathbb Q$ defined by $\mathrm{d}{\mathbb  Q}=\tilde \Pi^*_a(0,T)\mathrm{d}\mathbb{P}$. Thus, by the Bayes rule (see \cite{Karatzas91}), we can rewritten the robust optimal insurance strategy as follows
 \begin{equation} \label{markov_pi}
   \begin{aligned}
 \kappa_t^* =& \frac{\lambda_t-a_t+\rho b_t\iota_t}{2(1-\rho^2)b_t^2} - \frac{\mathbb{E}_{\mathbb Q}\left[ \sqrt{ \mathbb{E}[\delta_y(Y_0)|{\cal F}_T]   } \left(y-\int_0^T\varphi_s\mathrm{d}\bar W_s\right)   \big|{\cal F}_t   \right]}{2 b_t\Vert \varphi\Vert^2_{[T,T_0]}  \mathbb{E}_{\mathbb Q}\left[  \sqrt{\mathbb{E}[\delta_y(Y_0)|{\cal F}_T]  } |{\cal F}_t   \right]   } \varphi_t  \Bigg{|}_{y=Y_0}
  \\
  =&  \frac{\lambda_t-a_t+\rho b_t\iota_t}{2(1-\rho^2)b_t^2}  - \frac{\mathbb{E}_{\mathbb Q}\Big[ \exp\Big\{-\frac{(y+s_T-\int_0^T\varphi_s\mathrm{d}\bar W_{\mathbb Q}(s) )^{ 2}}{4\Vert\varphi \Vert_{[T,T_0]}^2}  \Big\}\Big( y+s_T-\int_0^T\varphi_s\mathrm{d}\bar W_{\mathbb Q}(s)   \Big)  \big|{\cal F}_t   \Big]}{2b_t \Vert \varphi \Vert^2_{[T,T_0]}\mathbb{E}_{\mathbb Q}\Big[ \exp\Big\{-\frac{(y+s_T-\int_0^T\varphi_s\mathrm{d}\bar W_{\mathbb Q}(s) )^{ 2}}{4\Vert\varphi \Vert_{[T,T_0]}^2}  \Big\} \big|{\cal F}_t  \Big]}    
  \varphi_t  \Bigg{|}_{y=Y_0},
\end{aligned} 
 \end{equation}
where $s_t:=-\frac{1}{2}\int_0^t\varphi_s    \frac{\lambda_s-a_s }{ b_s} \mathrm{d}s$, and $\bar W_{\mathbb Q}(t):=\rho W^1_{\mathbb Q}(t)+\sqrt{1-\rho^2} W^2_{\mathbb Q}(t)$, $t\in[0,T]$.
On the other hand, the conditional $\mathbb Q$ law of $\int_0^T\varphi_s\mathrm{d}\bar W_{\mathbb Q}(s)$, given ${\cal F}_t$, is normal with mean $\int_0^t\varphi_s\mathrm{d}\bar W_{\mathbb Q}(s)$ and variance $\Vert\varphi\Vert^2_{[t,T]}$ due to the Markov property of It\^{o} diffusion processes (see \cite{Karatzas91}). Thus (\ref{markov_pi}) leads to 
 \begin{equation}\label{bsde_barzt1_donsk01}  
   \begin{aligned}
 \kappa_t^* 
  =& \frac{\lambda_t-a_t+\rho b_t\iota_t}{2(1-\rho^2)b_t^2}-  \frac{\int_{\mathbb{R}}   (y+s_T-x) \exp\bigg\{ -\frac{(y+s_T-x)^2}{4\Vert\varphi \Vert_{[T,T_0]}^2} -\frac{(x-z)^2}{2\Vert\varphi \Vert_{[t,T]}^2} \bigg\}  \mathrm{d}x }
  {2b_t \Vert \varphi \Vert^2_{[T,T_0]}  \int_{\mathbb{R}}     \exp\bigg\{ -\frac{(y+s_T-x)^2}{4\Vert\varphi \Vert_{[T,T_0]}^2} -\frac{(x-z)^2}{2\Vert\varphi \Vert_{[t,T]}^2} \bigg\}  \mathrm{d}x   }    
  \varphi_t \Bigg|_{\substack{z=\int_0^t\varphi_s\mathrm{d}\bar W_{\mathbb Q}(s) \\ y=Y_0}}    \\
=& \frac{\lambda_t-a_t+\rho b_t\iota_t}{2(1-\rho^2)b_t^2} -  \frac{Y_0-\int_0^t\varphi_s\mathrm{d}\bar W_s-\frac{1}{2}\int_t^T \varphi_s  \frac{\lambda_s-a_s }{ b_s} \mathrm{d}s }{b_t\left(\Vert\varphi \Vert_{[t,T_0]}^2+\Vert\varphi \Vert_{[T,T_0]}^2\right)}\varphi_t  . \\
\end{aligned} 
 \end{equation}

By the equations (\ref{halfequ_21}) and (\ref{halfequ_22}) in Corollary \ref{mainth6}, we can also obtain
 \begin{equation}\label{bsde_barzt1_donsktheta}
   \begin{aligned}
 &\theta_1^*(t)=-\frac{\iota_t}{2}+\Bigg[\frac{Y_0-\int_0^t\varphi_s\mathrm{d}\bar W_s-\frac{1}{2}\int_t^T \varphi_s  \frac{\lambda_s-a_s }{ b_s} \mathrm{d}s }{ \Vert\varphi \Vert_{[t,T_0]}^2+\Vert\varphi \Vert_{[T,T_0]}^2} -\frac{Y_{0}-\int_0^t\varphi_s\mathrm{d}\bar W_s}{\Vert \varphi\Vert^2_{[t,T_0]}}\Bigg]\rho \varphi_t,
  \\&\theta_2^*(t)= \frac{\lambda_t-a_t+\rho b_t\iota_t}{2\sqrt{1-\rho^2}b_t}+\Bigg[ \frac{Y_0-\int_0^t\varphi_s\mathrm{d}\bar W_s-\frac{1}{2}\int_t^T \varphi_s  \frac{\lambda_s-a_s }{ b_s} \mathrm{d}s }{ \Vert\varphi \Vert_{[t,T_0]}^2+\Vert\varphi \Vert_{[T,T_0]}^2}   - 
 \frac{Y_{0}-\int_0^t\varphi_s\mathrm{d}\bar W_s}{\Vert \varphi\Vert^2_{[t,T_0]}}   \Bigg] \sqrt{1-\rho^2} \varphi_t . 
\end{aligned} 
 \end{equation}

To sum up, we give the following theorem.

    \begin{theorem} \label{ex_mainth_idontknow_donsker}
Assume that $U(x)=\ln x$, $G^1(\mathrm{d}z)=G^2(\mathrm{d}z)=0$, $\mu(t,x)=\mu_0(t) $ for some c\`{a}gl\`{a}d function $\mu_0(t)$, $b\ge \epsilon>0$ for some positive constant $\epsilon$, and $g(s,v)=\frac{1}{2} (\theta_1^2+\theta_2^2)$. Assume further that $\{{\cal H}_t\}_{0\le t\le T}$ is given by (\ref{donsker_filtration}) and all parameter processes are deterministic functions. Suppose $(u^*,v^*)\in {\cal A}_1'\times{\cal A}_2'$ is optimal for Problem \ref{sdg} under the conditions of Theorem \ref{mainth4}. Then $(u^*,v^*)$ is given by (\ref{bsde_barzt1_donsk}), (\ref{bsde_barzt1_donsk01}) and (\ref{bsde_barzt1_donsktheta}).
 \end{theorem}

 Moreover, if $\varphi=1$, $u^*$ is given by
  \begin{equation}\label{cdonsu}
   \begin{aligned}
 \pi_t^*&= \frac{\iota_t}{2\sigma_t} +  \frac{\rho(\lambda_t-a_t+\rho b_t\iota_t)}{2(1-\rho^2)\sigma_tb_t} ,\\
 \kappa^*_t&= \frac{\lambda_t-a_t+\rho b_t\iota_t}{2(1-\rho^2)b_t^2}-  \frac{\bar W_{T_0}-\bar W_{t}-\frac{1}{2}\int_t^T    \frac{\lambda_s-a_s }{ b_s}\mathrm{d}s }{b_t\left(T_0-t+T_0-T\right)}   ,
 \end{aligned} 
 \end{equation}
 and $v^*$ is given by
   \begin{equation}\label{cdonsv}
   \begin{aligned}
&\theta_1^*(t)=-\frac{\iota_t}{2}+\rho \Bigg[\frac{\bar W_{T_0}-\bar W_t-\frac{1}{2}\int_t^T  \frac{\lambda_s-a_s }{ b_s} \mathrm{d}s }{ T_0-t+T_0-T} -\frac{\bar W_{T_0}-\bar W_t}{T_0-t}\Bigg]  ,
  \\&\theta_2^*(t)=\frac{\lambda_t-a_t+\rho b_t\iota_t}{2\sqrt{1-\rho^2}b_t}+\sqrt{1-\rho^2} \Bigg[ \frac{\bar W_{T_0}-\bar W_t-\frac{1}{2}\int_t^T    \frac{\lambda_s-a_s }{ b_s} \mathrm{d}s }{T_0-t+T_0-T}   - 
 \frac{\bar W_{T_0}-\bar W_t}{T_0-t}   \Bigg]   . 
 \end{aligned} 
 \end{equation}
 
By (\ref{c3determined_donsker}) we have 
 \begin{equation}  \label{compareX}
  \begin{aligned}
\mathbb{E}\big[\sqrt{\Pi^*(0,T)}|{\cal H}_0\big]= \left(\mathbb{E}\sqrt{\Pi^*_a(0,T)\frac{\mathbb{E}[\delta_y(Y_0)|{\cal F}_T]}{\mathbb{E}[\delta_y(Y_0)|{\cal F}_0]}}\right) \Bigg{|}_{y=Y_0}\end{aligned}
 \end{equation} 
 Substituting (\ref{compareX}) into (\ref{dontknowv}), we have by Girsanov theorem and tedious calculation that
 \begin{equation}  \label{value2}
  \begin{aligned}
  V=&\ln X_0-2\mathbb{E}\left[\left( \ln  \mathbb{E}\sqrt{\Pi_a^*(0,T)\frac{\mathbb{E}[\delta_y(Y_0)|{\cal F}_T]}{\mathbb{E}[\delta_y(Y_0)|{\cal F}_0]}}\right)\Bigg{|}_{y=Y_0}  \right]
  \\=& \ln X_0+ \int_0^T r_t \mathrm{d}t  +\frac{1}{4}\int_0^T\left(\iota_t^2+\frac{(\lambda_t-a_t+\rho b_t\iota_t)^2}{(1-\rho^2)b^2_t}    \right)\mathrm{d}t  -2\mathbb{E}\left[  \left(\ln \mathbb{E}_{\mathbb Q}\sqrt{\frac{\mathbb{E}[\delta_y(Y_0)|{\cal F}_T]}{\mathbb{E}[\delta_y(Y_0)|{\cal F}_0]}   } \right)\Bigg{|}_{y=Y_0}    \right]\\
  =& \ln X_0+ \int_0^T r_t \mathrm{d}t  +\frac{1}{4}\int_0^T\left(\iota_t^2+\frac{(\lambda_t-a_t+\rho b_t\iota_t)^2}{(1-\rho^2)b^2_t}    \right)\mathrm{d}t+\frac{1}{2}\ln \left( 1-\frac{T^2}{(2T_0-T)^2} \right)^{-1}\\
  &+\frac{T}{2(2T_0-T)}  +\frac{1}{4(2T_0-T)}\left( \int_0^T \frac{\lambda_t-a_t}{b_t}   \mathrm{d}t  \right)^2.
   \end{aligned}
 \end{equation}

Thus we obtain the following corollary.
 
     \begin{corollary} \label{ex_mainth_idontknow_donskerc}
Assume that $U(x)=\ln x$, $G^1(\mathrm{d}z)=G^2(\mathrm{d}z)=0$, $\mu(t,x)=\mu_0(t) $ for some c\`{a}gl\`{a}d function $\mu_0(t)$, $b\ge \epsilon>0$ for some positive constant $\epsilon$, and $g(s,v)=\frac{1}{2} (\theta_1^2+\theta_2^2)$. Assume further that $\{{\cal H}_t\}_{0\le t\le T}$ is given by (\ref{donsker_filtration}) with $\varphi=1$ and all parameter processes are deterministic functions. Suppose $(u^*,v^*)\in {\cal A}_1'\times{\cal A}_2'$ is optimal for Problem \ref{sdg} under the conditions of Theorem \ref{mainth4}. Then $(u^*,v^*)$ is given by (\ref{cdonsu}) and (\ref{cdonsv}), and the value $V$ is given by (\ref{value2}).
 \end{corollary}
 
  \begin{remark}
  By comparing Corollary \ref{ex_mainth_idontknow_donskerc} with Corollary \ref{corex_mainth_idontknow_donsker} we can see that, if the insurer captures insider information in the insurance market and $\bar W_{T_0}>\bar W_{t}+\frac{1}{2}\int_t^T\frac{\lambda_s-a_s}{b_s}\mathrm{d}s$ at time $t$, she should reduce the liability ratio by $\frac{\bar W_{T_0}-\bar W_{t}-\frac{1}{2}\int_t^T  \left(  \frac{\lambda_s-a_s }{ b_s}\right)\mathrm{d}s }{b_t\left(T_0-t+T_0-T\right)} $ to maximize her utility under model uncertainty. Given $\bar W_{T_0}-\bar W_t=x>\frac{1}{2}\int_t^T\frac{\lambda_s-a_s}{b_s}\mathrm{d}s$, we can also see that the closer the future time $T_0$ of the insider information is, the more she should reduce her liability ratio. However, the robust optimal investment strategy is not affected since she has no insider information about the risky asset. Moreover, her optimal expected utility under the worst-case probability is gained by $\Delta V=\frac{1}{2}\ln \left( 1-\frac{T^2}{(2T_0-T)^2} \right)^{-1}+\frac{T}{2(2T_0-T)}  +\frac{1}{4(2T_0-T)}\left( \int_0^T \frac{\lambda_t-a_t}{b_t}   \mathrm{d}t  \right)^2>0$, and $\Delta V$ is greater as $T_0$ is closer (i.e., the insurer has `better' insider information). On the other hand, by comparing Corollary \ref{ex_mainth_idontknow_donskerc} with Proposition \ref{co2rex_mainth_idontknow_donsker}, if $T_0^*$ is the solution of $V=\tilde V$, then the optimal expected utility of the insurer under the worst-case probability with insider information equals that under the original reference probability without insider information. We call $T_0^*$ the critical future time, which indicates that if the insurer who is ambiguity averse wants to exceed the optimal expected utility under the original reference probability, the cost is that she needs the insider information about the value of risk at time $T_0$ such that $T<T_0\le T_0^*$. In particular, if all parameter processes are constant, $\rho=0$, and $T=1$, then $T_0^*$ satisfies
 \begin{equation}\label{optimalfuture}
\iota^2+c^2-\frac{ c^2+2}{2T_0^*-1}=-2\ln\left(  1-\frac{1}{(2T_0^*-1)^2} \right),
 \end{equation}
 where $c=\frac{\lambda-a}{b}$. By mathematical analysis, equation (\ref{optimalfuture}) has a unique solution $T^*_0$. Moreover, the critical future time $T^*_0$ gets closer as $\iota$ increases. The reason is that the increase of $\mu$ improves the optimal expected utility of the insurer under the original reference probability more considerably. Thus she requires `better' extra insider information to improve her optimal expected utility under the worst-case probability such that $V=\tilde V$.
  \end{remark}
  
   \begin{remark}\label{re_comp}
 By similar procedure, we can also obtain the robust optimal strategy when the insurer has insider information about the future value of the risky asset. For instance, if ${\cal H}_t=\bigcap_{s>t}\left(  {\cal F}_s\vee W^1_{T_0} \right)$, the robust optimal strategy is given by
 \begin{equation}\begin{aligned}
 &\pi^*_t=\frac{\iota_t}{2\sigma_t}+ \frac{\rho(\lambda_t-a_t+\rho b_t\iota_t)}{2(1-\rho^2)\sigma_t b_t}+\frac{W^1_{T_0}-W^1_t+\frac{1}{2}\int_t^T\iota_s\mathrm{d}s}{\sigma_t\left( T_0-t+T_0-T  \right)},
 \\&\kappa^*_t=\frac{\lambda_t-a_t+\rho b_t\iota_t}{2(1-\rho^2)b^2_t},
 \end{aligned}\end{equation}
 which indicates that the insurer should increase the proportion of her total wealth invested in the risky asset by $\frac{W^1_{T_0}-W^1_t+\frac{1}{2}\int_t^T\iota_s\mathrm{d}s}{\sigma_t\left( T_0-t+T_0-T  \right)}$ to maximize her utility under model uncertainty if she knows the future value of the risky asset and $W^1_{T_0} > W^1_t-\frac{1}{2} \int_t^T \iota_s \mathrm{d}s$ at time $t$. Given $ W^1_{T_0}-  W^1_t=x>-\frac{1}{2}\int_t^T\iota_s\mathrm{d}s$, we can also see that the closer the future time $T_0$ is, the more she should increase her proportion with respect to the risky asset. However, the robust optimal insurance strategy is not affected since she has no information about the risk in the insurance market. Moreover, the value $V$ can be given by
  \begin{equation}  \label{vinsiderofr}
  \begin{aligned}
  V&= \ln X_0+ \int_0^T r_t \mathrm{d}t  +\frac{1}{4}\int_0^T\left(\iota_t^2+\frac{(\lambda_t-a_t+\rho b_t\iota_t)^2}{(1-\rho^2)b^2_t}    \right)\mathrm{d}t +\frac{1}{2}\ln \left( 1-\frac{T^2}{(2T_0-T)^2} \right)^{-1}+\frac{T}{2(2T_0-T)}\\
  &  +\frac{1}{4(2T_0-T)}\left( \int_0^T\iota_t   \mathrm{d}t  \right)^2.
   \end{aligned}
 \end{equation} 
 \end{remark}
 Thus the optimal expected utility under the worst-case probability is gained by $\Delta V=\frac{1}{2}\ln \left( 1-\frac{T^2}{(2T_0-T)^2} \right)^{-1}+\frac{T}{2(2T_0-T)}  +\frac{1}{4(2T_0-T)}\left( \int_0^T \iota_t   \mathrm{d}t  \right)^2>0$, and $\Delta V$ is greater as $T_0$ is closer (i.e., the insurer has `better' insider information). By comparing (\ref{vinsiderofr}) with (\ref{value2}) we find that, if $\iota_t=\frac{\lambda_t-a_t}{b_t}$, the insurer could derive the same optimal expected utility under the worst-case probability when she owns the insider information about the financial market or the insurance market. On the other hand, by Proposition \ref{co2rex_mainth_idontknow_donsker}, the critical future time $T_0^*$ can be defined similarly. In particular, if all parameter processes are constant, $\rho=0$, and $T=1$, then $T_0^*$ satisfies
 \begin{equation}\label{optimalfuture}
\iota^2+c^2-\frac{ \iota^2+2}{2T_0^*-1}=-2\ln\left(  1-\frac{1}{(2T_0^*-1)^2} \right),
 \end{equation}
 where $c=\frac{\lambda-a}{b}$. The critical future time $T^*_0$ gets closer as $c$ increases since the increase of $c$ improves the optimal expected utility of the insurer under the original reference probability more considerably. Thus she requires `better' extra insider information to improve her optimal expected utility under the worst-case probability such that $V=\tilde V$.

Next, we concentrate on the special situation without model uncertainty, i.e., Problem \ref{sdg2fb}. Then the terminal value condition of $X^{u^*}_t$ in (\ref{final_gex_1_donsker2}) is replaced by
\begin{equation}
X^{u^*}_T=\frac{X_0}{\Pi^*(0,T)}=\frac{X_0 \mathbb{E}[\delta_y(Y_0)|{\cal F}_T]}{ \mathbb{E}[\delta_y(Y_0)|{\cal F}_0]\Pi^*_a(0,T)}\bigg{|}_{y=Y_0}.
\end{equation}
By similar procedure, we obtain the following proposition.

     \begin{proposition} \label{pex_mainth_idontknow_donskerc}
Assume that $U(x)=\ln x$, $G^1(\mathrm{d}z)=G^2(\mathrm{d}z)=0$, $\mu(t,x)=\mu_0(t) $ for some c\`{a}gl\`{a}d function $\mu_0(t)$, $b\ge \epsilon>0$ for some positive constant $\epsilon$, and no model certainty is considered. Assume further that $\{{\cal H}_t\}_{0\le t\le T}$ is given by (\ref{donsker_filtration}) with $\varphi=1$ and all parameter processes are deterministic functions. Suppose $u^* \in {\cal A}_1' $ is optimal for Problem \ref{sdg2fb} under the conditions of Theorem \ref{mainth4}. Then $u^* $ and $\tilde V$ are given by  
\begin{equation} \label{valueno}
   \begin{aligned}
&\pi^*_t=\frac{\iota_t}{\sigma_t}+\frac{\rho(\lambda_t-a_t+\rho b_t\iota_t)}{(1-\rho^2)\sigma_t b_t}  ,\\
&\kappa^*_t  =\frac{\lambda_t-a_t+\rho b_t\iota_t}{(1-\rho^2)b^2_t} - \frac{\bar W_{T_0}-\bar W_t}{b_t(T_0-t)}  ,\\
&\tilde V=\ln X_0+\int_0^T r_t\mathrm{d}t+\frac{1}{2}\int_0^T \left(\iota_t^2 +\frac{(\lambda_t-a_t+\rho b_t\iota_t)^2}{(1-\rho^2)b^2_t}  \right)\mathrm{d}t +\frac{1}{2}\ln\frac{T_0}{T_0-T}. 
 \end{aligned} 
 \end{equation}
 \end{proposition}
 
 \begin{remark}
 By comparing Corollary \ref{ex_mainth_idontknow_donskerc} with Proposition \ref{pex_mainth_idontknow_donskerc}, when the insurer owns the insider information about the future value of risk in the insurance market, the difference between her optimal expected utility under the worst-case probability and that under the original reference probability is given by 
 \begin{equation}   \begin{aligned}
 V-\tilde V=&\left[\frac{1}{2}\ln \left( 1-\frac{T^2}{(2T_0-T)^2} \right)^{-1}+ \frac{T}{2(2T_0-T)}-\frac{1}{2}\ln\frac{T_0}{T_0-T} \right]\\
 &+\left[  \frac{1}{4(2T_0-T)}\left( \int_0^T \frac{\lambda_t-a_t}{b_t}   \mathrm{d}t  \right)^2 -\frac{1}{4}\int_0^T\left(\iota_t^2+\frac{(\lambda_t-a_t+\rho b_t\iota_t)^2}{(1-\rho^2)b^2_t}    \right)\mathrm{d}t\right]\le 0.
  \end{aligned} \end{equation}
due to the H\"{o}lder inequality. On the other hand, suppose there is no model certainty, i.e., the insurer is not ambiguity averse or believes that her model is accurate. By comparing Proposition \ref{pex_mainth_idontknow_donskerc} with Proposition \ref{co2rex_mainth_idontknow_donsker} we see that, if the insurer captures insider information in the insurance market and $\bar W_{T_0}>\bar W_{t}$ at time $t$, she should reduce the liability ratio by $\frac{\bar W_{T_0}-\bar W_{t}  }{b_t\left(T_0-t \right)} $ to maximize her utility. Given $\bar W_{T_0}-\bar W_t=x>0$, we can also see that the closer the future time $T_0$ is, the more she should reduce her liability ratio. However, the optimal investment strategy is not affected. Moreover, her optimal expected utility is gained by $\Delta \tilde V=\frac{1}{2}\ln \frac{T_0}{T_0-T}>0$, and $\Delta \tilde V$ is greater as $T_0$ is closer (i.e., she has `better' insider information). 
 \end{remark}
 
  \begin{remark}\label{resimi}
 By similar procedure, we can also obtain the optimal strategy without model uncertainty when the insurer owns the insider information about the future value of the risky asset. For instance, if ${\cal H}_t=\bigcap_{s>t}\left({\cal F}_s\vee W^1_{T_0}\right)$, the optimal strategy is given by
  \begin{equation}\begin{aligned}
&\pi^*_t=\frac{\iota_t}{\sigma_t}+\frac{\rho(\lambda_t-a_t+\rho b_t\iota_t)}{(1-\rho^2)\sigma_t b_t}+\frac{W^1_{T_0}-W^1_t}{\sigma_t(T_0-t)}  ,\\
&\kappa^*_t  =\frac{\lambda_t-a_t+\rho b_t\iota_t}{(1-\rho^2)b^2_t} .
 \end{aligned}\end{equation}
 Moreover, the value $V$ can be given by
 \begin{equation}\label{vinsiderofr2}  \begin{aligned}
\tilde V=\ln X_0+\int_0^T r_t\mathrm{d}t+\frac{1}{2}\int_0^T \left(\iota_t^2 +\frac{(\lambda_t-a_t+\rho b_t\iota_t)^2}{(1-\rho^2)b^2_t}  \right)\mathrm{d}t +\frac{1}{2}\ln\frac{T_0}{T_0-T}, 
\end{aligned}  \end{equation}
which is consistent with the fact that the difference between the optimal expected utility of the insurer under the worst-case probability and that under the original reference probability is nonpositive when she owns the insider information ${\cal H}_t$ by Remark \ref{re_comp}. On the other hand, suppose there is no model uncertainty. By Proposition \ref{co2rex_mainth_idontknow_donsker}, the insurer should increase the proportion of her total wealth invested in the risky asset by $\frac{W^1_{T_0}-W^1_t}{\sigma_t(T_0-t)} $ if she knows the future value of the risky asset and $W^1_{T_0}>W^1_t$ at time $t$. Given $ W^1_{T_0}-  W^1_t=x>0$, we can also see that the closer the future time $T_0$ is, the more she should increase her proportion with respect to the risky asset. However, the optimal insurance strategy is not affected. Moreover, her optimal expected utility is gained by $\Delta \tilde V=\frac{1}{2}\ln \frac{T_0}{T_0-T}>0$, and $\Delta \tilde V$ is greater as $T_0$ is closer (i.e., she has `better' insider information). By comparing (\ref{vinsiderofr2}) with (\ref{valueno}) we find that, the insurer could derive the same optimal expected utility when she owns the insider information about the financial market or the insurance market. 
 \end{remark}

\subsection{With Poisson jumps}
\label{sec_exampleg_sub2}
Now we focus on the situation where Possion jumps might happen on the insurance market. More specifically, we suppose that $G^1(\mathrm{d}z)=0$, $G^2(\mathrm{d}z)=\bar\lambda\delta_1(\mathrm{d}z)$, $b=\theta_2=q_{12}=q_{22}=0$, $\gamma_2(t,z)=\gamma_2(t)$, and $q_{24}(t,z)=q_{24}(t)$, where $\bar\lambda>0$, and $\delta_1$ is the unit point mass at $1$. Assume that the mean rate of return $\mu(t,x)=\mu_0(t)$ for some ${\cal G}_t^1$-adapted c\'{a}gl\'{a}d process $\mu_0(t)$. Put $\iota_t=\frac{\mu_0(t)-r_t}{\sigma_t}$, $\tilde\phi_1(t)=\iota_t+\phi_1(t)$, and $\tilde\phi_4(t)=\frac{\lambda_t-a_t}{\int_{\mathbb{R}_0}\gamma_2(t)G_{\cal H}^2(t,\mathrm{d}z)}-\frac{\int_{\mathbb{R}_0}(G^2_{\cal H}(t,\mathrm{d}z)-G^2(\mathrm{d}z))}{\int_{\mathbb{R}_0}G_{\cal H}^2(t,\mathrm{d}z)}$.

Assume further the penalty function $g$ verifies
\begin{equation}
\int_0^T |\tilde\phi_1(t)|^2 \mathrm{d}t+\int_0^T\int_{\mathbb{R}_0}\left(|\ln(1+\tilde\phi_4(t))|^k+ |\tilde \phi_4(t)|^k\right)  G_{\mathcal H}^2(t,\mathrm{d}z)\mathrm{d}t<\infty,\quad k=1,2.
\end{equation}
 by the Girsanov theorem. Then we have
 \begin{equation}\label{jump_pen}\begin{aligned}
g(s,v)=\frac{1}{2}\theta_1^2+ \int_{\mathbb{R}_0}\left[ (1+\theta_4)\ln (1+\theta_4)-\theta_4  \right]G^2_{\cal H}(s,\mathrm{d}x).
  \end{aligned}\end{equation}   
  
  We make the following assumption before our procedure.

\begin{assumption}\label{extraa2}
Suppose the coefficients satisfy the following integrability
\begin{equation}
\int_0^T |\tilde\phi_1(t)|^2 \mathrm{d}t+\int_0^T\int_{\mathbb{R}_0} |\tilde \phi_4(t)|^2  G_{\cal H}^2(t,\mathrm{d}z)\mathrm{d}t<\infty.
\end{equation}
Moreover, $\tilde\phi_4(t)>-1$ is ${\cal H}_t$-predictable.
\end{assumption}

  By the Hamiltonian system (\ref{half_equ3_0})-(\ref{half_equ3}) in Theorem \ref{mainth4} and the similar procedure in Section \ref{sec_exampleg_sub1}, we have
  \begin{equation}\label{anequaboutc22J}\begin{aligned}
 & \ln\varepsilon^{v^*}_T+U(X^{u^*}_T)=c^*_2,\\
 & \varepsilon^{v^*}_t=\mathbb{E}\left[ \left(\mathbb{E}\left[  e^{-U(X^{u^*}_T)}\big |{\cal H}_0\right] e^{U(X^{u^*}_T)} \right)^{-1} \big{|}{\cal H}_t\right],\\
  & \varepsilon^{v^*}_T=\left(\mathbb{E}\left[ e^{-U(X^{u^*}_T)}\big| {\cal H}_0  \right] e^{U(X^{u^*}_T)} \right)^{-1},
 \end{aligned} \end{equation}
  where $c^*_2=p_2^*(0)$ is an ${\cal H}_0$-measurable random variable. We also have
   \begin{equation}\label{dffunctionofu2222J}\begin{aligned}
 X^{u^*}_T=  I\left( \frac{c_1^*\Pi_{J}^*(0,T)}{ \varepsilon^{v^*}_T} \right)=\tilde I\left( c_3^* \Pi_{J}^*(0,T) \right),
 \end{aligned} \end{equation}
where $c_1^*=p_1^*(0)$ and $c^*_3=c_1^*\mathbb{E}\left[e^{-U(X^{u^*}_T)}\big| {\cal H}_0 \right]$ are all ${\cal H}_0$-measurable random variables, and $\Pi_J^*(0,T)$ is defined by
  \begin{equation}\label{Pi_jump}\begin{aligned}
\Pi_J^*(0,t):= \exp \Bigg\{&-\int_0^t r_s\mathrm{d}s-\int_0^t\tilde\phi_1(s)\mathrm{d}W^1_{\cal H}(s)-\frac{1}{2}\int_0^t\tilde\phi_1(s)^2\mathrm{d}s 
+\int_0^t\int_{\mathbb{R}_0}\ln(1+\tilde\phi_4(s))\tilde N_{\cal H}^2(\mathrm{d}s,\mathrm{d}z) \\&+\int_0^t\int_{\mathbb{R}_0}\left(\ln(1+\tilde\phi_4(s))-\tilde\phi_4(s)\right) G_{\cal H}^2( s,\mathrm{d}z)\mathrm{d}s \Bigg\}.
 \end{aligned} \end{equation}

 Put $z^*_t=(z^*_1(t),z_4^*(t))=\left(\sigma_t\pi^*_tX^{u^*}_{t-} ,-\kappa^*_t\gamma_2(t)X^{u^*}_{t-} \right)$. Then we have
   \begin{equation}\label{u_jump}\begin{aligned}
   & \pi_t^*=\frac{z^*_1(t)}{\sigma_tX^{u^*}_{t-}},\\
   &  \kappa^*_t=-\frac{z_4^*(t)}{\gamma_2(t)X_{t-}^{u^*}}.
  \end{aligned} \end{equation}
 Then SDE (\ref{wealthsde2}) leads to the following linear BSDE
  \begin{eqnarray}\label{final_gex_1_pj}
  \left\{ \begin{aligned}&\mathrm{d}X^{u^*}_t=-f_{\text{LJ}}(t,X^{u^*}_t,z^*_t,\omega)\mathrm{d}t+z^*_1(t)\mathrm{d}W^1_{\cal H}(t)+\int_{\mathbb{R}_0}z_4^*(t)\tilde N_{\cal H}^2(\mathrm{d}t,\mathrm{d}z), \quad 0\le t\le T,
  \\& X^{u^*}_T= \tilde I\left( c_3^* \Pi_{J}^*(0,T) \right),\end{aligned} \right.   
 \end{eqnarray}
where the generator $f_{\text{LJ}}:[0,T]\times \mathbb{R}\times \mathbb{R}^2\times\Omega\rightarrow \mathbb{R}$  is given by
 \begin{equation*} 
\begin{aligned}
f_{\text{LJ}}(t,x,z,\omega)=-r_tx-\tilde\phi_1(t)z_1+\int_{\mathbb{R}_0}\tilde\phi_4(t)z_4G^2_{\cal H}(t,\mathrm{d}z).
\end{aligned} 
 \end{equation*}
 
 Moreover, when the following condition holds,
 \begin{equation}\label{integrationofJ}\begin{aligned}
 \mathbb{E}\Bigg[&  \left(  \int_0^T\Pi^*_J(0,t)^2\left(  z_1^*(t)-\tilde\phi_1(t) X^{u^*}_t \right)^2\mathrm{d}t \right)^{\frac{1}{2}}\\
&+  \left(  \int_0^T\int_{\mathbb{R}_0}\Pi^*_J(0,t)^2\left(  z_4^*(t)+\tilde\phi_4(t) X^{u^*}_t +z^*_4(t)\tilde \phi_4(t)\right)^2N^2(\mathrm{d}t,\mathrm{d}z ) \right)^{\frac{1}{2}}+(X_T^{u^*})^2\Bigg]<\infty,
 \end{aligned}\end{equation}
we have
\begin{equation}\label{XofJ}
X^{u^*}_t=\mathbb{E}\left[  \Pi^*_J(t,T)\tilde I(c_3^*\Pi^*_J(0,T)) | {\cal H}_t  \right],
\end{equation}
where $\Pi^*_J(t,T):=\Pi^*_J(0,T)/\Pi^*_J(0,t)$. The ${\cal H}_0$-random variable $c^*_3$ can be determined by
\begin{equation}\label{c3ofJ}
X_0=\mathbb{E}\left[  \Pi^*_J(0,T)\tilde I(c_3^*\Pi^*_J(0,T)) | {\cal H}_0 \right].
\end{equation}

 Substituting (\ref{dffunctionofu2222J}) into (\ref{anequaboutc22J}) we obtain
 \begin{equation}  \label{anequaboutc22_finalJ}
  \begin{aligned}
&  \varepsilon^{v^*}_t=\mathbb{E}\left[ \left(\mathbb{E}\left[   e^{ -U\left(\tilde I\left(  c_3^*\Pi^*_J(0,T)   \right) \right )   }\big |{\cal H}_0\right]   e^{ U\left(\tilde I\left(  c_3^*\Pi^*_J(0,T)   \right) \right )   }  \right)^{-1} \big{|}{\cal H}_t\right],\\
&  \varepsilon^{v^*}_T=\left(\mathbb{E}\left[   e^{ -U\left(\tilde I\left(  c_3^*\Pi^*_J(0,T)   \right) \right )   }\big |{\cal H}_0\right]   e^{ U\left(\tilde I\left(  c_3^*\Pi^*_J(0,T)   \right) \right )   }  \right)^{-1}.
   \end{aligned}
 \end{equation}

To sum up, we have the following result for the robust optimal strategy with jumps.
  
      \begin{theorem} \label{ex_mainth_idontknow_j}
Assume that $G^1(\mathrm{d}z)=0$, $G^2(\mathrm{d}z)=\bar\lambda\delta_1(\mathrm{d}z)$, $b=\theta_2=q_{12}=q_{22}=0$, $\gamma_2(t,z)=\gamma_2(t)$, $q_{24}(t,z)=q_{24}(t)$, $\mu(t,x)=\mu_0(t) $ for some ${\cal G}_t^1$-adapted c\'{a}gl\'{a}d process $\mu_0(t)$, and $g(s,v)$ is given by (\ref{jump_pen}). Suppose $(u^*,v^*)\in {\cal A}_1'\times{\cal A}_2'$ is optimal for Problem \ref{sdg} under the conditions of Theorem \ref{mainth4}, and Assumption \ref{extraa2} and the integrability condition (\ref{integrationofJ}) hold. Then $u^*$, $v^*$, $X^{u^*}_t$ and $\varepsilon^{v^*}_t$ are given by (\ref{u_jump}), (\ref{halfequ_1})-(\ref{halfequ_2}), (\ref{XofJ}), (\ref{anequaboutc22_finalJ}), respectively, where $\Pi^*_J$ is given by (\ref{Pi_jump}), $c^*_3$ is determined by (\ref{c3ofJ}), and $(X^{u^*},  z^*_1, z^*_4)$ solves the linear BSDE (\ref{final_gex_1_pj}).
 \end{theorem}
 
  \begin{remark}
  Similar to Theorem \ref{ex_mainth_idontknow}, the existence and uniqueness of the linear BSDE (\ref{final_gex_1_pj}) hold under some mild conditions. We refer to \cite{Oksendal19,Eyraoud05,Li06,Lu13,Wang07} for more details.
  \end{remark}
  
  When the utility function is of the logarithmic form, i.e., $U(x)=\ln x $, the fixed point function is given by (\ref{fixedpln}), i.e.,
\begin{equation}
\tilde I(y)=\frac{1}{\sqrt{y}},\quad y>0.
\end{equation}

Combining (\ref{XofJ}) with (\ref{c3ofJ}), we have
\begin{equation} \begin{aligned}\label{linearbsde_bdg_2_iitj}
 X^{u^*}_t=\frac{X_0\mathbb{E} \left[   \sqrt{  \Pi^*_J(t,T)}  | {\cal H}_t\right]}{\mathbb{E} \left[  \sqrt{  \Pi^*_J(0,T)} | {\cal H}_0\right]\sqrt{  \Pi^*_J(0,t)}},
\end{aligned}\end{equation}
and
 \begin{equation}\label{linearbsde_bdg_2_iiTj}
 \begin{aligned}
 X^{\pi^*}_T=\frac{X_0 }{\mathbb{E} \big[  \sqrt{  \Pi^*_J(0,T)} | {\cal H}_0\big]\sqrt{  \Pi^*_J(0,T)}}.
\end{aligned}
\end{equation}
Thus the BSDE (\ref{final_gex_1_pj}) can be rewritten as
 \begin{eqnarray}\label{final_gex_1fj}
  \left\{ \begin{aligned}&\mathrm{d}X^{u^*}_t=-f_{\text{LJ}}(t,X^{u^*}_t,z^*_t,\omega)\mathrm{d}t+z^*_1(t)\mathrm{d}W^1_{\cal H}(t)+\int_{\mathbb{R}_0}z_4^*(t)\tilde N_{\cal H}^2(\mathrm{d}t,\mathrm{d}z), \quad 0\le t\le T,
  \\& X^{u^*}_T=\frac{X_0 }{\mathbb{E} \big[  \sqrt{  \Pi_J^*(0,T)} | {\cal H}_0\big]\sqrt{  \Pi_J^*(0,T)}}.\end{aligned} \right.   
 \end{eqnarray}
 where $f_{\text{LJ}} $  is given by
 \begin{equation*} 
\begin{aligned}
f_{\text{LJ}}(t,x,z,\omega)=-r_tx-\tilde\phi_1(t)z_1+\int_{\mathbb{R}_0}\tilde\phi_4(t)z_4G^2_{\cal H}(t,\mathrm{d}z).
\end{aligned} 
 \end{equation*}

We can also calculate the value of Problem \ref{sdg} as follows
\begin{equation}  \label{dontknowvj}
  \begin{aligned}
 V =\ln X_0-2\mathbb{E}\left[ \ln \mathbb{E}\big[ \sqrt{\Pi^*_J(0,T)}  |{\cal H}_0\big] \right].
   \end{aligned}
 \end{equation}
 
     \begin{corollary} \label{corofvaluej}
Assume that $U(x)=\ln x$, $G^1(\mathrm{d}z)=0$, $G^2(\mathrm{d}z)=\bar\lambda\delta_1(\mathrm{d}z)$, $b=\theta_2=q_{12}=q_{22}=0$, $\gamma_2(t,z)=\gamma_2(t)$, $q_{24}(t,z)=q_{24}(t)$, $\mu(t,x)=\mu_0(t) $ for some ${\cal G}_t^1$-adapted c\'{a}gl\'{a}d process $\mu_0(t)$, and $g(s,v)$ is given by (\ref{jump_pen}). Suppose $(u^*,v^*)\in {\cal A}_1'\times{\cal A}_2'$ is optimal for Problem \ref{sdg} under the conditions of Theorem \ref{mainth4}, and Assumption \ref{extraa2} and the integrability condition (\ref{linearbsde_bdg}) hold. Then $(u^*,v^*)$ and $V$ are given by (\ref{u_jump}), (\ref{halfequ_1})-(\ref{halfequ_2}) and (\ref{dontknowvj}), respectively, where $\Pi^*_J$ is given by (\ref{Pi_jump}), and $(X^{u^*},  z^*_1, z^*_2)$ solves the linear BSDE (\ref{final_gex_1fj}). 
 \end{corollary}

  When the insurer has no insider information, i.e., ${\cal H}_t={\cal F}_t$, we have $\tilde\phi_1(t)=\iota_t$, and $\tilde\phi_4(t)=\frac{\lambda_t-a_t}{\bar\lambda\gamma_2(t)}$. Assume further that all parameter processes are deterministic functions. By similar procedure in Section \ref{sec_exampleg_sub1}, we have the following corollary.

      \begin{corollary} \label{corex_mainth_idontknow_donskerj}
Assume that $U(x)=\ln x$,  $U(x)=\ln x$, $G^1(\mathrm{d}z)=0$, $G^2(\mathrm{d}z)=\bar\lambda\delta_1(\mathrm{d}z)$, $b=\theta_2=q_{12}=q_{22}=0$, $\gamma_2(t,z)=\gamma_2(t)$, $q_{24}(t,z)=q_{24}(t)$, $\mu(t,x)=\mu_0(t) $ for some c\'{a}gl\'{a}d function $\mu_0(t)$, and $g(s,v)$ is given by (\ref{jump_pen}). Assume further that ${\cal H}_t={\cal F}_t$ and all parameter processes are deterministic functions. Suppose $(u^*,v^*)\in {\cal A}_1'\times{\cal A}_2'$ is optimal for Problem \ref{sdg} under the conditions of Theorem \ref{mainth4}. Then $(u^*,v^*)$ is given by 
\begin{equation}\label{dknowj}\begin{aligned}
& \pi^*_t=\frac{\iota_t}{2\sigma_t},
\\& \kappa^*_t=\frac{\lambda_t-a_t}{(\lambda_t-a_t)\gamma_2(t)+\bar\lambda \gamma_2^2(t)+\sqrt{\bar\lambda(\lambda_t-a_t)\gamma^3_2(t)+\bar\lambda^2\gamma^4_2(t)}},
\\& \theta_1^*(t)=-\frac{\iota_t}{2},
\\& \theta_4^*(t)=\frac{\lambda_t-a_t}{\bar\lambda \gamma_2(t)}\left( 1-\kappa^*_t\gamma_2(t) \right)-\kappa^*_t\gamma_2(t).
\end{aligned}\end{equation}
\end{corollary}

 \begin{remark}
Different from the continuous case in Section \ref{sec_exampleg_sub1}, the value $V$ of Problem \ref{sdg} is hard to be calculated analytically by using the Girsanov theorem directly when jumps are considered.
 \end{remark} 
 
 When there is no model uncertainty, we can also obtain the following proposition.
     \begin{proposition}  \label{co2rex_mainth_idontknow_donskerj}
Assume that $U(x)=\ln x$,  $U(x)=\ln x$, $G^1(\mathrm{d}z)=0$, $G^2(\mathrm{d}z)=\bar\lambda\delta_1(\mathrm{d}z)$, $b=\theta_2=q_{12}=q_{22}=0$, $\gamma_2(t,z)=\gamma_2(t)$, $q_{24}(t,z)=q_{24}(t)$, $\mu(t,x)=\mu_0(t) $ for some c\'{a}gl\'{a}d function $\mu_0(t)$, $g(s,v)$ is given by (\ref{jump_pen}), and no model uncertainty is considered. Assume further that ${\cal H}_t={\cal F}_t$ and all parameter processes are deterministic functions. Suppose $u^*\in {\cal A}_1'$ is optimal for Problem \ref{sdg2fb} under the conditions of Theorem \ref{mainth4} (with ${\cal A}_2'=\{(0,0)\}$). Then $u^* $ is given by
\begin{equation}\label{dknowj2}\begin{aligned}
& \pi^*_t=\frac{\iota_t}{\sigma_t},
\\& \kappa^*_t=\frac{\lambda_t-a_t}{(\lambda_t-a_t)\gamma_2(t)+\bar\lambda\gamma_2^2(t)}.
\end{aligned}\end{equation}
 \end{proposition}

\subsubsection{A particular case}
\label{subsubconti2}
Next, we give a particular case. Assume that $U(x)=\ln x $, the insider information filtration is given by 
\begin{equation}\label{donsker_filtration_j}
{\cal H}_t=\bigcap_{s>t}({\cal F}_s\vee Y_0):=\bigcap_{s>t}\left({\cal F}_s\vee \int_0^{T_0}\varphi({s'})\mathrm{d}W^1_{s'}\right),\quad 0\le t\le T,
\end{equation}
for some $T_0>T$, and all the parameter processes are assumed to be deterministic functions.  Here, $\varphi_t $ is some deterministic function satisfying $\Vert \varphi\Vert^2_{[s,t]}:=\int_s^{t}\varphi^2(s')    \mathrm{d}s'<\infty$ for all $0\le s\le t\le T_0$, and $\Vert \varphi\Vert^2_{[T,T_0]}>0$. 

By the Donsker $\delta$ functional $\delta_y(Y_0)$ and similar procedure in Section \ref{subsubconti1}, we have
\begin{equation}  
\begin{aligned}
&\mathbb{E}[\delta_y(Y_0) |{\cal F}_t]=\frac{1}{\sqrt{2\pi\Vert \varphi\Vert_{[t,T_0]}^2}}\exp\Bigg\{-\frac{(y-\int_0^t\varphi_s\mathrm{d} W^1_s)^{ 2}}{2\Vert\varphi \Vert_{[t,T_0]}^2}  \Bigg\},\quad 0\le t\le T,
 \end{aligned}\end{equation}
 and
 \begin{equation} \label{weknowphi1_j}
\begin{aligned}
&\phi_1(t)=\phi_1(t,y)|_{y=Y_0}=\frac{y-\int_0^t\varphi_s\mathrm{d} W^1_s}{\Vert \varphi\Vert^2_{[t,T_0]}}\varphi_t\Bigg{|}_{y=Y_0},
\\&  G^2_{\cal H}(t,\mathrm{d}z)=G^2(\mathrm{d}z)=\bar\lambda\delta_1(\mathrm{d}z).
  \end{aligned}\end{equation}

Substituting (\ref{weknowphi1_j}) into (\ref{Pi_jump}) and using the It\^{o} formula we can rewrite the expression of $\Pi_J^*(0,t)$ as follows
    \begin{equation} \label{redefinitionofPi_j}
  \begin{aligned}
\Pi^*_J(0,t)&=\Pi^*_J(0,t,y)|_{y=Y_0}\\
&= \exp\left\{-\int_0^t\phi_1(s,Y_0)\mathrm{d}W^1_s +\frac{1}{2}\int_0^t \phi_1^2(s,Y_0) \mathrm{d}y   \right\}\Pi^*_{J,a}(0,t) \\
&=\mathbb{E}[\delta_y(Y_0)|{\cal F}_0]\frac{\Pi^*_a(0,t)}{\mathbb{E}[\delta_y(Y_0)|{\cal F}_t]}\big|_{y=Y_0},
  \end{aligned}
 \end{equation}
where
    \begin{equation}  \label{pi_adef_j}
  \begin{aligned}
\Pi^*_{J,a}(0,t):= \exp\Bigg\{&  -\int_{0}^{t} r_s  \mathrm{d}s
  -\int_{0}^{t} \iota_s  \mathrm{d}W^1_s
   -\frac{1}{2}\int_{0}^{t}    \iota_s ^2 \mathrm{d}s+\int_0^t\int_{\mathbb{R}_0}\ln(1+\frac{\lambda_s-a_s}{\bar\lambda \gamma_2(s)}) \tilde N^2(\mathrm{d}s,\mathrm{d}z)\\&+\int_0^t\int_{\mathbb{R}_0} \left( \ln(1+\frac{\lambda_s-a_s}{\bar\lambda \gamma_2(s)})- \frac{\lambda_s-a_s}{\bar\lambda \gamma_2(s)} \right)G^2(\mathrm{d}z)\mathrm{d}s  \Bigg\}
  \end{aligned}
 \end{equation}
 is an ${\cal F}_t$-adapted process.
  
By similar procedure in Section \ref{subsubconti1}, the BSDE (\ref{final_gex_1fj}) is equivalent to the following classical linear BSDE with respect to the filtration $\{{\cal F}_t\}_{0\le t\le T}$
 \begin{eqnarray}\label{final_ex_1dons_j}
  \left\{ \begin{aligned}&\mathrm{d}X^{u^*}_t(y)=-\bar f_{\text{LJ}}(t, X^{u^*}_t(y),z^*_t(y))\mathrm{d}t+ z^*_1(t,y)\mathrm{d}W^1_t+\int_{\mathbb{R}_0}  z_4^*(t,y) \tilde N^2(\mathrm{d}t,\mathrm{d}z) , \quad 0\le t\le T,
  \\&X^{u^*}_T(y)=\tilde c^*_3(y)  \sqrt{\frac{ \mathbb{E}[\delta_y(Y_0)|{\cal F}_T]}{    \Pi^*_{J,a}(0,T)}},\end{aligned} \right.   
 \end{eqnarray}
where the generator $\bar f_{\text{LJ}}:[0,T] \times\mathbb{R}\times \mathbb{R}^2 \rightarrow \mathbb{R}$  is given by
  \begin{equation*} 
\begin{aligned}
&\bar f_{\text{LJ}}(t,x,z )= -r_tx-\iota_tz_1+\int_{\mathbb{R}_0} \frac{\lambda_t-a_t}{\bar\lambda\gamma_2(t)}z_4G^2(\mathrm{d}z),
\end{aligned} 
 \end{equation*}
 and $\tilde c^*_3(y):=\frac{X_0}{\mathbb{E}\big[\sqrt{\Pi^*_J(0,T)} | {\cal H}_0\big](y) \sqrt{\mathbb{E}[\delta_y(Y_0)|{\cal F}_0]}  }$.

By \cite[Theorem 4.8]{Oksendal19}, the unique strong solution of (\ref{final_ex_1dons_j}) is given by
  \begin{equation}\label{linearbsde_bdg_2_donskerpj}
 \begin{aligned}
 X^{u^*}_t(y)= \mathbb{E} \left[   \Pi_{J,a}^*(t,T) \tilde c^*_3(y)\sqrt{\frac{ \mathbb{E}[\delta_y(Y_0)|{\cal F}_T]}{     \Pi^*_{J,a}(0,T)}} \bigg | {\cal F}_t\right],
\end{aligned}
\end{equation}
where $\Pi_{J,a}^*[t,T]:=\Pi_{J,a}^*[0,T]/\Pi_{J,a}^*[0,t]$. By (\ref{linearbsde_bdg_2_donsker}) and the initial value condition $X^{u^*}_0(y)=X_0$, the Borel measurable function $\tilde c^*_3(y)$ is given by
 \begin{equation} \label{c3determined_donskerj}
 \begin{aligned}
\tilde c^*_3(y) =\frac{X_0}{\mathbb{E}     \sqrt{  \Pi_{J,a}^*(0,T) \mathbb{E}[\delta_y(Y_0)|{\cal F}_T]} }=\frac{X_0}{\mathbb{E}\big[\sqrt{\Pi^*_J(0,T)} | {\cal H}_0\big](y) \sqrt{\mathbb{E}[\delta_y(Y_0)|{\cal F}_0]}  }.
\end{aligned}
\end{equation}
Substituting (\ref{c3determined_donskerj}) into (\ref{linearbsde_bdg_2_donskerpj}) we obtain
   \begin{equation}\label{linearbsde_bdg_2_donsker_2j}
 \begin{aligned}
 X^{u^*}_t(y)= \frac{X_0\mathbb{E}\left[ \sqrt{\Pi^*_{J,a}(t,T)\mathbb{E}[\delta_y(Y_0)|{\cal F}_T]  } |{\cal F}_t   \right]}{  \mathbb{E}\sqrt{\Pi^*_{J,a}(0,T)\mathbb{E}[\delta_y(Y_0)|{\cal F}_T]}\sqrt{\Pi^*_{J,a}(0,t)}}.
\end{aligned}
\end{equation}
By \cite[Proposition 3.5.1]{Delong13} and Malliavin calculus, we have
   \begin{equation} \label{iknow1j}
 \begin{aligned}
z_1^*(t,y)&=\frac{1}{2}\frac{X_0\mathbb{E}\left[ \sqrt{ \Pi^*_{J,a}(t,T)\mathbb{E}[\delta_y(Y_0)|{\cal F}_T]   } \left(y-\int_0^T\varphi_s\mathrm{d}W^1_s\right)   \big|{\cal F}_t   \right]}{ \Vert \varphi\Vert^2_{[T,T_0]} \mathbb{E}\sqrt{\Pi^*_{J,a}(0,T)\mathbb{E}[\delta_y(Y_0)|{\cal F}_T]}\sqrt{\Pi^*_{J,a}(0,t)}}\varphi_t +\frac{1}{2} \frac{X_0\mathbb{E}\left[ \sqrt{\Pi^*_{J,a}(t,T)\mathbb{E}[\delta_y(Y_0)|{\cal F}_T]  } |{\cal F}_t   \right]}{  \mathbb{E}\sqrt{\Pi^*_{J,a}(0,T)\mathbb{E}[\delta_y(Y_0)|{\cal F}_T]}\sqrt{\Pi^*_{J,a}(0,t)}}\iota_t,
\end{aligned}
\end{equation}
and 
   \begin{equation} \label{iknow2j}
 \begin{aligned}
z_4^*(t,y)&= -X^{u^*}_t(y)\frac{\lambda_t-a_t}{(\lambda_t-a_t) +\bar\lambda \gamma_2(t)+\sqrt{\bar\lambda(\lambda_t-a_t)\gamma_2(t)+\bar\lambda^2\gamma^2_2(t)}}.
\end{aligned}
\end{equation}
Substituting (\ref{iknow1j}) and (\ref{iknow2j}) into (\ref{u_jump}) we have
  \begin{equation} \label{iknowuj}
 \begin{aligned}
\pi_t^*&=\frac{1}{2}\frac{ \mathbb{E}\left[ \sqrt{ \Pi^*_{J,a}(t,T)\mathbb{E}[\delta_y(Y_0)|{\cal F}_T]   } \left(y-\int_0^T\varphi_s\mathrm{d}W^1_s\right)   \big|{\cal F}_t   \right]}{\sigma_t  \Vert \varphi\Vert^2_{[T,T_0]} \mathbb{E}\left[ \sqrt{\Pi^*_{J,a}(t,T)\mathbb{E}[\delta_y(Y_0)|{\cal F}_T]  } |{\cal F}_t   \right]}\varphi_t\Bigg{|}_{y=Y_0}
+\frac{\iota_t}{2\sigma_t},
\\ \kappa^*_t&=\frac{\lambda_t-a_t}{(\lambda_t-a_t)\gamma_2(t)+\bar\lambda \gamma_2^2(t)+\sqrt{\bar\lambda(\lambda_t-a_t)\gamma^3_2(t)+\bar\lambda^2\gamma^4_2(t)}}.
\end{aligned}
\end{equation}

By Corollary \ref{mainth6}, we have
 \begin{equation}\label{bsde_barzt1_donskthetaj}
   \begin{aligned}
 &\theta_1^*(t)=-\frac{\iota_t}{2}+\frac{1}{2}\frac{ \mathbb{E}\left[ \sqrt{ \Pi^*_{J,a}(t,T)\mathbb{E}[\delta_y(Y_0)|{\cal F}_T]   } \left(y-\int_0^T\varphi_s\mathrm{d}W^1_s\right)   \big|{\cal F}_t   \right]}{ \Vert \varphi\Vert^2_{[T,T_0]}  \mathbb{E}\left[ \sqrt{\Pi^*_{J,a}(t,T)\mathbb{E}[\delta_y(Y_0)|{\cal F}_T]  } |{\cal F}_t   \right]}\varphi_t\Bigg{|}_{y=Y_0} -\frac{y-\int_0^t\varphi_s\mathrm{d}W^1_s}{\Vert\varphi \Vert_{[t,T_0]}^2}\varphi_t\Bigg{|}_{y=Y_0},
  \\&\theta_4^*(t)= \frac{\lambda_t-a_t}{\bar\lambda \gamma_2(t)}\left( 1-\kappa^*_t\gamma_2(t) \right)-\kappa^*_t\gamma_2(t). 
  \end{aligned} 
 \end{equation}

To sum up, we give the following theorem.

   \begin{theorem} \label{ex_mainth_idontknow_donskerj}
Assume that $U(x)=\ln x$,  $U(x)=\ln x$, $G^1(\mathrm{d}z)=0$, $G^2(\mathrm{d}z)=\bar\lambda\delta_1(\mathrm{d}z)$, $b=\theta_2=q_{12}=q_{22}=0$, $\gamma_2(t,z)=\gamma_2(t)$, $q_{24}(t,z)=q_{24}(t)$, $\mu(t,x)=\mu_0(t) $ for some c\'{a}gl\'{a}d function $\mu_0(t)$, and $g(s,v)$ is given by (\ref{jump_pen}). Assume further that $\{{\cal H}_t \}_{0\le t\le T}$ is given by (\ref{donsker_filtration_j}) and all parameter processes are deterministic functions. Suppose $(u^*,v^*)\in {\cal A}_1'\times{\cal A}_2'$ is optimal for Problem \ref{sdg} under the conditions of Theorem \ref{mainth4}. Then $(u^*,v^*)$ is given by (\ref{iknowuj}) and (\ref{bsde_barzt1_donskthetaj}), where $\Pi_{J,a}^*$ is given by (\ref{pi_adef_j}).
 \end{theorem}

\begin{remark}
When the insider information is about the future value of the risk in the insurance market, i.e., ${\mathcal H}_t=\bigcap_{s>t}\left({\mathcal F}_s\wedge \eta^2_{T_0} \right)$, we have by \cite{Draouil16} that
\begin{equation*}\begin{aligned}
    &\mathbb{E}[\delta_y(\eta^2_{T_0})|{\mathcal{F}_t}]=\frac{1}{2\pi}\int_{\mathbb{R}}\exp\left\{ ix\eta^2_t+\bar\lambda(T_0-t)\left(e^{ix}-1-ix \right)-ixy  \right\}\mathrm{d}x,\\
    &\phi_1(t)=0 ,\\
    & G^2_{\mathcal{H}}(t,\mathrm{d}z)=G^2(\mathrm{d}z)+\frac{1}{T_0-t}\int_t^{T_0} \tilde N^2(\mathrm{d}r,\mathrm{d}z)  .
\end{aligned}\end{equation*} 
By similar procedure, we obtain the robust optimal strategy as follows
  \begin{equation}  
 \begin{aligned}
\pi_t^*=& \frac{\iota_t}{2\sigma_t},
\\ \kappa^*_t=& \frac{\lambda_t-a_t}{(\lambda_t-a_t)\gamma_2(t)+\bar\lambda \gamma_2^2(t)+\sqrt{\bar\lambda(\lambda_t-a_t)\gamma^3_2(t)+\bar\lambda^2\gamma^4_2(t)}}\\
&-\frac{\sqrt{\bar\lambda}\mathbb{E}\left[ \sqrt{\Pi^*_{J,a}(t,T)} Q(y)|{\mathcal{F}}_t \right] }{\sqrt{(\lambda_t-a_t)\gamma_2(t)+\bar \lambda\gamma_2^2(t)}  
\mathbb{E}\left[ \sqrt{\Pi^*_{J,a}(t,T)\mathbb{E}[ \delta_y(\eta^2_{T_0})  |{\mathcal{F}}_T]} |{\mathcal{F}}_t \right]}\Bigg{|}_{y=\eta^2_{T_0}},
\end{aligned}
\end{equation}
where $\Pi^*_{J,a}$ is given by (\ref{pi_adef_j}), and $Q(y)$ is given by
 \begin{equation*}  
 \begin{aligned}
Q(y) =&\sqrt{ \frac{1}{2\pi}\int_{\mathbb{R}}\exp\left\{ ix\eta^2_T+\bar\lambda(T_0-T)\left(e^{ix}-1-ix \right)-ixy  \right\}e^{ix}\mathrm{d}x  }\\
 &-
 \sqrt{ \frac{1}{2\pi}\int_{\mathbb{R}}\exp\left\{ ix\eta^2_T+\bar\lambda(T_0-T)\left(e^{ix}-1-ix \right)-ixy  \right\}\mathrm{d}x  }
 \end{aligned}
\end{equation*}
\end{remark}

Next, we concentrate on the special situation without model uncertainty, i.e., Problem \ref{sdg2fb}. Then the terminal value condition of $X^{u^*}_t$ in (\ref{final_gex_1_donsker2}) is replaced by
\begin{equation}
X^{u^*}_T=\frac{X_0}{\Pi_J^*(0,T)}=\frac{X_0 \mathbb{E}[\delta_y(Y_0)|{\cal F}_T]}{ \mathbb{E}[\delta_y(Y_0)|{\cal F}_0]\Pi^*_{J,a}(0,T)}\bigg{|}_{y=Y_0}.
\end{equation}
By similar procedure, we obtain the following proposition.

   \begin{proposition} \label{pex_mainth_idontknow_donskercj}
Assume that $U(x)=\ln x$,  $U(x)=\ln x$, $G^1(\mathrm{d}z)=0$, $G^2(\mathrm{d}z)=\bar\lambda\delta_1(\mathrm{d}z)$, $b=\theta_2=q_{12}=q_{22}=0$, $\gamma_2(t,z)=\gamma_2(t)$, $q_{24}(t,z)=q_{24}(t)$, $\mu(t,x)=\mu_0(t) $ for some c\'{a}gl\'{a}d function $\mu_0(t)$, $g(s,v)$ is given by (\ref{jump_pen}), and no model certainty is considered. Assume further that $\{{\cal H}_t \}_{0\le t\le T}$ is given by (\ref{donsker_filtration_j}) and all parameter processes are deterministic functions. Suppose $u^* \in {\cal A}_1' $ is optimal for Problem \ref{sdg2fb} under the conditions of Theorem \ref{mainth4}. Then $u^* $ is given by  
\begin{equation} \label{valuenoj}
   \begin{aligned}
&\pi^*_t=\frac{\iota_t}{\sigma_t}+\frac{Y_0-\int_0^t\varphi_s\mathrm{d}W^1_s}{\sigma_t\Vert\varphi\Vert^2_{[t,T_0]}}\varphi_t  ,\\
&\kappa^*_t  = \frac{\lambda_t-a_t}{(\lambda_t-a_t)\gamma_2(t)+\bar\lambda \gamma_2^2(t)}  . 
 \end{aligned} 
 \end{equation}
 \end{proposition}

\begin{remark}
When the insider information is about the future value of the risk in the insurance market, i.e., ${\mathcal H}_t=\bigcap_{s>t}\left({\mathcal F}_s\wedge  \eta^2_{T_0} \right)$, the optimal strategy can be calculated by
\begin{equation}  
   \begin{aligned}
&\pi^*_t=\frac{\iota_t}{\sigma_t}  ,\\
&\kappa^*_t  = \frac{\lambda_t-a_t}{(\lambda_t-a_t)\gamma_2(t)+\bar\lambda \gamma_2^2(t)} - \frac{ \eta^2_{T_0}-\eta^2_{t}}{(T_0-t)\left[(\lambda_t-a_t)+\bar\lambda \gamma_2(t)\right]} . 
 \end{aligned} 
 \end{equation}
 \end{remark}

\section{The large insurer case: combined method}
\label{sec_example1}
Now we consider the special case when the utility function of the insurer is characterized by the logarithmic utility, i.e., $U(x)=\ln(x)$, and the insurer is `large', which means that the mean rate of return $\mu$ on the risky asset could be influenced by her investment strategy $\pi$. 
 
In this situation, the degree of the influence on $\mu$ could not change the robust optimal insurance strategy $\kappa^*$ if the correlation coefficient $\rho=0$ intuitively. In fact, this can be verified by the equation (\ref{reduceth6_2_2}) below when the insurer is not ambiguity averse. Thus we will always suppose $b\neq 0$ when the insurer is `large'. For simplicity, we only consider the continuous case in this section, that is, $G^1(\mathrm{d}z)=G^2(\mathrm{d}z)=0$.  
 
  Just as in Section \ref{sec_exampleg_sub1}, we assume that the mean rate of return $\mu(t,x)=\mu_0(t)+\varrho_tx$ for some ${\cal G}^1_t$-adapted c\`{a}gl\`{a}d processes $\mu_0(t)$ and $\varrho_t$ with $0\le \varrho_t<\frac{1}{2}\sigma_t^2$, and $b\ge \epsilon>0$ for some positive constant $\epsilon$. Put $\iota_t=\frac{\mu_0(t)-r_t}{\sigma_t}$, $\tilde\sigma_t=\sigma_t-\frac{2\varrho_t}{\sigma_t}$, $\tilde \phi_1(t)=\iota_t+\phi_1(t)$, and $\tilde \phi_2(t)=\frac{\lambda_t-a_t+\rho b_t\iota_t}{\sqrt{1-\rho^2}b_t}-\phi_2(t)$. Assume further the penalty function $g$ is given by $g(s,v)=g(v)=\frac{1}{2}(\theta_1^2+\theta_2^2)$. Then we have
 \begin{equation}\label{utilityofe}
 \mathbb{E}\left[  \int_0^T\varepsilon_s^{v^*}  g(v^*_s)\mathrm{d}s \right]
= \mathbb{E}\left[\varepsilon^{v^*}_T  \ln \varepsilon^{v^*}_T  \right].
 \end{equation}
 
In this continuous setting, the equations (\ref{halfequ_21}) and (\ref{halfequ_22}) in Corollary \ref{mainth6} can be reduced to
 \begin{equation}\label{reduceth6} 
 \begin{aligned}
 &\theta_1^*(t)=\tilde \sigma_t\pi^*_t-\rho b_t\kappa^*_t-\tilde\phi_1(t) ,
  \\&\theta_2^*(t)=-\sqrt{1-\rho^2}b_t\kappa^*_t+\rho\frac{\sigma_t-\tilde\sigma_t}{\sqrt{1-\rho^2}}\pi^*_t+\tilde\phi_2(t). 
  \end{aligned}
 \end{equation}
 
By a similar procedure in Section \ref{sec_exampleg_sub1} with respect to the Hamiltonian system (\ref{half_equ3}), we have (see (\ref{equeg1_42fb}))
 \begin{equation}  \label{equeg1_42} 
  \begin{aligned}
\ln\big( \varepsilon^{v^*}_TX_T^{u^*}\big)=c^*_2,
 \end{aligned}
 \end{equation} 
where $c^*_2=p^*_2(0)$ is an ${\cal H}_0$-measurable random variable.

The It\^{o} formula for It\^{o} integrals combined with the expressions of $\varepsilon_t^{v^*}$ and $X^{u^*}_t$ (see Theorem \ref{novikov}) yields the following SDE:
     \begin{equation}\label{bsde_lt} 
 \begin{aligned}
\mathrm{d}\ln\big(\varepsilon^{v^*}_tX^{u^*}_t\big)=&\Bigg[r_t+(\mu_0(t)-r_t)\pi_t^*+(\lambda_t-a_t)\kappa_t ^*+( \sigma_t\pi_t^*-\rho b_t\kappa_t^*)  \phi_1(t)-\sqrt{1-\rho^2}b_t\kappa_t^* \phi_2(t)  \Bigg]\mathrm{d}t \\
&-\frac{1}{2}\Bigg[  (\sigma_t\pi^*_t)^2 -2\rho\sigma_tb_t\pi^*_t\kappa_t^*+(b_t\kappa^*_t)^2+\theta^*_1(t)^2+ \theta^*_2(t)^2-\sigma_t(\sigma_t-\tilde\sigma_t)(   \pi^*_t)^2 \Bigg]\mathrm{d}t\\
&+\left(\sigma_t\pi^*_t-\rho b_t\kappa^*_t+\theta_1^*(t)\right)\mathrm{d}W^1_{\cal H}(t)-\left(\sqrt{1-\rho^2}b_t\kappa^*_t-\theta_2^*(t)\right)\mathrm{d}W^2_{\cal H}(t).
\end{aligned} 
 \end{equation}
Put $L^*_t=\ln\big(\varepsilon^{v^*}_tX^{u^*}_t\big)$ and $ z^*_t=( z_1^*(t), z^*_2(t))\\=\left(\sigma_t\pi^*_t-\rho b_t\kappa^*_t+\theta_1^*(t), -\sqrt{1-\rho^2}b_t\kappa^*_t+\theta_2^*(t) \right)$. Suppose that $\sigma_t+\tilde\sigma_t-2\rho^2\sigma_t\neq0$. Then by (\ref{reduceth6}) we have
   \begin{equation}\label{bsde_zt1}  
   \begin{aligned}
 &\pi_t^*=\frac{1-\rho^2}{\sigma_t+\tilde\sigma_t-2\rho^2\sigma_t}\left( z_1^*(t)+\tilde\phi_1(t)\right)-\frac{\rho\sqrt{1-\rho^2}}{\sigma_t+\tilde\sigma_t-2\rho^2\sigma_t}\left( z^*_2(t)-\tilde\phi_2(t)\right) ,
  \\&\kappa_t^*=-\frac{\sqrt{1-\rho^2}(\sigma_t+\tilde\sigma_t)}{2(\sigma_t+\tilde\sigma_t-2\rho^2\sigma_t)b_t}\left( z^*_2(t)-\tilde\phi_2(t)\right)+\frac{\rho(\sigma_t-\tilde\sigma_t)}{2(\sigma_t+\tilde\sigma_t-2\rho^2\sigma_t)b_t}\left( z^*_1(t)+\tilde\phi_1(t)\right),   
\end{aligned} 
 \end{equation}
and
   \begin{equation}\label{bsde_zt2} 
\begin{aligned}
     &\theta^*_1(t)=\frac{2\tilde\sigma_t-\rho^2\sigma_t-\rho^2\tilde\sigma_t}{2(\sigma_t+\tilde\sigma_t-2\rho^2\sigma_t)} z_1^*(t) -\frac{2\sigma_t-3\rho^2\sigma_t+\rho^2\tilde\sigma_t}{2(\sigma_t+\tilde\sigma_t-2\rho^2\sigma_t)}\tilde\phi_1(t) +\frac{\rho\sqrt{1-\rho^2}(\sigma_t-\tilde\sigma_t)}{2(\sigma_t+\tilde\sigma_t-2\rho^2\sigma_t)}\left( z_2^*(t)-  \tilde\phi_2(t) \right) ,\\
     \\
  &\theta_2^*(t)=\frac{\sigma_t+\tilde\sigma_t-3\rho^2\sigma_t+\rho^2\tilde\sigma_t}{2(\sigma_t+\tilde\sigma_t-2\rho^2\sigma_t)} z^*_2(t)+\frac{(1-\rho^2)(\sigma_t+\tilde\sigma_t)}{2(\sigma_t+\tilde\sigma_t-2\rho^2\sigma_t)}\tilde\phi_2(t)+\frac{\rho\sqrt{1-\rho^2}(\sigma_t-\tilde\sigma_t)}{2(\sigma_t+\tilde\sigma_t-2\rho^2\sigma_t)}\left( z^*_1(t)+\tilde\phi_1(t)\right). 
\end{aligned} 
 \end{equation}
The SDE (\ref{bsde_lt}) combined with (\ref{bsde_zt1}), (\ref{bsde_zt2}) and (\ref{equeg1_42}) leads to the following quadratic BSDE  
 \begin{eqnarray}\label{final_ex_1}
  \left\{ \begin{aligned}&\mathrm{d}L^*_t=-f_{\text{Q}}(t, z^*_t,\omega)\mathrm{d}t+ z^*_t\mathrm{d}W_{\cal H}(t), \quad 0\le t\le T,
  \\&L^*_T=c^*_2,\end{aligned} \right.   
 \end{eqnarray}
where the generator $f_{\text{Q}}:[0,T]\times \mathbb{R}^2\times\Omega\rightarrow \mathbb{R}$  is given by
  \begin{equation*} 
\begin{aligned}
f_{\text{Q}}(t,z,\omega)=&  \frac{z_1^2+z_2^2}{4}  -\frac{\tilde\phi_1(t)}{2}z_1 + \frac{\tilde\phi_2(t)}{2}z_2     -r_t
-  \frac{\tilde\phi_1(t)^2+\tilde\phi_2(t)^2}{4}  \\
  &    -\frac{ \sigma_t-\tilde\sigma_t}{4(1-\rho^2)(\sigma_t+\tilde\sigma_t-2\rho^2\sigma_t)} 
  \left[  (1-\rho^2)\left(z_1 +\tilde\phi_1(t)\right)- \rho\sqrt{1-\rho^2} \left(z_2-\tilde\phi_2(t)\right)    \right]^2 .
\end{aligned} 
 \end{equation*}
 
 By (\ref{utilityofe}) and (\ref{equeg1_42}), the value $V$ can be calculated by
\begin{equation}\label{quavalue}
\begin{aligned}
V=\mathbb{E}\left[\varepsilon^{v^*}_T\ln\left(\varepsilon^{v^*}_T X^{u^*}_T  \right)  \right]=\mathbb{E} c^*_2=\mathbb{E}L^*_T.
\end{aligned}\end{equation}

If the filtration $\{{\cal H}_t\}_{0\le t\le T}$ is the augmentation of the natural filtration of $W^1_{\cal H}(t)$ and $W^2_{\cal H}(t)$, then $c^*_2$ is a constant. Suppose that $c^*_2$, $\tilde\phi_1$, $\tilde\phi_2$, $r$, $\sigma$, $\tilde\sigma$ and $1/|\sigma_t+\tilde\sigma_t-2\rho^2\sigma_t|$ are bounded. Then, by \cite[Theorems 4.1]{Fujii18}, the quadratic BSDE (\ref{final_ex_1}) has a unique strong solution $(L^*, z^*_1, z^*_2)$, i.e., $L^*_t$ is a bounded continuous ${\cal H}_t$-adapted process, $ z^*_i(t)$ is an ${\cal H}_t$-progressively measurable process with $\mathbb{E}\int_0^T | z^*_i(t)|^2\mathrm{d}t<\infty$ and $\int_0^t z^*_i(s)\mathrm{d}W^i_{\cal H}(s)$ is an ${\cal H}_t$-${\cal BMO}$-martingale (see Definition \ref{h1andbmo}), $i=1,2$, and $(L^*, z^*_1, z^*_2)$ satisfies the BSDE (\ref{final_ex_1}). 

\begin{remark}
Under mild conditions on the Malliavin derivative, we can obtain the formulae for $ z^*_1(t) $ and $z^*_2(t))$ as follows (see \cite[Corollary 5.1]{Fujii18})
 \begin{eqnarray}\label{final_ex_bmoofz}
  \begin{aligned}&z^*_1(t)=D^1_t L_t^*,
  \\&z^*_2(t)= D^2_t L_t^*.
  \end{aligned}     
 \end{eqnarray}
\end{remark}

To sum up, we give the following theorem.

   \begin{theorem} \label{ex_mainth1}
Assume that $U(x)=\ln x$, $G^1(\mathrm{d}z)=G^2(\mathrm{d}z)=0$, $\mu(t,x)=\mu_0(t)+\varrho_tx$ for some ${\cal G}^1_t$-adapted c\`{a}gl\`{a}d processes $\mu_0(t)$ and $\varrho_t$ with $0\le \varrho_t< \frac{1}{2}\sigma_t^2$, $b\ge \epsilon>0$ for some positive constant $\epsilon$, and $g(s,v)=\frac{1}{2} (\theta_1^2+\theta_2^2)$. Suppose $(u^*,v^*)\in {\cal A}_1'\times{\cal A}_2'$ is optimal for Problem \ref{sdg} under the conditions of Theorem \ref{mainth4} and $\sigma_t+\tilde\sigma_t-2\rho^2\sigma_t\neq 0$. Then $(u^*,v^*)$ and $V$ are given by (\ref{bsde_zt1}), (\ref{bsde_zt2}) and (\ref{quavalue}), where $(L^*, z^*_1, z^*_2)$ solves the quadratic BSDE (\ref{final_ex_1}), and the ${\cal H}_0$-measurable random variable $c^*_2$ can be determined by the initial value condition $L^*_0=\ln X_0$\footnote{In fact, integrating (\ref{final_ex_1}) from $t$ to $T$ yields $L^*_T-L^*_t=-\int_t^Tf_{\text{Q}}(s, z^*_s,\omega)\mathrm{d}s+\int_t^T z^*_s\mathrm{d}W_{\cal H}(s)$. Taking conditional expectation and assuming the It\^{o} integrals are $L^2$-martingales, we get $L^*_t=\mathbb{E}\left[  \int_t^Tf_{\text{Q}}(s, z^*_s,\omega)\mathrm{d}s+L^*_T \big{|}{\cal H}_t\right]$. Taking $t=0$ and using the initial value condition we have $c^*_2=\ln X_0-\mathbb{E}\left[  \int_0^Tf_{\text{Q}}(s, z^*_s,\omega)\mathrm{d}s \big{|}{\cal H}_0\right]$.}. Furthermore, if $\{{\cal H}_t\}_{0\le t\le T}$ is the augmentation of the natural filtration of $W^1_{\cal H}(t)$ and $W^2_{\cal H}(t)$, and $c^*_2$, $\tilde\phi_1$, $\tilde\phi_2$, $r$, $\sigma$, $\tilde\sigma$ and $1/|\sigma_t+\tilde\sigma_t-2\rho^2\sigma_t|$ are bounded, then the quadratic BSDE (\ref{final_ex_1}) has a unique strong solution, where $c^*_2$ can be determined by traversing all constants such that the condition $L^*_0=\ln X_0$ holds.
 \end{theorem}

\begin{remark}
If the filtration $\{{\cal H}_t\}_{0\le t\le T}$ in Theorem \ref{ex_mainth1} is not  the augmentation of the natural filtration of $W^1_{\cal H}(t)$ and $W^2_{\cal H}(t)$, or the coefficients of the generator $f_{\text{Q}}$ is not necessarily bounded, we refer to \cite{Draouil15,Eyraoud05,Li06,Lu13,Wang07} for further results. Meanwhile, the ${\cal H}_0$-measurable random variable $c^*_2$ can be determined by traversing all ${\cal H}_0$-measurable random variable such that the condition $L^*_0=\ln X_0$ holds. Moreover, if ${\cal H}_0$ is generated by a random variable $Y$ and all $\mathbb{P}$-negligible sets, then $c^*_2=c^*_2(Y)$ with some Borel measurable function $c^*_2(y)$. Thus, $c^*_2$ can be determined by traversing all Borel measurable functions $c^*_2(y)$ such that the initial value condition $L^*_0=\ln X_0$ holds. 
\end{remark}

In particular, we assume that the insider information is related to the future value of risk in the insurance market and is of the initial  enlargement type, i.e.,
\begin{equation}\label{filtp}
{\cal H}_t=\bigcap_{s>t}({\cal F}_s\vee Y_0):=\bigcap_{s>t}\left( {\cal F}_s\vee \int_0^{T_0}\varphi(s')\mathrm{d}\bar W_{s'}  \right),\quad 0\le t\le T,
\end{equation}
for some $T_0>T$, and all parameter processes are deterministic functions. Here, $\varphi_t $ is some deterministic function satisfying $\Vert \varphi\Vert^2_{[s,t]}:=\int_s^{t}\varphi^2(s')    \mathrm{d}s'<\infty$ for all $0\le s\le t\le T_0$, and $\Vert \varphi\Vert^2_{[T,T_0]}>0$. 

By the Donsker $\delta$ functional $\delta_y(Y_0)$ and similar procedure in Section \ref{subsubconti1}, we have
\begin{equation} \label{dons_phi12}
\begin{aligned}
&\phi_1(t)=\phi_1(t,y)|_{y=Y_0}=\frac{y-\int_0^t\varphi_s\mathrm{d}\bar W_s}{\Vert \varphi\Vert^2_{[t,T_0]}}\rho \varphi_t\Bigg{|}_{y=Y_0},
\\&  \phi_2(t)=\phi_2(t,y)|_{y=Y_0}=\frac{y-\int_0^t\varphi_s\mathrm{d}\bar W_s}{\Vert \varphi\Vert^2_{[t,T_0]}}\sqrt{1-\rho^2}\varphi_t \Bigg{|}_{y=Y_0}.
  \end{aligned}\end{equation}
Thus the BSDE (\ref{final_ex_1}) is equivalent to the following classical quadratic BSDE with respect to the filtration $\{{\cal F}_t\}_{0\le t\le T}$
 \begin{eqnarray}\label{final_ex_1dons}
  \left\{ \begin{aligned}&\mathrm{d}L^*_t(y)=-\bar f_{\text{Q}}(t, z^*_t(y),y,\omega)\mathrm{d}t+ z^*_t(y)\mathrm{d}W_t, \quad 0\le t\le T,
  \\&L^*_T(y)=c^*_2(y),\end{aligned} \right.   
 \end{eqnarray}
where the generator $\bar f_{\text{Q}}:[0,T]\times \mathbb{R}^2\times\mathbb{R}\times\Omega\rightarrow \mathbb{R}$  is given by
  \begin{equation*} 
\begin{aligned}
&\bar f_{\text{Q}}(t,z,y,\omega)=  \frac{z_1^2+z_2^2}{4}  -\frac{\iota_t-\phi_1(t,y)}{2}z_1 +  \frac{c_t+\phi_2(t,y)}{2}  z_2     -r_t
-  \frac{(\iota_t+\phi_1(t,y))^2+(c_t-\phi_2(t,y))^2}{4}  \\
  &    -\frac{(\sigma_t+\tilde\sigma_t-2\rho^2\sigma_t)(\sigma_t-\tilde\sigma_t)}{4(1-\rho^2)} 
  \Bigg[   \frac{1-\rho^2}{\sigma_t+\tilde\sigma_t-2\rho^2\sigma_t}\left(z_1 +\iota_t+\phi_1(t,y)\right)-\frac{\rho\sqrt{1-\rho^2}}{\sigma_t+\tilde\sigma_t-2\rho^2\sigma_t}\left(z_2-c_t+\phi_2(t,y)\right)    \Bigg]^2, 
\end{aligned} 
 \end{equation*}
and $c_t=\frac{\lambda_t-a_t+\rho b_t\iota_t}{\sqrt{1-\rho^2}b_t}$.

Moreover, the value can be calculated by
\begin{equation}\label{quavalue2}
V=\mathbb{E}(L^*_T(y)|_{y=Y_0}).
\end{equation}

Thus we have the following result.

   \begin{theorem} 
Assume that $U(x)=\ln x$, $G^1(\mathrm{d}z)=G^2(\mathrm{d}z)=0$, $\mu(t,x)=\mu_0(t)+\varrho_tx$ for some c\`{a}gl\`{a}d functions $\mu_0(t)$ and $\varrho_t$ with $0\le \varrho_t< \frac{1}{2}\sigma_t^2$, $b\ge \epsilon>0$ for some positive constant $\epsilon$, and $g(s,v)=\frac{1}{2} (\theta_1^2+\theta_2^2)$. Assume further that $\{{\cal H}_t\}_{0\le t\le T}$ is given by (\ref{filtp}) and all parameter processes are deterministic functions. Suppose $(u^*,v^*)\in {\cal A}_1'\times{\cal A}_2'$ is optimal for Problem \ref{sdg} under the conditions of Theorem \ref{mainth4} and $\sigma_t+\tilde\sigma_t-2\rho^2\sigma_t\neq 0$. Then $(u^*,v^*)$ and $V$ are given by (\ref{bsde_zt1}), (\ref{bsde_zt2}) and (\ref{quavalue2}), where $(L^*(y), z^*_1(y), z^*_2(y))$ solves the classical quadratic BSDE (\ref{final_ex_1dons}), $\phi_1(t,y)$ and $\phi_2(t,y)$ are given by (\ref{dons_phi12}), and the ${\cal H}_0$-measurable random variable $c^*_2$ can be determined by traversing all Borel measurable functions $c^*_2(y)$ such that $L^*_0=\ln X_0$.
 \end{theorem}
 
 \begin{remark}
We can also give a classical quadratic BSDE like (\ref{final_ex_1dons}) to characterize the optimal pair $(u^*,v^*)$ when the insurer has insider information about the future value of the risky asset, i.e., ${\cal H}_t=\bigcap_{s>t}\left({\cal F}_s\vee \int_0^{T_0} \varphi(s')\mathrm{d}W^1_{s'}\right)$. In this case, we have $\phi_1(t)=\phi_1(t,y)|_{y=Y_0}=\frac{y-\int_0^t\varphi_s\mathrm{d} W^1_s}{\Vert \varphi\Vert^2_{[t,T_0]}}  \varphi_t\Big{|}_{y=Y_0}$ and $\phi_2(t)=\phi_2(t,y)|_{y=Y_0}=0$.
 \end{remark}

Since the quadratic BSDE (\ref{final_ex_1}) or (\ref{final_ex_1dons}) has no closed-form solution and can only be solved by numerical methods, we concentrate on the special situation without model uncertainty, that is, ${\cal A}_2'=\{(0,0)\}$. Then Problem \ref{sdg} degenerates to the following anticipating stochastic control problem.

\begin{problem}\label{sdg2}
 Select $u^*\in {\cal A}_1' $ such that
\begin{equation} \label{valueofsdg2}
\tilde V:=\tilde J(u^*)= \sup_{u\in {\cal A}_1'}   \tilde J(u),
\end{equation}  
where $\tilde J(u):=\mathbb{E}\left[  \ln X_T^u \right]$. We call $\tilde V$ the value (or the optimal expected utility) of Problem \ref{sdg2}.
\end{problem}
 
 Since ${\cal A}_2'=\{(0,0)\}$, which implies that $\theta^*_1=\theta^*_2=0$ in (\ref{reduceth6}). Thus, the optimal strategy $u^*\in{\cal A}'_1$ is given by
 \begin{equation}\label{reduceth6_2} 
 \begin{aligned}
 &\pi^*_t=\frac{(1-\rho^2)\tilde \phi_1(t)+\rho\sqrt{1-\rho^2}\tilde\phi_2(t)}{ (\tilde\sigma_t/  \sigma_t-\rho^2)\sigma_t},
  \\&\kappa_t^*=\frac{\sqrt{1-\rho^2} \tilde\phi_2(t)+\rho(\sigma_t/ \tilde\sigma_t-1)\tilde\phi_1(t)}{(1-\rho^2\sigma_t/ \tilde\sigma_t)b_t},
  \end{aligned}
 \end{equation}
 v.
 
Substitute (\ref{reduceth6_2}) into (\ref{valueofsdg2}). Then by (\ref{halfequ_21})-(\ref{halfequ_22}) and tedious calculation we have
  \begin{equation}\label{value3} 
  \begin{aligned}
   \tilde V=&\ln X_0+\mathbb{E}\int_0^Tr_t\mathrm{d}t+\frac{1}{2}\mathbb{E}\int_0^T\frac{1}{1-\rho^2\sigma_t/ \tilde\sigma_t}\Big[ \left(  \big(1-2\rho^2\big)\sigma_t/ \tilde\sigma_t+ \rho^2\right) \tilde\phi_1(t)^2\\
  & +2\rho\sqrt{1-\rho^2}( \sigma_t/ \tilde\sigma_t-1) \left[\tilde \phi_1(t)\tilde\phi_2(t)\right] + (1-\rho^2)  \tilde\phi_2(t)^2 \Big]\mathrm{d}t.
    \end{aligned}
 \end{equation}
 
 Assume further that the insider filtration $\{{\cal H}_t\}_{0\le t\le T}$ is given by
  \begin{equation} \label{filtrationex_conti}
 \begin{aligned}
 {\cal F}_t\subset {\cal H}_t\subset   \bigcap_{s>t}\left({\cal F}_s\vee  W^1_{T_0}\vee \bar W_{T_0} \right)=:\bar{\cal H}_t, \quad 0\le t\le T,  
   \end{aligned}
 \end{equation}
 for some $T_0>T$.
Then the enlargement of filtration technique can be applied to give the concrete expression of $\tilde \phi_i$ in (\ref{reduceth6_2}), $i=1,2$. We give the following lemma to characterize the decomposition of the ${\cal H}_t$-semimartingale $W^i$ in Theorem \ref{mainth3}, $i=1,2$. The proof can be found in \cite[page 327]{DiNunno09}.

\begin{lemma}
[Enlargement of filtration] 
\label{lemofenlarge1}
The process $W^i_t$, $t\in[0,T]$, is a semimartingale with respect to the filtration $\{{\cal H}_t\}_{0\le t\le T}$ given by (\ref{filtrationex_conti}), $i=1,2$. Its semimartingale decomposition is 
  \begin{equation}  
 \begin{aligned}
W ^i_t=W_{\cal H}^i(t)+\int_0^t\frac{\mathbb{E}\big[W^i_{T_0}|  {\cal H}_s\big]-W^i_s}{T_0-s}\mathrm{d}s,\quad 0\le t\le T,  
   \end{aligned}
 \end{equation}
\end{lemma}
where $W_{\cal H}^i(t)$ is an ${\cal H}_t$-Brownian motion, $i=1,2$.

Combined with Lemma \ref{lemofenlarge1}, the ${\cal H}_t$-progressively measurable process $\phi^i_t$ in Theorem \ref{mainth3} is of the form
   \begin{equation}  
 \begin{aligned}
 \phi^i_t=\frac{\mathbb{E}\big[W^i_{T_0}|{\cal H}_t\big]-W^i_t}{T_0-t},\quad 0\le t\le T,
   \end{aligned}
 \end{equation}
 $i=1,2$. Thus, we can give a concrete characterization of the optimal strategy $u^*$ as follows
  \begin{equation}\label{reduceth6_2_2} 
 \begin{aligned}
 &\pi^*_t=\frac{(1-\rho^2)  \left( \iota_t+\frac{\mathbb{E}\big[W^1_{T_0}|{\cal H}_t\big]-W^1_t}{T_0-t}\right)+\rho\sqrt{1-\rho^2}\left( \frac{\lambda_t-a_t+\rho b_t\iota_t}{\sqrt{1-\rho^2}b_t}  -\frac{\mathbb{E}\big[W^2_{T_0}|{\cal H}_t\big]-W^2_t}{T_0-t}    \right)}{ (\tilde\sigma_t/  \sigma_t-\rho^2)\sigma_t},
  \\&\kappa_t^*=\frac{\sqrt{1-\rho^2} \left( \frac{\lambda_t-a_t+\rho b_t\iota_t}{\sqrt{1-\rho^2}b_t}  -\frac{\mathbb{E}\big[W^2_{T_0}|{\cal H}_t\big]-W^2_t}{T_0-t}  \right)+\rho(\sigma_t/ \tilde\sigma_t-1)  \left( \iota_t+\frac{\mathbb{E}\big[W^1_{T_0}|{\cal H}_t\big]-W^1_t}{T_0-t}\right)}{(1-\rho^2\sigma_t/ \tilde\sigma_t)b_t}. 
  \end{aligned}
 \end{equation}
 
 To sum up, we give the following theorem for this special case.
 
    \begin{theorem} \label{ex_mainth1_special}
Assume that $U(x)=\ln x$, $G^1(\mathrm{d}z)=G^2(\mathrm{d}z)=0$, $\mu(t,x)=\mu_0(t)+\varrho_tx$ for some ${\cal G}^1_t$-adapted c\`{a}gl\`{a}d processes $\mu_0(t)$ and $\varrho_t$ with $0\le \varrho_t< \frac{1}{2}\sigma_t^2$, $b\ge \epsilon>0$ for some positive constant $\epsilon$, and no model uncertainty is considered. Suppose $u^* \in {\cal A}_1' $ is optimal for Problem \ref{sdg2} under the conditions of Theorem \ref{mainth2} and $\tilde\sigma_t\neq \rho^2\sigma_t$. Then $u^* $ is given by (\ref{reduceth6_2}), and $\tilde V$ is given by (\ref{value3} ). Furthermore, if the filtration $\{{\cal H}_t\}_{0\le t\le T}$ is of the form (\ref{filtrationex_conti}), then $u^*$ is given by (\ref{reduceth6_2_2}).
 \end{theorem}
 
\begin{remark}
Theorem \ref{ex_mainth1_special} covers the result of the continuous case in \cite{Peng16}. However, in Theorem \ref{ex_mainth1_special}, the coefficients of the model (\ref{fin-m})-(\ref{ins-m}) in Section \ref{Sec:model} are all anticipating processes influenced by uncertain economic environment (i.e., the classical SDEs in \cite{Peng16} are replaced by the anticipating SDEs (\ref{fin-m})-(\ref{ins-m})), the mean rate of return $\mu$ on the risky asset is influenced by the strategy $\pi$ of the large insurer, and the filtration $\{{\cal H}_t\}_{0\le t\le T}$ is more general than $\{\bar{\cal H}_t\}_{0\le t\le T}$ introduced in \cite{Peng16}. These make Theorem \ref{ex_mainth1_special} more general than \cite{Peng16}.
\end{remark}

Suppose all parameter processes are deterministic functions. When ${\cal H}_t={\cal F}_t$ (i.e., the insurer has no insider information) and ${\cal H}_t=\bigcap_{s>t}\left( {\cal F}_s\vee \bar W_{T_0}\right)$ (i.e., the insurer owns the insider information $\bar W_{T_0}$ about the future risk in the insurance market), we have the following results.

   \begin{corollary}  \label{endcor1}
Assume that $U(x)=\ln x$, $G^1(\mathrm{d}z)=G^2(\mathrm{d}z)=0$, $\mu(t,x)=\mu_0(t)+\varrho_tx$ for some deterministic c\`{a}gl\`{a}d functions $\mu_0(t)$ and $\varrho_t$ with $0\le \varrho_t<\frac{1}{2}\sigma_t^2$, $b\ge \epsilon>0$ for some positive constant $\epsilon$, and no model uncertainty is considered. Assume further that ${\cal H}_t=  {\cal F}_t$ and all parameter processes are deterministic functions. Suppose $u^* \in {\cal A}_1' $ is optimal for Problem \ref{sdg2} under the conditions of Theorem \ref{mainth2} and $\tilde\sigma_t\neq \rho^2\sigma_t$. Then $u^*$ and $\tilde V$ are given by 
  \begin{equation} 
 \begin{aligned}
 &\pi^*_t=    \frac{(1-\rho^2)\iota_t}{ (\tilde\sigma_t/  \sigma_t-\rho^2)\sigma_t}+\frac{\rho(\lambda_t-a_t+\rho b_t\iota_t)}{ (\tilde\sigma_t/  \sigma_t-\rho^2)\sigma_tb_t},
  \\&\kappa_t^*=\frac{\lambda_t-a_t+\rho b_t\iota_t}{(1-\rho^2\sigma_t/ \tilde\sigma_t)b_t^2}+\frac{\rho(\sigma_t/ \tilde\sigma_t-1)\iota_t}{(1-\rho^2\sigma_t/ \tilde\sigma_t)b_t},
  \\&\tilde V=\ln X_0+\int_0^Tr_t\mathrm{d}t+ \frac{1}{2}\int_0^T\frac{1}{1-\rho^2\sigma_t/ \tilde\sigma_t}\Big[ \left(  \big(1-2\rho^2\big)\sigma_t/ \tilde\sigma_t+ \rho^2 \right)\iota_t^2\\
 &\quad\quad +2\rho( \sigma_t/ \tilde\sigma_t-1)\iota_t\frac{\lambda_t-a_t+\rho b_t\iota_t}{ b_t} +  \frac{(\lambda_t-a_t+\rho b_t\iota_t)^2}{ b_t^2}\Big]\mathrm{d}t.
  \end{aligned}
 \end{equation}
 \end{corollary}

Note that if $\tilde\sigma=\sigma$ (i.e., the insurer is `small'), the optimal strategy $u^*$ and the value $\tilde V$ in Corollary \ref{endcor1} are consistent with those in Proposition \ref{co2rex_mainth_idontknow_donsker}.

 \begin{remark}
When there is no insider information and $\rho=0$ (i.e., there is little uncertainty in the insurance marker, see \cite{Zou14}), we can see from Corollary \ref{endcor1} that, the insurer should invest more in the risky asset if her investment has more influence on the mean rate of return. The reason is that she introduces a higher appreciate rate in the risky asset price when investing. However, her optimal insurance strategy is not affected. Moreover, the optimal expected utility increases with her influence on the mean rate of return.
 \end{remark}
 
  \begin{corollary}  \label{endcor2}
Assume that $U(x)=\ln x$, $G^1(\mathrm{d}z)=G^2(\mathrm{d}z)=0$, $\mu(t,x)=\mu_0(t)+\varrho_tx$ for some deterministic c\`{a}gl\`{a}d functions $\mu_0(t)$ and $\varrho_t$ with $0\le \varrho_t< \frac{1}{2}\sigma_t^2$, $b\ge \epsilon>0$ for some positive constant $\epsilon$, and no model uncertainty is considered. Assume further that ${\cal H}_t=\bigcap_{s>t}\left( {\cal F}_s\vee \bar W_{T_0}\right)$ and all parameter processes are deterministic functions. Suppose $u^* \in {\cal A}_1' $ is optimal for Problem \ref{sdg2} under the conditions of Theorem \ref{mainth2} and $\tilde\sigma_t\neq \rho^2\sigma_t$. Then $u^*$ and $\tilde V$ are given by 
  \begin{equation} \label{endoptimal}
 \begin{aligned}
 &\pi^*_t=\frac{(1-\rho^2)\iota_t}{ (\tilde\sigma_t/  \sigma_t-\rho^2)\sigma_t}+\frac{\rho(\lambda_t-a_t+\rho b_t\iota_t)}{ (\tilde\sigma_t/  \sigma_t-\rho^2)\sigma_tb_t},
  \\&\kappa_t^*=\frac{\lambda_t-a_t+\rho b_t\iota_t}{(1-\rho^2\sigma_t/ \tilde\sigma_t)b_t^2}+\frac{\rho(\sigma_t/ \tilde\sigma_t-1)\iota_t}{(1-\rho^2\sigma_t/ \tilde\sigma_t)b_t}-\frac{ \bar W_{T_0}-\bar W_{t} }{ b_t(T_0-t )} ,
  \\&\tilde V=\ln X_0+\int_0^Tr_t\mathrm{d}t+\frac{1}{2}\int_0^T\frac{1}{1-\rho^2\sigma_t/ \tilde\sigma_t}\Bigg[ \left(  \big(1-2\rho^2\big)\sigma_t/ \tilde\sigma_t+ \rho^2\right)\left(\iota_t^2+\frac{\rho^2}{T_0-t}  \right)
  \\&\quad\quad+2\rho ( \sigma_t/ \tilde\sigma_t-1)\left(\iota_t\frac{\lambda_t-a_t+\rho b_t\iota_t}{ b_t}-\frac{\rho(1-\rho^2)}{T_0-t}  \right) +  \left(\frac{(\lambda_t-a_t+\rho b_t\iota_t)^2}{ b^2_t} +\frac{(1-\rho^2)^2}{T_0-t} \right) \Bigg]\mathrm{d}t.
  \end{aligned}
 \end{equation}
 \end{corollary}

Note that if $\tilde\sigma=\sigma$ (i.e., the insurer is `small'), the optimal strategy $u^*$ and the value $\tilde V$ in Corollary \ref{endcor2} are consistent with those in Proposition \ref{pex_mainth_idontknow_donskerc}.
 
 \begin{remark}
 When the insurer has insider information about the insurance market and $\rho=0$, we can see from Corollary \ref{endcor2} that, the insurer should invest more in the risky asset if her investment has more influence on the mean rate of return. However, her optimal insurance strategy is not affected. Moreover, the optimal expected utility increases with her influence on the mean rate of return. On the other hand, by comparing Corollary \ref{endcor2} with Corollary \ref{endcor1}, the insurer should reduce the liability ratio by $\frac{\bar W_{T_0}-\bar W_t}{b_t(T_0-t)}$ if she captures insider information in the insurance market and $\bar W_{T_0}>\bar W_{t}$. Given $ \bar W_{T_0}- \bar W_t=x>0$, we can also see that the closer the future time $T_0$ is, the more she should reduce her liability ratio. Moreover, her optimal expected utility (compared with the case she has no insider information) is gained by $\Delta \tilde V=\frac{1}{2}\ln \frac{T_0}{T_0-T}>0$ if $\rho=0$, and $\Delta\tilde V$ becomes greater as $T_0$ gets closer (i.e., the insurer has `better' insider information). 
 \end{remark}
 
 \begin{remark}
  By similar procedure, we can also obtain the optimal strategy when the insurer has insider information about the future value of the risky asset. For instance, if ${\cal H}_t=\bigcap_{s>t}\left( {\cal F}_s\vee   W^1_{T_0}\right)$, the optimal strategy is given by 
    \begin{equation} 
 \begin{aligned}
 &\pi^*_t=\frac{(1-\rho^2)\iota_t}{ (\tilde\sigma_t/  \sigma_t-\rho^2)\sigma_t}+\frac{\rho(\lambda_t-a_t+\rho b_t\iota_t)}{ (\tilde\sigma_t/  \sigma_t-\rho^2)\sigma_tb_t}+\frac{(1-\rho^2) }{ (\tilde\sigma_t/  \sigma_t-\rho^2) }\frac{W^1_{T_0}-W^1_t}{\sigma_t(T_0-t)},
  \\&\kappa_t^*=\frac{\lambda_t-a_t+\rho b_t\iota_t}{(1-\rho^2\sigma_t/ \tilde\sigma_t)b_t^2}+\frac{\rho(\sigma_t/ \tilde\sigma_t-1)\iota_t}{(1-\rho^2\sigma_t/ \tilde\sigma_t)b_t} +\frac{\rho(\sigma_t/ \tilde\sigma_t-1) }{1-\rho^2\sigma_t/ \tilde\sigma_t}  \frac{W^1_{T_0}-W^1_t}{b_t(T_0-t)} .
    \end{aligned}
 \end{equation}
 Then the insurer should invest more in the risky asset if her investment has more influence on the mean rate
of return whereas her optimal insurance strategy is not affected given that $\rho=0$. Moreover, the value $\tilde V$ can be given by
     \begin{equation} \label{value5}
 \begin{aligned}
  \\ \tilde V= \ln X_0+\int_0^Tr_t\mathrm{d}t+\frac{1}{2}\int_0^T\frac{1}{1-\rho^2\sigma_t/ \tilde\sigma_t}\Bigg[& \left(  \big(1-2\rho^2\big)\sigma_t/ \tilde\sigma_t+ \rho^2\right)\left(\iota_t^2+\frac{1}{T_0-t}  \right)
  \\&+2\rho ( \sigma_t/ \tilde\sigma_t-1)\iota_t\frac{\lambda_t-a_t+\rho b_t\iota_t}{ b_t}   + \frac{(\lambda_t-a_t+\rho b_t\iota_t)^2}{ b^2_t}  \Bigg]\mathrm{d}t.
  \end{aligned}
 \end{equation}
The optimal expected utility increases with her influence on the mean rate of return. On the other hand, given that $\rho=0$, the insurer should increase the proportion of her total wealth invested in the risky asset by $\frac{W^1_{T_0}-W^1_t}{\tilde\sigma_t(T_0-t)}$ if she knows the future value of the risky asset and $W^1_{T_0}>W^1_t$. Given $ W^1_{T_0}- W^1_t=x>0$, we can also see that the closer the future time $T_0$ is, the more she should increase her proportion with respect to the risky asset. Moreover, her optimal expected utility (compared with the case she has no insider information) is gained by $\Delta \tilde V=\frac{1}{2}\int_0^T\frac{\sigma_t}{\tilde\sigma_t}\frac{1}{T_0-t}>0$ if $\rho=0$, and $\Delta\tilde V$ becomes greater as $T_0$ gets closer (i.e., the insurer has `better' insider information). By comparing (\ref{value5}) with (\ref{endoptimal}) we find that, $\rho=0$, the insider information about the financial market (compare with that about the insurance market) will provide the insurer more optimal expected utility if she can really influence the mean rate of return on the risky asset. 
  \end{remark}

\section{Conclusion}
\label{sec:conclusion}

In this paper, we investigate the optimal investment and risk control problem for an insurer under both model uncertainty and insider information. The insurer's risk process and the financial risky asset process are assumed to be correlated jump diffusion processes with very general L\'{e}vy forms. The insider information is about the financial risky asset and the insurance polices, and is of the most general form rather than the initial enlargement type. We establish the corresponding anticipating stochastic differential game problem and give the characterization of robust optimal strategy by combining the theory of forward integrals with the stochastic maximum principle. The two typical situations when the insurer is `small' and `large' are discussed and the corresponding optimal strategies are presented. Our results generalize some results in the literature. We also discuss the impacts of the model uncertainty, insider information and the `large' insurer on the optimal strategy.

For further work, the quadratic BSDE corresponding to the robust optimal strategy for a large insurer need to be studied. Moreover, the some constraints on the strategy can be imposed, which is also a subject of ongoing research.

\section*{Acknowledgments}
The authors are very grateful to the editors and the anonymous referee for their helpful suggestions and comments. The work is funded by the National Natural Science Foundation of China (No. 72071119).
Thanks to Xiaoqun Wang for his help and advice.

\bibliographystyle{IEEEtran}
\bibliography{malliavin_op2}

\begin{appendices}
\section{Basic theories of It\^{o} integrals with jumps}\label{appendix_A}
\setcounter{equation}{0} 
\setcounter{theorem}{0}

 \renewcommand\theequation{A.\arabic{equation}}
\renewcommand\thetheorem{A.\arabic{theorem}}
\renewcommand\theremark{A.\arabic{remark}}
 
 In this section, we give basic definitions and theories of It\^{o} integrals in discontinuous cases, i.e., It\^{o} integrals with jumps. For It\^{o} integrals in continuous cases, we refer to \cite{Karatzas91}.
  
We begin with a filtered probability space $(\Omega, {\cal F},\{{\cal F}_t\}_{t\in\mathbb{R}_+},\mathbb{P})$, where $\{{\cal F}_t\}$ satisfies the usual condition, and we identify two stochastic process if they are indistinguishable (see \cite{Karatzas91}).

Denote by ${\cal M}_{2 }$, ${\cal M}_{2,\text{loc}}$ and ${\cal M}_{\text{loc}}$ the spaces of all ${\cal F}_t$-square-integrable martingales\footnote{We say a martingale $M$ is square-integrable if $\mathbb{E}M^2_t<\infty$ for all $t\ge 0$, which is different from definitions in some literature (see \cite{He92}).}, ${\cal F}_t$-locally square-integrable martingales and ${\cal F}_t$-local martingales with zero initial values, respectively (see \cite{Karatzas91} for definitions). By the decomposition theorem for local martingales (see \cite{He92}), we can define ${\cal M}^c_{\text{loc}}$ and ${\cal M}^d_{\text{loc}}$ the subspaces of continuous local martingales and purely discontinuous local martingales, respectively (similar statements hold for ${\cal M}_{2 }$ and ${\cal M}_{2,\text{loc}}$). Given $M, N\in {\cal M}_{\text{loc}}$, we can define the covariation process $[M,N]_t$ of $M$ and $N$ (see \cite{He92}), which is an ${\cal F}_t$-adapted bounded variation process. If $[M,N]_t$ is locally interable \footnote{We say an ${\cal F}_t$-adapted increasing process $A_t$ is locally integrable, if there exists an ${\cal F}_t$-stopping times $\{T_n\}_{n=1}^\infty$ such that $T_n\uparrow \infty$ and $A_{Tn}$ is integrable for all $n\in\mathbb{N_+}$. If an ${\cal F}_t$-adapted bounded variation process $V_t$ can be represented by the difference of two ${\cal F}_t$-adapted locally integrable increasing processes, we say $V_t$ is locally integrable. In this case we can define the ${\cal F}_t$-compensator (or the ${\cal F}_t$-dual predictable projection) $\hat V_t$ of $V_t$ (see \cite{He92}), which is an ${\cal F}_t$-adapted predictable bounded variation process.}, denote by $\langle M,N\rangle_t$ the ${\cal F}_t$-compensator of $[M,N]_t$, which is called the ${\cal F}_t$-predictable covariation process of $M$ and $N$. Note that if $M$ and $N$ both belong to ${\cal M}_{2,\text{loc}}$, the ${\cal F}_t$-predictable covariation process $\langle M,N\rangle_t$ exists. If $M=N$, we write $[M]_t=[M,M]_t$ (or $\langle M\rangle_t=\langle M,M\rangle_t$) for simplicity, which is called the quadratic variation process (or ${\cal F}_t$-predictable quadratic variation process) of $M$.

For fixed $M\in {\cal M}_{\text{loc}}$, we say a process $\varphi$ is It\^{o} integrable with respect to $M$ if $\varphi \in {\mathfrak{L}}^*_{\text{loc}}([M])$, i.e., $\varphi$ is ${\cal F}_t$-predictable and $\sqrt{\int_0^t\varphi_s^2\mathrm{d}[M]_s}$ is locally integrable. We can define the It\^{o} integral $I_t^M(\varphi)=\int_0^t\varphi_s\mathrm{d}M_s$ of $\varphi$ with respect to $M$ (see \cite{He92}). Then the It\^{o} integral $I^M_t$ is a linear operator from ${\mathfrak{L}}^*_{\text{loc}}([M])$ to $ {\cal M}_{\text{loc}}$. If $M\in  {\cal M}_{2,\text{loc}}$, $I^M_t$ maps from ${\mathfrak{L}}^*_{\text{loc}}(\langle M\rangle)$ to $ {\cal M}_{2,\text{loc}}$, where ${\mathfrak{L}}^*_{\text{loc}}(\langle M\rangle)$ is the space of all ${\cal F}_t$-predictable processes $\varphi$ such that $\int_0^t\varphi^2_s\mathrm{d}\langle M\rangle_s$ is locally integrable. If $M\in  {\cal M}_{2 }$, $I^M_t$ maps from ${\mathfrak{L}}^* (\langle M\rangle)$ to $ {\cal M}_{2 }$, where ${\mathfrak{L}}^* (\langle M\rangle)$ is the space of all ${\cal F}_t$-predictable processes $\varphi$ such that $\mathbb{E}\int_0^t\varphi^2_s\mathrm{d}\langle M\rangle_s<\infty$, $\forall t\ge 0$.

\begin{remark}\label{integral_ito_rm1}
Fix $M\in {\cal M}_{\text{loc}}$. Then $\varphi$ is an ${\cal F}_t$-predictable process such that $\int_0^t\varphi^2_s\mathrm{d}[ M]_s$ is locally integrable if and only if $\varphi\in  {\mathfrak{L}}^*_{\text{loc}}([M])$ and if $I^M(f)\in  {\cal M}_{2,\text{loc}}$; $\varphi$ is an ${\cal F}_t$-predictable process such that $\mathbb{E}\int_0^t\varphi^2_s\mathrm{d}[ M]_s<\infty$ for all $t\ge 0$ if and only if $\varphi\in  {\mathfrak{L}}^*_{\text{loc}}([M])$ and if $I^M(f)\in  {\cal M}_{2}$.
\end{remark}

We say an ${\cal F}_t$-adapted process $Z$ is an ${\cal F}_t$-semimartingale if $Z$ has a decomposition $Z_t=Z_0+M_t+V_t$ for some $M\in {\cal M}_{\text{loc}}$ and ${\cal F}_t$-adapted bounded variation process $V$. If $V$ is ${\cal F}_t$-predictable, we say $Z$ is an ${\cal F}_t$-special semimartingale, in which case the above decomposition is unique such that $V$ is ${\cal F}_t$-predictable, and we call it the canonical decomposition of $Z$. We say a process $\varphi$ is It\^{o} integrable with respect to a semimartingale $Z$ if $Z$ has a decomposition $Z_t=Z_0+M_t+V_t$ such that $\varphi$ is It\^{o} integrable with respect to $M$ and pathwise Lebesgue-Stieltjes integrable with respect to $V$, which is well-defined. The definition of the covariation and the ${\cal F}_t$-predictable covariation can also be extended to the case of semimartingales as well. We refer to \cite{He92,Jacod03} for more It\^{o} theories of semimartingales.

In particular, we can obtain finer results when we consider the two special types of ${\cal F}_t$-martingales, namely, the ${\cal H}^1$-martingale and the ${\cal BMO}$-martingale. The properties of them also play a crucial role in the theory of the quadratic BSDE (see \cite{Fujii18}). Here, we only give the definitions. We refer to \cite{He92,Fujii18} for more details.

\begin{definition}\label{h1andbmo}
Denote by ${\cal H}^1$ the space of all local martingales $M$ such that $\Vert M\Vert_{{\cal H}^1}:=  \mathbb{E}\sqrt{[M]_\infty} <\infty$. Denote by ${\cal BMO}$ the space of all $L^2$-bounded martingales such $M$ that $\Vert M\Vert_{\cal BMO}:=  \sup_{T\in {\cal T}} \sqrt{\frac{\mathbb{E}(M_\infty-M_{T-}1_{\{T>0\}})^2}{\mathbb{P}(T<\infty)}}<\infty$, where ${\cal T}$ is the set of all ${\cal F}_t$-stopping times. Each element in ${\cal H}^1$ (resp. ${\cal BMO}$) is called an ${\cal H}^1$-martingale (resp. ${\cal BMO}$-martingale). Moreover, ${\cal BMO}$ is the dual space of the Banach space ${\cal H}^1$ (see \cite{He92}). 

\end{definition}

Now we turn to the It\^{o} integral with respect to a compensated random measure. Before that, we give the definition of the integer-valued random measure.

\begin{definition}\label{integer_rm}
We say a family $\mu=(\mu(\omega;\mathrm{d}t,\mathrm{d}z):\omega\in \Omega)$ is a ($\sigma$-finite) random measure on $\mathbb{R}_+\times\mathbb{R}^d$ ($d\in\mathbb{N}_+$) if $\mu(\omega,\cdot)$ is a ($\sigma$-finite) measure on $\mathbb{R}_+\times\mathbb{R}^d$ for fixed $\omega\in\Omega$ and $\mu(\cdot,A)$ is a random variable on $(\Omega, {\cal F},\mathbb{P})$ for fixed $A\in {\cal B}(\mathbb{R}_+)\times{\cal B}(\mathbb{R}^d)$. Moreover, we say a $\sigma$-finite random measure $\mu$ is an ${\cal F}_t$-integer-valued random measure on $\mathbb{R}_+\times\mathbb{R}^d$ if it satisfies the following extra conditions:
 \begin{itemize}
  \item[(i)] $\mu(\omega,\cdot)$ has values in $\mathbb{N}\cup \{\infty\}$ for any $\omega\in\Omega$;
    \item[(ii)] $\mu(\omega,\{0\}\times {\mathbb R}^d)=0$ for any $\omega\in\Omega$; 
    \item[(iii)]  $\mu(\omega,\{t\}\times \mathbb{R}^d)\le 1$ for any $\omega\in\Omega$, $t\ge 0$;
    \item[(iv)] $\mu $ is ${\cal F}_t$-optional in the sense that $\mu(\cdot,[0,t]\times B)$ is an ${\cal F}_t$-optional process for any $B\in{\cal B}(\mathbb{R}^d)$;
      \item[(v)] $\mu $ is ${\cal F}_t$-predictably $\sigma$-finite in the sense that there exists some strictly positive ${\cal F}_t$-predictable random field $V$ on $\Omega\times{\mathbb R}_+\times\mathbb{R}^d$ such that $\int_0^\infty\int_{\mathbb{R}^d}V(s,z)\mu(\mathrm{d}s,\mathrm{d}z)$ is integrable.
\end{itemize}
Denote by $\hat \mu(\mathrm{d}t,\mathrm{d}z)$ the ${\cal F}_t$-compensator (or the ${\cal F}_t$-dual predictable projection) of an ${\cal F}_t$-integer-valued random measure $\mu(\mathrm{d}t,\mathrm{d}z)$ (see \cite{He92} for the definition), which is a $\sigma$-finite random measure on $\mathbb{R}_+\times\mathbb{R}^d$. Define $\tilde \mu(\mathrm{d}t,\mathrm{d}z):=\mu(\mathrm{d}t,\mathrm{d}z)-\hat\mu(\mathrm{d}t,\mathrm{d}z)$, which is called the compensated random measure of $\mu(\mathrm{d}t,\mathrm{d}z)$.
\end{definition}

\begin{remark}\label{r_integer_rm}
Let $Z$ be an ${\cal F}_t$-adapted c\`{a}dl\`{a}g ${\mathbb R}^d$-valued process. Then 
    \begin{equation*} 
\begin{aligned}
\mu^Z(\mathrm{d}t,\mathrm{d}z):=\sum_{s>0}\delta_{(s,\Delta Z_s)}(\mathrm{d}t,\mathrm{d}z)1_{\{\Delta Z_s\neq 0\}}(s)
\end{aligned}
\end{equation*}
defines an ${\cal F}_t$-integer-valued random measure on $\mathbb{R}_+\times\mathbb{R}^d$, where $\delta_a$ denotes the Dirac measure at point $a$. We call $\mu^Z$ the jump measure of $Z $ and $\hat \mu^Z$ the L\'e{vy} system of $Z$.
\end{remark}

In the rest of this section, we consider an ${\cal F}_t$-integer-valued random measure $\mu$, and suppose the ${\cal F}_t$-compensator $\hat \mu$ of $\mu$ is absolutely continuous in time, i.e., $\hat\mu(\mathrm{d}t,\mathrm{d}z)= K(t,\mathrm{d}z)\mathrm{d}t$ for some random transition measure $K(t,\mathrm{d}z)$, the definition of which is similar to Definition \ref{integer_rm} (see \cite[page 37]{Cinlar11} for the definition of transition measures). 

Define the following three linear spaces:
  \begin{itemize}
  \item[(i)] $\mathfrak{L}^*_{\text{loc}}(\mu)$ is the space of all ${\cal F}_t$-predictable random fields $\zeta(t,z)$ such that $\sqrt{\int_0^t\int_{\mathbb{R}^d}\zeta^2(s,z)\mu(\mathrm{d}s,\mathrm{d}z)}$ is locally integrable;
    \item[(ii)] $\mathfrak{L}^*_{\text{loc}}(\hat\mu)$ is the space of all ${\cal F}_t$-predictable random fields $\zeta(t,z)$ such that $\int_0^t\int_{\mathbb{R}^d}\zeta^2(s,z)K(s,\mathrm{d}z)\mathrm{d}s<\infty$, $\forall t\ge 0$;
        \item[(iii)] $\mathfrak{L}^*(\hat\mu)$ is the space of all ${\cal F}_t$-predictable random fields $\zeta(t,z)$ such that $\mathbb{E}\int_0^t\int_{\mathbb{R}^d}\zeta^2(s,z)K(s,\mathrm{d}z)\mathrm{d}s<\infty$, $\forall t\ge 0$.
\end{itemize}
We have $\mathfrak{L}^*(\hat\mu)\subset \mathfrak{L}^*_{\text{loc}}(\hat\mu)\subset \mathfrak{L}^*_{\text{loc}}(\mu)$. We say a random field $\zeta$ is It\^{o} integrable with respect to $\tilde \mu$ if $\zeta\in \mathfrak{L}^*_{\text{loc}}(\mu)$. In this case we can define the It\^{o} integral $I^{\tilde \mu}_t(\zeta):=\int_0^t\int_{\mathbb{R}^d}\zeta(s,z)\tilde N(\mathrm{d}s,\mathrm{d}z)$ (see \cite{He92}), which belongs to $ {\cal{M}}^d_{\text{loc}}$. If $\mu$ is induced by some ${\cal F}_t$-adapted c\`{a}dl\`{a}g process $Z$, then we have $\Delta I^{\tilde \mu}_s(\zeta)=\zeta(s,\Delta Z_s)1_{\{\Delta Z_s\neq 0\}}$。 Moreover, $I^{\tilde \mu}_t$ maps from $\mathfrak{L}^*_{\text{loc}}(\hat\mu)$ to ${\cal M}^d_{2,\text{loc}}$, and from $\mathfrak{L}^*(\hat\mu)$ to ${\cal M}^d_{2 }$. We refer to \cite[Propositions 3.37 and 3.39]{Eberlein19}, \cite[Proposition II.1.14]{Jacod03} and \cite[Theorem 11.23]{He92} for more properties of the stochastic integral with respect to a compensated random measure in our settings.

\begin{remark}
If an ${\cal F}_t$-predictable random field $\zeta(t,z)$, $t\in\mathbb{R}_+$, satisfies that $\int_0^t\int_{\mathbb{R}^d}|\zeta(s,z)|K(t,\mathrm{d}z)\mathrm{d}s<\infty$, then $\zeta\in \mathfrak{L}^*_{\text{loc}}(\mu)$ and $\int_0^t\int_{\mathbb{R}^d}\zeta(s,z)\tilde N(\mathrm{d}s,\mathrm{d}z)=\int_0^t\int_{\mathbb{R}^d}\zeta(s,z)  N(\mathrm{d}s,\mathrm{d}z)-\int_0^t\int_{\mathbb{R}^d}\zeta(s,z)K(t,\mathrm{d}z)\mathrm{d}s $.
\end{remark}
 
 The well-known It\^{o} formula for general semimartingale or the L\'{e}vy-It\^{o} process can be found in many literature (see \cite{He92}). Here we do some extension in the one-dimensional case to fit our settings in Section \ref{sec_example1} (similar statements hold for the multi-dimensional case).
 
 \begin{theorem}
\label{itoforito}
Suppose $\mu$ is induced by a one-dimensional c\`{a}dl\`{a}g process $Z$. We say the following ${\cal F}_t$-adapted process $Y$ is an It\^{o} process 
  \begin{equation*} 
\begin{aligned}
Y_t:=Y_0+\int_0^t\alpha_s\mathrm{d}s+\int_0^t\varphi_s\mathrm{d}W_s+\int_0^t\int_{\mathbb{R}_0}\zeta(s,z)\tilde \mu(\mathrm{d}s,\mathrm{d}z),
\end{aligned}
\end{equation*}
where $W$ is an ${\cal F}_t$-Brownian motion, $\alpha$ and $\varphi$ are ${\cal F}_t$-progressively measurable processes such that $\int_0^T(|\alpha_t|+\varphi_t^2)\mathrm{d}t<\infty$, and $\zeta$ is an ${\cal F}_t$-predictable random field such that $\int_0^T\int_{\mathbb{R}_0}\zeta^2(t,z)K(t,\mathrm{d}z)\mathrm{d}t<\infty$. Then for any function $f\in C^2(\mathbb{R})$, $f(Y_t)$ is also an It\^{o} process and we have
   \begin{equation*} 
\begin{aligned}
f(Y_t)=&f(Y_0)+\int_0^tf'(Y_s)\alpha_s\mathrm{d}s+ \int_0^tf'(Y_s)\varphi_s\mathrm{d}W_s+\frac{1}{2}\int_0^tf''(Y_s)\varphi^2_s\mathrm{d}s\\
&+ \int_0^t\int_{\mathbb{R}_0}\left[f(Y_{s-}+\zeta(s,z))-f(Y_{s-})\right]\tilde \mu(\mathrm{d}s,\mathrm{d}z)\\
&+ \int_0^t\int_{\mathbb{R}_0}\left[f(Y_{s-}+\zeta(s,z))-f(Y_{s-})-f'(Y_{s-})\zeta(s,z)\right]K(s,\mathrm{d}z)\mathrm{d}s.
\end{aligned}
\end{equation*}
\end{theorem}
\begin{proof}
We may assume $f$ is a $C^2$ function with compact support without loss of generality. Then we have $\sum_{s\le t}[f(Y_s)-f(Y_{s-})-f'(Y_{s-})\Delta Y_{s}]=\sum_{s\le t}[f(Y_{s-}+\zeta(s,\Delta Z_s))-f(Y_{s-})-f'(Y_{s-})\zeta(s,\Delta Z_s)]1_{\{\Delta Z_s\neq0\}}=\int_0^t\int_{\mathbb{R}_0}[f(Y_{s-}+\zeta(s,z))-f(Y_{s-})-f'(Y_{s-})\zeta(s,z)] \mu(\mathrm{d}s,\mathrm{d}z)$.  Then the formula is an immediate consequence from the It\^{o} formula for general semimartingales (see \cite{He92}) and the Taylor expansion.
  
  \end{proof}

The following theorems are important results in the It\^{o} theory. The proofs can be found in \cite{Protter05,Protter06,He92,Jacod03,Eberlein19}.

\begin{theorem}
\label{novikov}
Suppose $Z$ is an ${\cal F}_t$-semimartingale. Let 
\begin{equation*}
{\cal E}(Z)_t:=\exp\left \{Z_t-Z_0-\frac{1}{2}\langle Z^c\rangle_t\right\}\Pi_{0<s\le t}(1+\Delta Z_s)e^{-\Delta Z_s}
\end{equation*}
 be the Dol\'{e}ans-Dade exponential of $Z$. Then ${\cal E}(Z)$ is the unique semimartingale satisfying the SDE $\mathrm{d}{\cal E}(Z)_t={\cal E}(Z)_{t-}\mathrm{d}Z_t$ with initial value $1$. Moreover, if $Z\in {\cal M}_{2,\text{loc}}$, $\Delta Z_t>-1$, and $\mathbb{E}\left[  e^{\frac{1}{2}\langle Z^c\rangle_t+\langle X^d\rangle_t} \right]<\infty$, then ${\cal E}(Z)$ is an ${\cal F}_t$-martingale.
\end{theorem}
 
 \begin{theorem}
\label{girsanov}
Suppose ${\mathbb{Q}}$ is another probability measure which is equivalent to $\mathbb{P}$. Let $Z_t:=\mathbb{E}[\frac{\mathrm{d}\mathbb{Q}}{\mathrm{d}\mathbb{P}}|{\cal F}_t]$, which can be proved to be an $({\cal F}_t,\mathbb{P})$-martingale. Assume further that $M\in {\cal M}_{\text{loc}}(\mathbb{P})$ and $[Z,M]$ is locally integrable for $\mathbb{P}$. Then $L_t:=M_t-\int_0^t\frac{1}{Z_{s-}}\mathrm{d}\langle Z,M\rangle_s^{\mathbb{P}}\in {\cal M}_{\text{loc}}(\mathbb{Q})$. 
\end{theorem}

   \begin{theorem}
\label{girsanov2}
 Fix $T>0$. Let $Y_t:=\int_0^t\varphi_s\mathrm{d}W_s+\int_0^t\int_{\mathbb{R}_0}\zeta(s,z)\tilde\mu(\mathrm{d}s,\mathrm{d}z)$, $t\in [0,T]$,where $W$ is an ${\cal F}_t$-Brownian motion, $(\varphi,\zeta)\in \mathfrak{L}^*_{\text{loc}}(\langle W\rangle)\times \mathfrak{L}^*_{\text{loc}}(\hat\mu)$, and $\zeta>-1$. If the Dol\'{e}ans-Dade exponential ${\cal E}(Y)_t$, $t\in[0,T]$, is an ${\cal F}_t$-martingale, then ${\mathbb Q}_T:=\int {\cal E}(Y)_T\mathrm{d}\mathbb{P}$ is a probability measure equivalent to $\mathbb{P}$, $W_{\mathbb{Q}_T}(t):=W_t-\int_0^t\varphi_s\mathrm{d}s$, $t\in[0,T]$, is the Brownian motion under $\mathbb{Q}_T$, and $\tilde \mu_{\mathbb{Q}_T}(\mathrm{d}t,\mathrm{d}z):=\tilde \mu(\mathrm{d}t,\mathrm{d}z) -\zeta(t,z)K(t,\mathrm{d}z)\mathrm{d}t$, $t\in [0,T]$, $z\in{\mathbb{R}}_0$, is the compensated random measure of $\mu(\mathrm{d}t,\mathrm{d}z)$ under $\mathbb{P}$. Moreover, $\varphi'\in\mathfrak{L}^*_{\text{loc}}(\langle W\rangle;\mathbb{P})\Leftrightarrow \varphi'\in\mathfrak{L}^*_{\text{loc}}(\langle W_{\mathbb{Q}_T}\rangle;\mathbb{Q}_T)$, and $\zeta'\in{\mathfrak L}^*_{\text{loc}}(\hat\mu;\mathbb{P})\Rightarrow \zeta'\in\mathfrak{L}^*_{\text{loc}}(\mu;\mathbb{Q}_T)$.
\end{theorem}

  \end{appendices}

\end{CJK}
\end{document}